\documentclass[11pt]{article}
\usepackage[letterpaper, margin=1in]{geometry}
\RequirePackage[OT1]{fontenc}
\RequirePackage{amsthm,amsmath}
\RequirePackage[numbers]{natbib}
\RequirePackage[colorlinks,citecolor=blue,urlcolor=blue]{hyperref}
\usepackage{graphicx}
\usepackage{personal}

\usepackage{bbm}
\usepackage{float}
\usepackage{amssymb}
\usepackage{cancel}
\usepackage{ulem}
\usepackage{subcaption,color}
\usepackage[ruled]{algorithm2e}
\usepackage{mdframed}

\newcommand{\rsmp}[2][ \Pi_n]{
    \varrho^\gamma_{#1}(#2)	
}

\newcommand{\Ind}[1]{\mathbbm{1}_{#1} }

\newcommand{\rr}[1]{\textcolor{black}{#1}}

\newcommand{\anv}{\mathtt{ a^*_{NV}}}
\newcommand{\alc}{\mathtt{ a^*_{LC}}}
\newcommand{\ars}{\mathtt{ a^*_{RS}}}
\newcommand{\nX}{\tilde{X}_n}
\newcommand{\nY}{\tilde{Y}_n}
\newcommand{\scKL}{\text{\sc KL}}
\newcommand{\qrsa}{Q^*_{a, \gamma}(\theta|\nX)}
\newcommand{\qrs}{Q^*_{a', \gamma}(\theta|\nX)}
\newcommand{\qnv}{Q^*(\theta|\nX)}
\newcommand\redsout{\bgroup\markoverwith{\textcolor{red}{\rule[0.5ex]{2pt}{0.4pt}}}\ULon}

\newcommand{\argmin}{\text{argmin}}
\newcommand{\argmax}{\text{argmax}}

\numberwithin{equation}{section}
\theoremstyle{plain}
\newtheorem{definition}{Definition}[section]
\newtheorem{lemma}{Lemma}[section]
\newtheorem{theorem}{Theorem}[section]
\newtheorem{conjecture}{Conjecture}
\newtheorem{corollary}{Corollary}[section]
\newtheorem{proposition}{Proposition}[section]

\newtheorem{assumption}{Assumption}[section]

\allowdisplaybreaks
\title{On the Statistical Consistency of Risk-Sensitive Bayesian Decision-Making}
\author{Prateek Jaiswal{$^\star$}, Harsha Honnappa{$^\star$} and Vinayak A. Rao{$^\dag$} }
\date{$^\star$\{jaiswalp,honnappa\}@purdue.edu School of Industrial Engineering, Purdue University\\$^\dag$\{varao@purdue.edu\} Department of Statistics, Purdue University.}

\begin{document}
\maketitle

\begin{abstract}

We study data-driven decision-making problems in the Bayesian framework, where the expectation in the Bayes risk is replaced by a risk-sensitive entropic risk measure. We focus on problems where calculating the posterior distribution is intractable, a typical situation in modern applications with large datasets and complex data generating models. We leverage a dual representation of the entropic risk measure to introduce a novel risk-sensitive variational Bayesian (RSVB) framework for jointly computing a risk-sensitive posterior approximation and the corresponding decision rule. The proposed RSVB framework can be used to extract computational methods for doing risk-sensitive approximate Bayesian inference. We show that our general framework includes two well-known computational methods for doing approximate Bayesian inference viz. \textit{naive} VB and \textit{loss-calibrated} VB. We also study the impact of these computational approximations on the predictive performance of the inferred decision rules and values. We compute the convergence rates of the RSVB approximate posterior and also of the corresponding optimal value and decision rules.  We illustrate our theoretical findings in both parametric and nonparametric settings with the help of three examples:  the single and multi-product newsvendor model and Gaussian process classification.

\end{abstract}



\section{Introduction}

This paper focuses on \textit{risk-sensitive} Bayesian decision-making, considering objective functions of the form
\begin{align}\tag{SO}
	\min_{a \in \sA}\rsmp{R(a,\theta)}
    := \frac{1}{\gamma}\log \bbE_{\Pi_n}[\exp(\gamma R(a,\theta))].
\end{align}
Here $\sA$
is the decision/action space, $\theta$ is a random model parameter lying in an arbitrary measurable space $(\Theta,\mathcal{T})$, and $R(a,\theta):\sA\times \Theta \mapsto  \bbR$  is a problem-specific model risk function.
The distribution $\Pi_n$ is the Bayesian posterior distribution over the parameters given observations $\nX$: \(\Pi_n(\theta) := \Pi(\theta|\nX)\),  while the scalar $\gamma \in \bbR$
 is user-specified and 
 characterizes the sensitivity of the decision-maker (DM) to the distribution $\Pi_n$.  
 Recall that the posterior distribution is obtained by updating a prior probability distribution $\Pi(\theta)$, capturing subjective beliefs of the decision maker over $\theta$,  according to the Bayes rule 
 \begin{align} 
     d\Pi(\theta|\nX)= \frac{ d\Pi(\theta)p^n_{\theta}(\nX)}{\int_{\Theta} d\Pi(\theta)p^n_{\theta}(\nX) },
     \label{eq:Post} 
 \end{align}
 where $p^n_{\theta}(\nX)$ is the  likelihood of observing $\nX$. 
 
 
 The functional $\varrho_\cdot^\gamma$ is also known as the {\it entropic risk measure}, and models a range of risk-averse or risk-seeking behavior in a succinct manner through the parameter $\gamma$. Consider only strictly positive $\gamma$, and observe that
 \begin{align*}
     \lim_{\gamma \downarrow 0} \frac{1}{\gamma}\log \bbE_{\Pi_n}[\exp(\gamma R(a,\theta))] = \mathbb E_{ \Pi_n}(R(a,\theta));
 \end{align*}
 that is, there is no sensitivity to potential risks due to large tail effects and the decision-maker is risk neutral. On the other hand,
 \begin{align*}
     \lim_{\gamma \to +\infty} \rsmp{R(a,\theta)} = \text{ess}\sup_{\Pi_n}(R(a,\theta)),
 \end{align*}
the essential supremum of the model risk $R(a,\theta)$. In other words, a decision maker is completely risk averse and anticipates the worst possible realization ($\Pi_n$-almost surely). While similar conclusions can be drawn when $\gamma < 0$, resulting in risk-seeking behavior, we restrict ourselves to $\gamma>0$ in  this paper. Observe that (SO) strictly generalizes the standard Bayesian decision-theoretic formulation of a decision-making problem, where the goal is to solve $\min_{a \in \mathcal A} \mathbb E_{\Pi_n}[R(a,\theta)]$. Furthermore, it also coincides with other risk-based Bayesian methods, such as the  penalized posterior variance method studied in~\cite{Lai2011} for solving the Markowitz portfolio optimization problem, under certain parameterizations. More precisely, for $R(a,\theta)=\frac{1}{\gamma}\log \bbE_{P_{\theta}}[\exp(\gamma \ell(a,\xi))]$ 
(for any loss function $\ell(\cdot,\cdot)$) and small, but strictly positive $\gamma$, a Taylor expansion of $\rsmp{R(a,\theta)}$ straightforwardly shows that (SO) is equivalent to problem (3.1) in~\cite{Lai2011}. 

The risk-sensitive formulation (SO) is very general and can be used to model a wide variety of decision-making problems in  operations research/management science~\citep{Sc1959b,Chu2008,LuShSh2015}, simulation optimization~\citep{Chick2006,wu2018bayesian}, and finance~\citep{Lai2011,Bauder2020,Bodnar2017}. Moreover, it presents a natural way to address epistemic model uncertainty by being Bayesian and risk sensitive. 
%
%
However, solving (SO) to compute an optimal decision over $\mathcal{A}$ is challenging. The difficulty mainly stems from the fact that, with the exception of conjugate priors, the posterior distribution in~\eqref{eq:Post} is an intractable quantity.
The  use of conjugate priors is restrictive and moreover, for many important likelihood models, they often do not exist. Canonically, posterior  intractability is addressed using either a sampling- or optimization-based approach. Sampling-based approaches, such as Markov chain Monte Carlo (MCMC), offer a tractable way to compute the integrals and theoretical guarantees of exact inference in the large computational budget limit. However, these asymptotic guarantees are offset by issues like poor mixing, large variance and complex diagnostics in practical settings with finite computational budgets.

In response, optimization-based methods such as variational Bayes (VB) or variational inference (VI) have emerged as a popular alternative~\citep{Bl2017}. The VB approximation of the true posterior is a tractable distribution, chosen from a `simpler' family of distributions known as variational family, by minimizing the discrepancy between the true posterior and members of that family.  Kullback-Liebler (\scKL) divergence is the most often used measure of the approximation discrepancy, although other divergences (such as the $\alpha$-R\'enyi divergence~\citep{LiTu2016,TuSh2011a,jaiswal2019asymptotic}) have been used. The minimizing member, termed {the} VB approximate posterior, {can be} used as a proxy for the true posterior. Empirical studies have shown that VB methods are computationally faster and far more scalable to higher-dimensional problems and large datasets.
Theoretical guarantees, such as large sample statistical inference, have been a topic of recent interest in theoretical statistics community, with asymptotic properties such as convergence rate and asymptotic normality of the VB approximate posterior recently established in~\citep{Zhang2018,Pati18} and~\cite{WaBl2017} respectively.

Our ultimate goal is not to {merely} approximate the posterior distribution, but to also make decisions when that posterior is intractable. A naive approach would be to plug in the VB approximation in place of the true posterior in  (SO) and compute the optimal decision. However, it  has been noted in~\citep{LaSiGh2011} that such a naive loss-unaware approach can be suboptimal. In particular,~\cite{LaSiGh2011} demonstrated, through an example, that a naive posterior approximation only captures the most dominant mode of the true posterior which may not be relevant from decision-making perspective. Consequently, they proposed a loss-calibrated variational Bayesian (LCVB) algorithm for solving Bayesian decision making problems where the underlying risk function is discrete.~\cite{kusmierczyk2019variational} extended their approach to continuous risk functions. Despite these algorithmic advances in developing decision-centric variational Bayesian methods, their statistical properties such as asymptotic consistency and convergence rates of the loss-aware posterior  approximation and the associated  decision  rule are  not well understood. In fact, it is not even clear that the convergence rates of VB approximate posterior established in~\cite{Zhang2018,Pati18} can be used to establish statistical guarantees on the decision rules learnt using the n\"aive approach. With an  aim to address these gaps, we summarize our contribution in this  paper below:
\begin{enumerate}
    \item We introduce a minimax optimization framework titled `risk sensitive variational Bayes'  (RSVB), extracted from the dual representation  of (SO) using {the so-called} Donsker-Varadhan variational free-energy principle~\citep{donsker1983asymptotic}. 
    The decision-maker computes a risk-sensitive approximation to the true posterior (termed as RSVB posterior) and the decision rule simultaneously  by solving a minimax optimization problem. Moreover, for $\gamma\to 0^+$ and $\gamma=1$, we recover the naive and LCVB approaches as special cases of RSVB.  
    \item We identify verifiable regularity conditions on the  prior, likelihood model and the risk function 
    {under which} the RSVB posterior enjoys the same rate of {convergence as the true posterior} to a Dirac delta distribution {concentrated} at the true model parameter $\theta_0$, as the sample size increases. Using this result, we also prove the rate of convergence of the RSVB decision rule, when the decision space $\sA$ is compact. Moreover, our theoretical results directly imply  the asymptotic properties of the LCVB posterior and the associated decision rule. It is also worth noting that our results are applicable to non-parametric problems such as Gaussian process classification, where the parameter space is infinite-dimensional, {as well as non independent and identically distributed data generating processes}. Moreover, our analysis also recovers consistency and rate of convergence of decision-rules under the `true' posterior distribution as a special case.
    \item We demonstrate our theoretical results with help of three applications:
    \begin{enumerate}
        \item First, we consider {the} classic single-product newsvendor problem and verify all the regularity conditions required to establish the convergence rate of the RSVB posterior and  the decision rule. We recover the frequentist rate of convergence $\sqrt{n}$ upto logarithmic factor. Moreover, we present simulation results demonstrating the interplay between the risk-sensitive parameter $\gamma$ and number of samples $n$.
        \item Second, we   consider the multi-product newsvendor problem and establish the rate of convergence of the corresponding RSVB posterior and decision rule. Here also, we recover the frequentist rate of convergence $\sqrt{n}$ upto logarithmic factor.
        \item Finally, we consider a binary Gaussian process classification problem, where  the model parameter $\theta$ lie in a set of continuous functions on a compact subset of $\bbR^d$. We construct a wavelet prior and prove all the regularity conditions and compute the rate of convergence  of  the RSVB posterior (on function space) and the decision rule. The rate of convergence of the RSVB posterior matches to that of the true posterior as established in~~\cite[Theorem 4.5]{Zanten08} for the same wavelet prior.
    \end{enumerate}
    \end{enumerate}
    
In our theoretical analysis, we mainly establish three important results. First, in Theorem~\ref{thrm:thm1}, we compute a bound on the expected distance of a model from the true model, where expectation is taken with respect to the RSVB posterior. The bound depends on the risk sensitivity parameter $\gamma$ and the number of samples $n$, and is a sum of two terms: first one quantifies the rate of convergence of the true posterior and  the second one is a consequence of the variational approximation. We further establish regularity conditions on the variational family to compute the rate of convergence of the second term in the bound. In the next two results, we use Theorem~\ref{thrm:thm1} to derive high probability bounds on the optimality gaps in values (Theorem~\ref{thrm:OGValue}) and decisions (Theorem~\ref{thrm:OGdecision}) computed using the RSVB approach. We define optimality gap in decisions as the deviation of the true  optimal decision (when true model is known) from the RSVB decision and define optimality gap in values as the absolute difference between oracle risk $R(\cdot,\theta_0)$ evaluated at true and RSVB decision rules. In our simulation results, we first demonstrate the consistency of the RSVB decision with respect to $n$ for various values of $\gamma$. We then demonstrate the effect of changing $\gamma$ on the  optimality gaps and the variance of the RSVB posterior for a given $n$. In particular, we observe that for smaller $n$, increasing $\gamma$ (after a certain value) result into a significantly more risk-averse decision, however the effect of increasing $\gamma$ on risk-averse decision-making reduces as $n$ increases.

Here is a brief roadmap for the rest of the paper. In the next section
we provide a literature survey of relevant results from machine learning, theoretical statistics and operations research, placing our results in appropriate context. In Section~\ref{sec:b-VB}, we present the problem formulation and introduce RSVB framework with relevant notations, definitions and regularity conditions. We develop our theoretical
results in Section~\ref{sec:FSB}. Thereafter, in Section~\ref{sec:SpecialRSVB}, we discuss naive and loss-calibrated VB as special cases of RSVB. We then illustrate the bounds obtained in Section~\ref{sec:FSB} by specializing the results to the single and multi-product newsvendor problem and Gaussian  process classification problem in Section~\ref{sec:App} and also present some numerical results. We end with
concluding remarks in Section~6.

\section{Existing literature and our work}~\label{sec:lit-survey}
Our paper fits in with a growing body of work in developing rigorous theoretical understanding of variational Bayesian methods in statistics and machine learning. Moreover, our proposed methodology also contributes to the work in machine learning and operations research
that lies
at the intersection of decision-making under uncertainty and
statistical estimation. 

The primary goal in data-driven
decision-making is to learn empirical decision-rules (or
\textit{predictive prescriptions} as~\cite{BeKa2014} term them) $a^*(\tilde X_n)$ that prescribes a decision, given an observation of the covariates $\tilde X_n$. Early work in this direction, including classic work by Herbert Scarf on Bayesian solutions to the newsvendor problem~\citep{Sc1960}, focused on two-stage solutions - estimation followed by optimization. Our setting is most related to recent work on {\it Bayesian risk optimization} (BRO) in~\cite{wu2018bayesian,zhou2017simulation}. In BRO, the authors consider optimal decision-making using various coherent risk measures computed under the posterior distribution. The authors establish several important results, including that the optimal values and decisions are asymptotically consistent as the sample size tends to infinity, and central limit theorems for these quantities. However, there are substantial differences with our paper. First, all of the analysis in~\cite{wu2018bayesian} presumes that the posterior risk measures are actually computable. The authors do not address the critical computational questions surrounding Bayesian methods or the impact of (inevitable) computational approximations on BRO -- indeed, this is not their focus. Second, extended coherent risk measures are not considered (in particular, the log-exponential risk measure used here), and it is unclear if the asymptotic results continue hold otherwise. Third, while we use a risk measure to derive the computational framework (RSVB), the focus in~\cite{wu2018bayesian} is purely on the analytical properties of optimal decisions.    

More recently, there has been significant interest in methods that use empirical risk minimization (ERM) or sample average approximation (SAA) for directly estimating decision-rules that optimize Monte Carlo or empirical approximations~\citep{BeKa2014,BeKa2016,BaRu2014,bertsimas2018optimization,deng2018coalescing,elmachtoub2017smart,wilder2018melding}. The survey by \cite{homem2014monte} consolidates recent results on Monte Carlo methods for stochastic optimization. It is important to note that this recent surge of work in
data-driven decision-making has largely focused on explicit black-box
models. On the other hand, there are many
situations where optimal decisions must be made in the presence of a
well-defined parametrized stochastic model. Bayesian methods are a natural means for estimating distributions over the 
parameters of a stochastic model; though, as noted before, the computational complexity of Bayesian algorithms can be high. The interplay between optimization
and estimation, in the sense of discovering predictive prescriptions
for Bayesian models has largely been ignored. Furthermore, as~\cite{LiSh2005} show in the newsvendor
context, SEO methods can be suboptimal in terms of expected regret and
long-term average losses. \cite{LiSh2005} introduced \textit{operational
	statistics} (OS) as an alternative to SEO (see \cite{Chu2008,LuShSh2015} as well), whereby the optimal empirical order quantity
is determined as a function of an optimization parameter that can be
determined for each sample size. OS has demonstrably better
performance, especially on single parameter newsvendor
problems (though there is much less known about its statistical
properties). 

In the machine learning literature, \cite{LaSiGh2011} observe that
calibrating a Gaussian process classification algorithm to a fixed loss function can improve classification performance over a loss-insensitive
algorithm -- indeed, this is the first documented presentation of the LCVB algorithm. Similarly, surrogate loss
functions~\citep{BaJoMc2006,TaChKoGu2005} that are regularized upper
bounds that depend on the cost function, also implicitly loss-calibrate
frequentist classification algorithms.

While standard VB methods for posterior estimation have been extensively
used in machine learning \citep{Bl2017}, it is only recently that the
theoretical questions surrounding VB have been addressed~\cite{WaBl2017,Zhang2018,Pati18,jaiswal2019asymptotic,Yang2020,Alquier2020,knoblauch2019frequentist,ChriefAbdellatif2018,banerjee2021pac}. In
particular, we note~\cite{WaBl2017} who prove asymptotic consistency
of VB in the large sample limit, \cite{Zhang2018} and~\cite{Pati18} on the other
hand establish bounds on the rate of convergence of the VB posterior
to the `true' posterior providing a more refined analysis, and~\cite{jaiswal2019asymptotic} where asymptotic consistency of $\alpha$-R\'enyi VB was demonstrated. Our analysis in this
paper, extends these results to establish convergence rates of the approximate posterior and learnt decision rules in risk-sensitive variational Bayesian decision-making framework. These bounds, in turn, are complementary to large sample analyses in~\cite{JaHoRa2019b}.

\section{Problem Setup}~\label{sec:b-VB}
Let $(\bigotimes_n\mathcal{X},\mathcal{S}^n,P_{\theta}^n)$ be a measure space with sigma-algebra $\mathcal{S}^n$ generated by $\bigotimes_n\mathcal{X}$, where, in general,  $\bigotimes_n A$ denote the $n$-fold product of a set $A$. Let $\tilde X_n := \{\xi_1,\ldots,\xi_n\} \subset \bigotimes_n\mathcal{X}$ represent a set of $n$ samples from the true model $P_{\theta_0}^n (\equiv P_{0}^n)$ with parameter $\theta_0 \in \Theta$.  Denoting the likelihood of observing $\nX$ as $p_{\theta}^n(\nX)$ and the {\it prior} distribution $\Pi(\theta)$, we define the {\it posterior distribution} as 
     \(d\Pi(\theta | \nX) =\frac{p_{\theta}^n(\nX) d\Pi(\theta)}{\int_{\Theta}p_{\theta}^n(\nX) d\Pi(\theta)}.\)
       We also write $\Pi(\theta|\nX)$ as $\Pi_n$ for brevity. Moreover, we denote the  corresponding prior and posterior density (if they exist) as $\pi(\cdot)$ and $\pi(\cdot|\nX)$.

As noted in the introduction, our objective is to optimize the posterior log-exponential or entropic risk measure of $R(a,\theta)$, that is
\begin{align}
    \tag{SO}
\min_{a \in \sA} \varrho^\gamma_{\Pi_n}(a) = \frac{1}{\gamma} \log \bbE_{\Pi_n}[e^{\gamma R(a,\theta)}], \text{ where $\gamma \in \bbR $.}
\end{align}

 In practical settings, the posterior $\Pi(\theta|\nX)$ typically cannot be easily computed, and decision makers are often led to restrictive modeling choices such as assuming the likelihood function has a conjugate prior. 
 Nonetheless, incorporating non-conjugate priors and complicated hierarchical models is critical for realizing the full utility of decision-theoretic Bayesian methods - however this entails the use of computational approximations. 
 Therefore, in the next paragraph we introduce a {{framework}} from which can be extracted computational methods for approximately computing and optimizing posterior decision risk.

\subsection{Risk-Sensitive Variational Bayes}\label{sec:rsvb}
Our approach exploits the dual representation of the log-exponential risk measure in~(SO), which is convex (or extended coherent)~\citep{rockafellar2007coherent,follmer2011entropic}. From the Donsker-Varadhan variational free energy principle~\citep{donsker1975asymptotic,donsker1975asymptoticb,donsker1976asymptotic,donsker1983asymptotic} we observe that,
\begin{align}\tag{DV}
    \varrho^\gamma_{\Pi_n}(a) &=\begin{cases}
        \min_{Q \in \mathcal M} \left\{ \mathbb
        E_Q[R(a,\theta)] - \frac{1}{\gamma}
        \text{KL}(Q\|\Pi_n)\right\}~&\gamma < 0,\\
        \max_{Q \in \mathcal M} \left\{ \mathbb
        E_Q[R(a,\theta)] - \frac{1}{\gamma}
        \text{KL}(Q||\Pi_n) \right\}~&\gamma > 0,
    \end{cases}
\end{align}
where $\mathcal M$ is the set of all distribution functions that are absolutely continuous with respect to the posterior distribution $\Pi_n$ and `$\textsc{KL}$' represents the Kullback-Leibler divergence. Formally, 
    for any two distributions $P$ and $Q$ defined on measurable space $(\Theta,\sT)$ , the $\scKL$ divergence is defined as
    \begin{align}
        \scKL(Q\|P)= \begin{cases}
            \int_{\Theta} dQ(\theta)\log \frac{dQ(\theta)}{dP(\theta)} &\text{~if } Q\ll P,
            \\
            \infty &\text{ otherwise }, 
        \end{cases}
    \end{align}
    
    where $Q\ll P$ denotes that measure $Q$ is absolutely continuous with respect to $P$. 
 Notice that this dual formulation exposes the reason we choose to use the log-exponential risk -- the right hand side provides a combined assessment of the risk associated with model estimation (computed by the KL divergence $\scKL(Q\|\Pi_n)$) and the decision risk under the estimated posterior $Q$ (computed by $\bbE_Q[R(a,\theta)]$). 

In this paper, we restrict our analyses to the risk-averse case, that  is $\gamma>0$. However, it can be extended easily to the case when $\gamma<0$ to obtain similar theoretical insights. 

As stated above, the reformulation presented in (DV) offers no computational gains. However, restricting ourselves to an appropriately chosen subset $\mathcal Q \subset \mathcal M$, that consists of distributions where the integral $\mathbb E_q[R(a,\theta)]$ can be tractably computed, we immediately obtain a  \textit {risk-sensitive variational Bayesian} (RSVB) formulation of~(DV):
\begin{align}
    \tag{RSVB}
    \frac{1}{\gamma}
    \log \mathbb \bbE_{\Pi_n}\left[e^{\gamma R(a,\theta)}\right] 
    \geq \max_{Q \in \mathcal Q} \left\{ \mathbb
     \bbE_Q[ R(a,\theta)] - \frac{1}{\gamma} 
    \text{KL}(Q||\Pi_n)\right\} = : \mathcal{F}(a;Q(\cdot), \nX, \gamma),
\end{align}
RSVB is our framework for data-driven risk-sensitive decision-making. The family of distributions $\sQ$ is popularly known as the  \textit{variational family}.
The choice of the family $\sQ$, disutility/ risk $R$, and parameter $\gamma$ encodes specific problem settings. Our analysis in subsequent Section~\ref{sbsec:PropEta} below reveals general guidelines on how to choose $\sQ$ that ensures a small optimality gap (defined below) with high probability.


 With an appropriate choice of $\sQ$, the optimization on the RHS can yield a good approximation to the log-exponential risk measurement on the left hand side (LHS).
%
%
For brevity, for a given $a \in \sA$ we define the RSVB approximation to the true posterior $\Pi(\theta|\nX)$ as 
\vspace{0em}
\[ \qrsa := \argmax \{ {Q \in \sQ}: \mathcal{F}(a;Q(\cdot), \nX, \gamma) \} \] and the RSVB optimal decision as
\vspace{0em}
 \[\ars := \argmin_{a \in \mathcal A}~\mathcal{F}(a;\qrsa, \nX, \gamma).\]
Observe that $\qrsa$ and $\ars$ are random quantities, conditional on the data $\tilde X_n$. 
    Intuitively, it can be observed that the risk averseness of $\ars$ increases with  increase in $\gamma$. To observe this consider the RSVB formulation and note that $\scKL>0$,  therefore as $\gamma$ increases there is more incentive to deviate from the true posterior and choose $Q \in \sQ$ that maximizes expected risk for a given $a\in \sA$. Consequently as $\gamma$ increases, the RSVB decision rule becomes more risk-averse.

%

Examples of $\sQ$ include the family of Gaussian 
distributions, delta functions, or the family of 
factorized `mean-field' distributions that discard correlations between components of $\theta$. The choice of $\sQ$ is decisive in determining the performance of the algorithm. In general, however the requirements on $\sQ$ are minimal, and part of the analysis in this paper is to articulate sufficient conditions on $\sQ$ that ensure small optimality gap (defined below) for the optimal decision, $\ars$. This establishes the ``statistical goodness'' of the procedure as number of samples increase. In this paper, we analyze the efficacy of the decision rules obtained using the RSVB approximation, by providing high-probability bounds on the optimality gap. We define the {\it optimality gap} for any $\mathtt a \in \sA$ with value $V=R(\mathtt a,\theta_0)$ as,
\begin{definition}[Optimality Gap]
    Let $V_0^* := \min_{a \in \sA} R(a,\theta_0)$  be the optimal value and $a_0^*$ $ :=
     \argmin_{a\in\sA} R(a,\theta_0)$ be the optimal decision for the true model parameter $\theta_0$. Then, the optimality gap in the value is the difference
    \(
    V- V_0^* ,
    \)
    ~and the optimality gap in decision variables is 
    \(
    \left\| a_0^* - \mathtt a\right\|,
    \)
    ~where $\|\cdot\|$ is the appropriate norm on the decision sapce $\sA$.
\end{definition}

A similar performance measure was used in~\cite{kusmierczyk2019variational}, to measure the effectiveness of loss-calibrated VB (LCVB) approach, which can be obtained by setting $\gamma=1$, as a special case of our RSVB formulation. Nonetheless, in Section~\ref{sec:SpecialRSVB}, we discuss two well-known variational Bayesian algorithms (one of them is LCVB) for decision making, which are special cases of RSVB. Moreover, we establish bounds on their respective optimality gaps as a corollary to the bounds derived for
RSVB.

Note that the RSVB algorithm described above is idealized -- clearly the objective $\mathcal{F}(a;$ $Q(\cdot), \nX,\gamma)$ cannot be computed since it requires the calculation of the posterior distribution -- the very object we are approximating! Note, however that optimizing $\mathcal{F}(a;Q(\cdot), \nX,\gamma)$ is equivalent to optimizing $\{\gamma \bbE_Q[ R(a,\theta)] - \bbE_{Q}[\log(dQ(\theta)/ (d\Pi(\theta)p^n_{\theta}(\nX)))]\}$, where $d\Pi(\theta)p^n_{\theta}(\nX))$ is known, and for which the optimizers are the same. Since our focus is on bounding the optimality gap, in the remainder of the paper any reference to the RSVB algorithm is an allusion to the idealized objective $\mathcal{F}(a;Q(\cdot), \nX,\gamma)$. 

In the following section, we lay down important assumptions and definitions used throughout the paper to establish our theoretical results. 

\subsection{Notations and Definitions}

We provide the definitions of important terms used throughout the paper. First, recall the definition of covering numbers: 
\begin{definition}[Covering numbers]\label{def:cover}
    Let $\mathcal{P}:=\{ P_{\theta},\theta \in \Theta \}$  be a parametric family of distributions and $d:\mathcal{P} \times \mathcal{P} \mapsto [0,\infty)$    be a metric. An $\e-$cover of a subset  $\mathcal{P}_{K} :=\{P_{\theta}:  \theta \in K \subset \Theta \}$ of the parametric family of distributions is a set $K' \subset K $ such that, for each $\theta \in K$ there exists a $\theta' \in  K'$ that satisfies $d(P_{\theta},P_{\theta'})\leq \e$. The $\e-$covering number of $\mathcal{P}_{K}$  is
    \(N(\e, \mathcal{P}_{K}, d ) = \min \{ card( K' ): K'\text{ is an $\e-$cover of } K \},\)
    where $card(\cdot)$ represents the cardinality of the set. 
\end{definition}


Next, recall the definition of a test function \citep{Schwartz1965}:
\begin{definition}[Test function]\label{def:test}
    Let $\nX$ be a sequence of random variables on measurable space $(\bigotimes_n \sX,\mathcal{S}^n)$. Then any $\mathcal S^n$-measurable sequence of functions $\{\phi_n\},~ \phi_n: \nX \mapsto [0,1]~\forall n \in \mathbb N$, is a \textit{test of a hypothesis} that a probability measure on $\mathcal S^n$ belongs  to a given set against the hypothesis that  it belongs to an alternative set. The test $\phi_n$ is \textit{consistent} for hypothesis $P_0^n$ against the alternative $P^n \in \{P_{\theta}^n : \theta \in \Theta\backslash\{\theta_0\} \}$ if $\mathbb{E}_{P^n}[\phi_n] \to \Ind{ \{\theta\in \Theta\backslash\{\theta_0\} \}  }(\theta), \forall \theta \in \Theta$ as $n \to \infty$, where $\Ind{\{\cdot\}}$ is an indicator function. 
\end{definition}
A classic example of a test function is $\phi^{\text{KS}}_n = \mathbbm{1}_{\{\text{KS}_n > K_{\nu}\}}(\theta)$ that is constructed using the Kolmogorov-Smirnov statistic $\text{KS}_n := \sup_t |\mathbb{F}_n(t) - \mathbb{F}_{\theta}(t)|$, where $\mathbb{F}_n(t)$ and $\mathbb{F}_{\theta}(t)$ are the empirical and true distribution respectively,  and $K_{\nu}$ is the confidence level. If the null hypothesis is true, the Glivenko-Cantelli theorem~\citep[Theorem 19.1]{van2000asymptotic} shows that the KS statistic converges to zero as the number of samples increases to infinity.

%

Furthermore, we define 
    the Hellinger distance $h(\theta_1,\theta_2)$ between the two probability distributions $P_{\theta_1}$ and $P_{\theta_2}$ is defined as
    \(d_{H}(\theta_1,\theta_2) = \left(\int \left( \sqrt{dP_{\theta_1}} - \sqrt{dP_{\theta_2}}\right)^2\right)^{1/2}. \)
We define 
    the one-sided Hausdorff distance $H(A\|B)$ between sets $A$ and $B$ in a metric space $D$ with distance function $d$ is defined as: \[{H}(A\|B) = \sup_{x \in A} d_{h}(x,B), \text{ where } d_{h}(x,B)= \inf_{y \in B} d(x,y) .\]
Next, we define an arbitrary loss function $L_n : \Theta \times \Theta \mapsto \bbR$ that measures the distance between models $(P^n_{\theta_1},P^n_{\theta_2}) \forall \{\theta_1,\theta_2\} \in \Theta$. At the outset, we assume that $L_n(\theta_1,\theta_2)$ is always positive. 
 We define $\{\e_n\}$ as a sequence such that $\e_n \to 0$ as $n \to \infty$ and $n \e_n^2 \geq 1$.
    

We also define 
\begin{definition}[$\Gamma-$convergence]
    A sequence of functions $F_n: \mathcal{U} \mapsto \bbR $, for each $n \in  \bbN$, $\Gamma-$converges to 
    $F : \mathcal{U} \mapsto \bbR $, 
    if 
    \begin{itemize}
    \item for every $u \in \sU$ and every
    $\{u_n, n\in \bbN \} $ such that $u_n \to u$, 
    \( F (x) \leq  \liminf_{n\to\infty} F_n(u_n);\)
    
    \item for every $u \in  \sU$ , there exists some
    $\{u_n, n\in \bbN \} $ such that $u_n \to u$, 
    \( F (x) \geq \limsup_{n\to\infty} F_n(u_n).\)
    \end{itemize}
    \end{definition}

In addition, we define 
\begin{definition}[Primal feasibility]   
    For any two functions $f:\sU\mapsto \bbR$ and $b:\sU \mapsto \bbR$, a point $u^* \in \sU$ is primal feasible to the following constraint optimization problem 
        \[ \inf_{u\in\sU} f(u) \text{ subject to } b(u) \leq  c,  \] 
        if  $b(u^*) \leq c$, for a given $c\in \bbR$.  
\end{definition}

\subsection{Assumptions}

%


In order to bound the optimality gap, we require some control over how quickly the posterior distribution concentrates at the true parameter $\theta_0$. Our next assumption in terms of a verifiable test condition on the model (sub-)space is one of the conditions required to quantify this rate.

\begin{assumption}[Model indentifiability]\label{assump:Asf1}
    Fix $n\geq 1$. Then, for any $\epsilon > \epsilon_n$ such that $\e_n \to 0$ as $n \to \infty$ and $n \e_n^2 \geq 1$,
    there exists a measurable sequence of test functions $\phi_{n,\e} : \nX \mapsto [0,1]$ and \textit{sieve set} $\Theta_n(\e) \subseteq \Theta$ such that 
    \vspace{0em}
        (i)~ $\bbE_{P^{n}_0}[ \phi_{n,\e} ]  \leq C_0\exp(-C n \e^2 ), \text{ and }$
        (ii)~ $\underset{\{ \theta\in \Theta_n(\e)  : L_n(\theta,\theta_0) \geq C_1 n \epsilon^2  \}}{\sup} \bbE_{P^{n}_{\theta}}[ 1- \phi_{n,\e} ] \leq  \exp(-C n \e^2 )$.
\end{assumption}

Observe that Assumption~\ref{assump:Asf1}$(i)$ quantifies the rate at which a type 1 error diminishes with the sample size, while the condition in~Assumption~\ref{assump:Asf1}$(ii)$ quantifies that of a type 2 error. Notice that both of these are stated through test functions; indeed, what is required are consistent test functions. Opportunely, \citep[Theorem 7.1]{GGV} (stated in Appendix as Lemma~\ref{lem:ggv71} for completeness) roughly implies that an appropriately bounded model subspace $\{ P_{\theta}^n,\theta \in \Theta \}$ (the size of which is measured using covering numbers) guarantees the existence of \textit{consistent} test functions, to test the null hypothesis that the true parameter is $\theta_0$ against an alternate hypothesis -- the alternate being defined using the `distance function' $L_n(\theta_1,\theta_2)$. Subsequently, we will use a specific distance function to obtain finite sample bounds for the optimality gap in decisions and values. In some problem instances, it is also possible to construct consistent test functions directly without recourse to Lemma~\ref{lem:ggv71}. We demonstrate  this in Section~\ref{sec:news} below.

Next, we assume a condition on the prior distribution that ensures that it provides sufficient mass to the set $\Theta_n(\e) \subseteq \Theta$, as defined above in Assumption~\ref{assump:Asf1}. 
\begin{assumption} \label{assump:Asf2}
    Fix $n\geq 1$. Then, for any $\epsilon > \epsilon_n $  such that $\e_n \to 0$ as $n \to \infty$ and $n \e_n^2 \geq 1$,
    the prior distribution satisfies 
    \begin{align*}
        \Pi( \Theta_n^c(\e) ) \leq \exp(-C n \e^2).
    \end{align*} 
\end{assumption}

Notice that Assumption~\ref{assump:Asf2} is trivially satisfied if  $\Theta_n(\e) = \Theta$. The next assumption ensures that the prior distribution places sufficient mass around a neighborhood -- defined using R\'enyi divergence -- of the true parameter $\theta_0$.

\begin{assumption}[Prior thickness]\label{assump:Asf3}
    Fix $n\geq 1$ and a constant $\lambda > 0$. Let   
    $
    A_n := \big\{ \theta \in \Theta :D_{1+\lambda} $ $\left( P_0^n \| P_{\theta}^n \right) \leq C_3 n \e_n^2 \big\}, 
    $
    where $D_{1+\lambda} \left( P_0^n \| P_{\theta}^n \right) := \frac{1}{\lambda} \log \int \left(\frac{dP_0^n}{dP_{\theta}^n}\right)^\lambda dP_0^n $ is the R\'enyi divergence between $P_0^n$ and  $P_{\theta}^n$, assuming $P_0^n$ is absolutely continuous with respect to $P_{\theta}^n$. The prior distribution satisfies 
    \vspace{0em}
    \begin{align*}
        \Pi (A_n)  \geq \exp(-n C_2 \e_n^2). 
    \end{align*} 
\end{assumption}
Notice that the set $A_n$ defines a neighborhood of the distribution corresponding to $\theta_0$ in the model subspace $\{P^n_{\theta}:\theta \in \Theta\}$. The assumption guarantees that the prior distribution covers this neighborhood with positive mass. This is a standard assumption and if it is violated then the posterior too will place no mass in this neighborhood ensuring asymptotic inconsistency.
The above three assumptions are adopted from~\cite{GGV} and has also been used in~\cite{Zhang2018} to prove convergence rates of variational posteriors. Interested readers may refer to~\cite{GGV} and~\cite{Zhang2018} to read more about the above assumptions.

It is apparent by the first term in (RSVB) that in addition to Assumption~\ref{assump:Asf1},~\ref{assump:Asf2}, and~\ref{assump:Asf3}, we also require regularity conditions on the risk function $R(a,\cdot)$. 
Thus, the next assumption restricts the prior distribution with respect to $R(a,\theta)$.
\begin{assumption}\label{assump:Asf14}
    Fix $n \geq 1$ and \textcolor{black}{$\gamma>0$}. For any $\epsilon > \epsilon_n $, $a\in  \sA$, 
    \begin{align*}	
       \bbE_{\Pi}[ \mathbbm{1}_{\left\{{\gamma R(a,\theta )} > {C_4(\textcolor{black}{\gamma})n \e^2} \right\} } e^{\gamma R(a,\theta )} ] \leq \exp(-C_5(\textcolor{black}{\gamma}) n\e^2),
    \end{align*}
    where $C_4(\textcolor{black}{\gamma})$ and  $C_5(\textcolor{black}{\gamma})$ are scalar positive functions of $\gamma$.
\end{assumption}
Note that the set $\left\{{\gamma R(a,\theta )} > {C_4(\textcolor{black}{\gamma}) n \e^2} \right\}$ represents the subset of the model space where the risk $R(a,\theta)$ (for a fixed decision $a$) is large, and the prior is assumed to place sufficiently small mass over such sets. Moreover, using  Cauchy-Schwarz inequality observe that
\begin{align*}	
    \bbE_{\Pi}[ \mathbbm{1}_{\left\{{\gamma R(a,\theta )} > {C_4(\textcolor{black}{\gamma})n \e^2} \right\} } e^{\gamma R(a,\theta )} ]
    &\leq  \left( \bbE_{\Pi}[ \mathbbm{1}_{\left\{{\gamma R(a,\theta )} > {C_4(\textcolor{black}{\gamma})n \e^2} \right\} }  ] \right)^{1/2} \left( \bbE_{\Pi}[  e^{2\gamma R(a,\theta )}] \right)^{1/2} 
    \\
    &\leq e^{-{C_4(\textcolor{black}{\gamma}) n \e^2}} \bbE_{\Pi}[  e^{2\gamma R(a,\theta )}],
    \end{align*}
which implies that if  the risk function is bounded in $(a,\theta)$, then above condition can be trivially satisfied.
Finally, we also require the following condition lower bounding the risk function $R$.
\begin{assumption}\label{assump:Asf15}
    ${ R(a,\theta)}$ is assumed to  satisfy
    \begin{align*}
        W := \inf_{\theta \in\Theta } \inf_{a \in \mathcal A}  e^{  R(a,\theta)}   > 0.
    \end{align*}
\end{assumption}
Note that any risk function which is bounded from below in both the arguments satisfies this condition.  Furthermore, following \cite{Pflugg2003} we define a {\it growth  condition} on the `true' risk function $R(a,\theta_0)$.
\begin{assumption}[Growth condition]~\label{assump:Growth}
    Let $\Psi(d) : [0,\infty) \mapsto [0,\infty)  $ be a {\it growth function} if it is strictly increasing as $d \to \infty$ and $\lim_{d \to 0} \Psi(d) = 0$. Then for any $A\subset \sA$, $R(a,\theta_0)$ satisfies a growth condition with  respect to $\Psi(\cdot)$, if 
    \begin{align}
        R(a,\theta_0) ) \geq \inf_{z \in \mathcal{A}} R(z,\theta_0)  + \Psi \left(H\left( A, \underset{z \in \mathcal{A} }{\arg \min}~ R(z,\theta_0)  \right)\right).
    \end{align}
\end{assumption}
The growth condition above is a generalization of strong-convexity. Indeed, if the true risk is strongly convex, then this condition is automatically satisfied. 

In the next, section we derive high-probability bounds on the optimality gap in values and decisions, by proving a series of results.

\section{Asymptotic Analysis of the Optimality Gaps}~\label{sec:FSB}
In this section, we establish high-probability bounds on the optimality gap in  values and decision rules computed using RSVB approach for sufficiently large $n$. 
Our results in here identify the regularity conditions on the data generating model $\{ P_{\theta}^n,\theta \in \Theta \}$, the prior distribution $\Pi(\theta)$, the variational family $\sQ$, the risk function $R(a,\theta)$ 
 to compute the bounds. 

We can now state our first result, establishing an upper bound on the expected deviation from the true model $P_{0}^n$, measured using distance function $L_n(\cdot,\theta_0)$, under the RSVB approximate posterior. We also note that the following result generalizes Theorem 2.1 of \cite{Zhang2018}, which is exclusively for the case when $\gamma \to 0^{+}$. However, the proof  techniques are motivated from the proof of Theorem 2.1 in~\cite{Zhang2018}.
\begin{theorem} \label{thrm:thm1}
	Fix $a' \in \sA$ and $\gamma > 0$. For any $L_n(\theta,\theta_0)\geq0$, under
     Assumptions~\ref{assump:Asf1},~\ref{assump:Asf2},~\ref{assump:Asf3}, ~\ref{assump:Asf14}, and~\ref{assump:Asf15}, and for $\min(C,C_4(\textcolor{black}{\gamma}) + C_5(\textcolor{black}{\gamma})) > C_2 + C_3 + C_4(\gamma) + 2$ and $$ \eta_n^R(\gamma) := \frac{1}{n}\inf_{Q \in \sQ} \bbE_{P^{n}_0} \left[  \scKL(Q(\theta)\|\Pi(\theta|\nX))  
	- \gamma \inf_{a \in  \mathcal{A}} \bbE_Q[R(a,\theta)] \right], $$
	the RSVB approximator of the true posterior $\qrs$ satisfies,
	\begin{align}
		\bbE_{P^{n}_0} \left[ \int_{\Theta} L_n(\theta,\theta_0) d\qrs \right] \leq  n\left( \textcolor{black}{M(\gamma)} \e_n^2 + M \eta_n^R(\gamma) \right), \label{eq:eq_lcb}
	\end{align}
	for a positive mapping \rr{$M(\gamma)= 2\left( C_{1}+ M C_4(\gamma)  \right)$}~, where $M= \frac{2C_1}{\min(C,\lambda,1)}$
    , for sufficiently large $n$.
	
\end{theorem}

First recall that $\e_n$ is the convergence rate of the true posterior~\cite[Theorem 7.3]{GGV}. Notice that the additional term  $\eta_n^R(\gamma)$ emerges from the posterior approximation and depends on the choice of the variational family $\sQ$, risk function $R(\cdot,\cdot)$, and the parameter $\gamma$. The appearance of this term in  the bound also signifies that, to minimize expected gap between true model and any other model, defined using $n^{-1}L_n(\theta,\theta_0)$, under the RSVB posterior, the average (with  respect to $P_0^n$) RSVB objective has to be maximized.   
 Later in this section, we specify the conditions on the family of distributions $\{ P_{\theta}^n,\theta \in \Theta \}$, the prior and the variational family $\sQ$ that ensure $\eta_n^R(\gamma)\to 0$ as $n \to \infty$. Moreover, we also identify  mild regularity conditions on $\sQ$ to show that $\eta_n^R(\gamma)$ is $O(\e_n^2)$. Furthermore, we show that as $\gamma$ increases $\eta_n^R(\gamma)$ decreases. 
We discuss this result and the bound therein later in the next subsection. Before that, we establish our main result (the bounds on the optimality gap) using the theorem above.

Since the result in Theorem~\ref{thrm:thm1} holds for any positive distance function, we now fix
\begin{align}
     L_n(\theta,\theta_0) = n \left(\sup_{a \in \mathcal{A}} \vert R(a,\theta) - R(a,\theta_0) \vert \right)^2. \label{eq:lossr}
\end{align}
Notice that for a given $\theta$, $n^{-1/2} \sqrt{L_n(\theta,\theta_0)}$ is the uniform distance between the $R(a,\theta)$ and $R(a,\theta_0)$. Intuitively, Theorem~\ref{thrm:thm1} implies that the expected uniform difference $\frac{1}{n}L_n(\theta,\theta_0)$ with respect to the RSVB approximate posterior is $O( \textcolor{black}{M(\gamma)} \e_n^2 + M \eta_n^R(\gamma))$, and if $\textcolor{black}{M(\gamma)} \e_n^2 + M \eta_n^R(\gamma) \to 0$ as $n \to \infty$ then it converges to zero at that rate. 

Also, note that in order to use~\eqref{eq:lossr} we must demonstrate that it satisfies Assumption~\ref{assump:Asf1}. This can be achieved by constructing bespoke test functions for a given $R(a,\theta)$. We demonstrate this approach by an example in~Section~\ref{sec:Mnews}. Nonetheless, we also provide sufficient conditions for the existence of the test functions in the appendix. These conditions are typically easy to verify when the loss function $R(\cdot,\cdot)$ are bounded, for instance.

Now, we first bound the optimality gap between  $R(\ars,\theta_0)$ and $V_0^*$.
\begin{theorem}\label{thrm:OGValue}
    Fix $\gamma>0$. Suppose that the set $\sA$ is compact. 
    Then, under Assumptions~\ref{assump:Asf1}, ~\ref{assump:Asf2},~\ref{assump:Asf3},~\ref{assump:Asf14}, and~\ref{assump:Asf15}, for $\min(C,C_4(\gamma) + C_5(\gamma)) > C_2 + C_3 + C_4(\gamma) + 2$ and for any $\tau>0$, the $P_0^n-$ probability of the following event 
    \begin{align}
        \bigg\{ \nX:   R( \ars,\theta_0) - \inf_{z \in \mathcal{A}} R(z,\theta_0)  \leq 2\tau  \left[ \textcolor{black}{M(\gamma)} \e_n^2 + M \eta_n^R(\gamma) \right]^{\frac{1}{2}}  \bigg\}
    \end{align}
    is at least $1-\tau^{-1}$, for a positive mapping \textcolor{black}{$M(\gamma)= 2\left( C_{1}+ M C_4(\gamma)  \right)$}~, where $M= \frac{2C_1}{\min(C,\lambda,1)}$
    for sufficiently large $n$.
\end{theorem}

Next, we bound the  optimality gap between the approximate optimal decision rule $\ars$ and the true optimal decision. The bound, in particular, depends on the curvature of $R(a,\theta_0)$ around the true optimal decision, defined using the growth condition  in Assumption~\ref{assump:Growth}. 

\begin{theorem}\label{thrm:OGdecision}
    \noindent Fix $\gamma>0$. Suppose that the set $\sA$ is compact and $R(a,\theta_0)$ satisfies the growth condition in~Assumption~\ref{assump:Growth},  with $\Psi(d)$  such that $  \Psi(d)/d^{\delta} = \kappa$, for any  $\delta >0$. Then, under Assumptions~\ref{assump:Asf1}, ~\ref{assump:Asf2},~\ref{assump:Asf3},~\ref{assump:Asf14}, and~\ref{assump:Asf15}, for $\min(C,C_4(\gamma) + C_5(\gamma)) > C_2 + C_3 + C_4(\gamma) + 2$ and for any $\tau>0$, the $P_0^n-$ probability of the following event 
    \[ \left\{ \tilde X_n :  H \left(\ars(\tilde X_n), \underset{z \in \mathcal{A} }{\arg \min}   ~ R(z,\theta_0) \right)    \leq  \left[\frac{2 \tau \left[ \textcolor{black}{M(\gamma)} \e_n^2 + M \eta_n^R(\gamma) \right]^{\frac{1}{2}}}{\kappa} \right]^\frac{1}{\delta} \right\}  \]
    is at least $1-\tau^{-1}$, for a positive mapping \textcolor{black}{$M(\gamma)= 2\left( C_{1}+ M C_4(\gamma)  \right)$}~, where $M= \frac{2C_1}{\min(C,\lambda,1)}$
    for sufficiently large $n$.
\end{theorem}

To fix the intuition, suppose $\delta=2$ and $\Psi(d)= \frac{\kappa}{2}d^2$, then $\kappa$ represents the Hessian of the true risk, $R(a,\theta_0)$, near its optimizer. It is easy to see from the above result the rate of convergence of $\ars$ is scaled by a factor $\kappa^{-1}$. That is, higher the curvature near the optimizer, the faster $\ars$ converges. 

 Evidently, the bounds obtained in all three results that we have proved so far depends on $\eta_n^R(\gamma)$. Consequently, in the next section, with an aim to understand the properties of the bounds in Theorem~\ref{thrm:thm1},~\ref{thrm:OGValue}, and~\ref{thrm:OGdecision}, we prove some of the important properties of $\eta_n^R(\gamma)$ with respect to $n$ and $\gamma$ under some additional regularity conditions.

\subsection{Properties of $\eta_n^R(\gamma)$}~\label{sbsec:PropEta}
In order to characterize $\eta_n^R(\gamma)$,  we specify conditions on variational family $\sQ$ such that  
$\eta_n^R(\gamma)=O(\e_n'^2)$, for some $\e_n' \geq \frac{1}{\sqrt n}$ and $\e_n'\to 0$. We impose following condition on the variational family $\sQ$ that lets us obtain a bound on $\eta_n^R(\gamma)$ in terms of $n$ and $\gamma$.
\begin{assumption}\label{assump:Asf11}
    There exists a sequence of distribution $\{q_n(\cdot)\}$ in the variational family $\sQ$ such that for a positive constant $C_9$,
    \begin{align}\frac{1}{n}  \left[ \scKL\left(Q_n(\theta)\|\Pi(\theta) \right) + \bbE_{Q_n(\theta)} \left[ \scKL\left(dP^n_{0}(\nX)\| dP^n_{\theta}(\nX) \right) \right]  \right] \leq C_9 \e_n'^2. 
        \label{eq:eq1c}
    \end{align}
\end{assumption} 
If the observations in $\nX$ are i.i.d, then observe that \[ \frac{1}{n}    \bbE_{Q_n(\theta)} \left[ \scKL\left(dP^n_{0}(\nX))\| dP^n_{\theta}(\nX) \right) \right]  = \bbE_{Q_n(\theta)} \left[ \scKL\left(dP_0)\| dP_{\theta}(\xi) \right) \right].  \]
 Intuitively, this assumption implies that the variational family must contain a sequence of distributions that converges weakly to a Dirac delta distribution concentrated at the true parameter $\theta_0$ otherwise the second term in the LHS of~\eqref{eq:eq1c} will be non-zero. Also note that the above assumption does not imply that the minimizing sequence $\qrs$ (automatically) converges weakly to a dirac-delta distribution at the true parameter $\theta_0$. Furthermore, unlike Theorem 2.3 of \cite{Zhang2018}, our condition on $\sQ$ in Assumption~\ref{assump:Asf11}, to obtain a bound on $\eta_n^R(\ \gamma)$, does not require the support of the distributions in $\sQ$ to shrink to the true parameter $\theta_0$ at some appropriate rate, as the numbers of samples increases.  
%
\begin{proposition}\label{prop:eta_n}
    Under Assumption~\ref{assump:Asf11} and for a constant $C_8= -  \inf_{Q \in \sQ}$ $ \inf_{a \in \sA}   \bbE_Q[R(a,\theta)]  $ and $C_9>0$,
    \vspace{0em}
    \[\eta_n^R(\gamma) \leq \gamma  n^{-1}C_8 +C_9\e_n'^2. \]
\end{proposition}
In Section~\ref{sec:App}, we present an example where the likelihood is exponentially distributed, the prior is inverse-gamma (non-conjugate), and the variational family is the class of gamma distributions, where we construct a sequence of distributions in the variational family that satisfies Assumption~\ref{assump:Asf11}. We also provide another example where the  likelihood  is  multivariate  Gaussian with  unknown mean and variational family is uncorrelated Gaussian restricted to compact subset of $\bbR^d$ with an uniform prior on the same compact set satisfy Assumption~\ref{assump:Asf11}. 

By definition $\e_n^2 \to 0$ and $\e'_n\to 0$ as $n \to \infty$, and therefore it follows from Proposition~\ref{prop:eta_n} that $\textcolor{black}{M(\gamma)} \e_n^2 + M \eta_n^R(\gamma) \to 0$.
However, the bound obtained in the last proposition might be loose with respect to $\gamma$, when $C_8<0$. To see this, we prove the following result.
\begin{proposition}~\label{prop:eta_gamma}
    If the solution to the optimization problem in $\eta_n^R(\gamma)$ is primal feasible then  $\eta_n^R(\gamma)$ decreases as $\gamma$ increases.

\end{proposition}

\section{Special Cases of RSVB}~\label{sec:SpecialRSVB}
Recall from the RSVB formulation that $\gamma$ encodes the risk sensitivity of the decision maker. In this section, we show that RSVB generalizes two well-known variational Bayesian approaches for decision making, \textit{`naive' VB} (NVB) and \textit{loss-calibrated VB}(LCVB). In particular, the RSVB method is equivalent to NVB when $\gamma\to 0^{+}$ and LCVB for  $\gamma=1$. In what follows, we discuss NVB  and  LCVB briefly and demonstrate our theoretical results to these settings.  


\subsection{Naive VB}

The naive VB (NVB) method, summarized below in Algorithm~\ref{algo:NVB}, is a \textit{``separated estimation and
    optimization''}  method wherein we use the VB approximation to the posterior distribution as a plug-in estimator for computing the posterior predictive loss, and then optimize the resulting approximate posterior predictive loss. 
\begin{algorithm}[ht]
    \SetKwInOut{Input}{Input}\SetKwInOut{Output}{Output}
    \Input{$R(\cdot,\cdot),~\nX,\mathcal Q$ }
    \Output{$\anv$}
    \qquad Step 1. Compute approximate posterior: $\qnv := \arg\min_{Q \in  \sQ}\text{KL}(Q(\cdot) \| \Pi(\cdot|\nX))$\;
    \qquad Step 2. Compute: $\anv := \arg\min_{a \in \sA} \bbE_{\qnv}[R(a,\theta)]$.
    \caption{Naive VB}~\label{algo:NVB}
\end{algorithm}

The NVB method completely isolates the statistical estimation problem from the decision-making problem. Observe that as $\gamma\to 0^+$, $\qrs$  and $\ars $ converges to $\qnv$ and $\anv$ respectively; that is 
\vspace{0em}
\begin{align*}
	\lim_{\gamma\to 0^+} \qrsa &=  \lim_{\gamma\to 0^+} \argmax_{Q \in \sQ} \left\{ \mathbb
            \bbE_Q[ R(a,\theta)] -  \frac{1}{\gamma}
            \text{KL}(Q||\Pi_n)\right\} \\
            & =  \argmin_{\tilde Q \in \sQ} \scKL(\tilde Q(\theta) \| \Pi(\theta|\nX)):=\qnv. 
\end{align*}
To see this, recall the RSVB formulation and multiply by $\gamma>0$ on either side to obtain:
\begin{align}
    \label{eq:eqRSVB}
    \log \mathbb \bbE_{\Pi_n}\left[\exp(\gamma
    R(a,\theta))\right] 
    &\geq \max_{Q \in \sQ} \left\{ \gamma \mathbb
    \bbE_Q[ R(a,\theta)] -  
    \text{KL}(Q||\Pi_n)\right\} 
    \\
    \nonumber
    &= - \min_{Q \in \sQ} \left\{ 
    \text{KL}(Q||\Pi_n)- \gamma \mathbb
    \bbE_Q[ R(a,\theta)]\right\} .
\end{align} 

Note that, since $\text{KL}(Q||\Pi_n)- \gamma \mathbb
\bbE_{Q}[ R(a,\theta)]$ converges uniformly in $\gamma$ to $\text{KL}(Q||\Pi_n)$ as $\gamma\to 0^{+}$, therefore former $\Gamma-$converges to the latter and hence their respective minimizers and minimum values~\citep{Braides2002}. 
In particular, to prove the uniform convergence, let $\{r_k\}$ be a  sequence of rational numbers on $\bbR^{+}$, such that  $\{r_k\},k\in \bbN$ is dense in $\bbR^{+}$ and $r_k\to 0^{+}$ as $k\to \infty$. Now observe that for every $\e>0$ and given $a\in \sA$ and $\theta\in \Theta$, there exists a $K\in \bbN$, such that  for all $k\geq K$ and $Q \in  \sQ$, $| r_k \bbE_{Q}[R(a,\theta)]|<\e$, hence  uniform convergence follows.  

Now taking limit $\gamma\to 0^+$, the equation~\eqref{eq:eqRSVB} reduces to the well known \textit{evidence lower bound}~\citep{Bl2017} , that is
\begin{align*}
    0
    \geq \max_{Q \in \sQ} \left\{ -  
    \text{KL}(Q||\Pi_n)\right\} & \equiv \log \int_{\Theta}dP^n_{\theta}(\nX)\Pi(\theta)d\theta 
    \\
    &\geq \max_{Q \in  \sQ} \left\{ - \scKL(Q(\theta)\|\Pi(\theta))  + \bbE_{Q}[dP^n_{\theta}(\nX)] \right\},
\end{align*} 
where $\Pi(\theta)$ is the  prior density. Therefore, it follows that for any $\gamma>0$
\begin{align*}
    \argmin_{a\in A} \left\{ \mathbb
     \bbE_{\qnv}[ R(a,\theta)] -  \frac{1}{\gamma}
    \scKL(\qnv \| \Pi(\theta|\nX))\right\} =\anv.
    \end{align*}
Since $\lim_{\gamma\to0^+}\gamma R(\cdot,\cdot)=0$, we do not require  Assumption~\ref{assump:Asf14} and~\ref{assump:Asf15} to obtain an analogous result to Theorem~\ref{thrm:thm1a} for NVB method. Therefore, the condition on the constants in Theorem~\ref{thrm:thm1a} ( $\min(C,C_4(\gamma) + C_5(\gamma)) > C_2 + C_3 + C_4(\gamma) + 2$ ) is simplified to $C > C_2 + C_3 + 2$ by choosing $C_4(\gamma)$ as a small and $C_5(\gamma)$ as a large number. 

\begin{theorem}\label{thrm:NV}
	Let $\e_n$ be a sequence such that $\e_n \to 0$ and $n \e_n^2 \to \infty$ as $n \to \infty$ and $$L_n(\theta,\theta_0) = n \left(\sup_{a \in \mathcal{A}} \vert  R(a,\theta) -  R(a,\theta_0)  \vert \right)^2. $$ 
     Then under Assumptions~\ref{assump:Asf1},~\ref{assump:Asf2},~\ref{assump:Asf3},~\ref{assump:Asf14}, and~\ref{assump:Asf15}, and for $C > C_2 + C_3 + 2$
    the NVB approximation of the true posterior satisfies,
	\vspace{0em}
	\begin{align}
	\bbE_{P^{n}_0} \left[ \int_{\Theta} L_n(\theta,\theta_0) \qnv d\theta \right] \leq \bar M n(\e_n^2 + \eta_n(0)),
	\label{eq:prop1}
	\end{align}
	where positive constant $\bar M$ depends only on $C,C_0 ,C_1,$ and $\lambda$, and
	$$ \eta_n(0) := \eta^{R}_n(0)= \frac{1}{n} \inf_{Q \in  \sQ} \bbE_{P^{n}_0} [\scKL(Q(\theta)\|\Pi(\theta|\nX)) ].  	$$	
\end{theorem}

The next result establishes a bound on the optimality gap of the naive VB estimated optimal value $R( \anv,\theta_0)$ from the true optimal value $V_0=\inf_{z \in \mathcal{A}} R(z,\theta_0) $.
\begin{theorem}~\label{thrm:NVOGvalue}
    Suppose that the set $\sA$ is compact and Assumptions~\ref{assump:Asf1},~\ref{assump:Asf2},~\ref{assump:Asf3},~\ref{assump:Asf14}, and~\ref{assump:Asf15} are satisfied with $C > C_2 + C_3 + 2$. 
    Then for any $\tau>0$, the $P_0^n-$ probability of the following event 
    \begin{align}
        \bigg\{ \nX:   R( \anv,\theta_0) - \inf_{z \in \mathcal{A}} R(z,\theta_0)  \leq 2\tau  \left[ \bar M (\e_n^2 + \eta_n(0)) \right]^{\frac{1}{2}}  \bigg\}
    \end{align}
    is at least $1-\tau^{-1}$, where $\bar M$ is a positive constant.
\end{theorem}
Next, we bound the  optimality gap between the approximate optimal decision rule $\anv$ and the true optimal decision. The bound, in particular, depends on the curvature of $R(a,\theta_0)$ around the true optimal decision. 
  The growth function is denoted as $\Psi(\cdot) $. The following theorem is a special case of the general result for $\ars$ in Theorem~\ref{thrm:OGdecision}.
\begin{theorem}\label{thrm:NVOGdecision}
	\noindent Suppose that the set $\sA$ is compact and $R(a,\theta_0)$ satisfies the growth condition,  with $\Psi^1(d)$  such that $  \Psi(d)/d^{\delta} = \kappa$, for a $\delta >0$. Moreover, Assumptions~\ref{assump:Asf1},~\ref{assump:Asf2},~\ref{assump:Asf3},~\ref{assump:Asf14}, and~\ref{assump:Asf15} are satisfied with $C > C_2 + C_3 + 2$. Then for any $\tau>0$ , the $P_0^n-$ probability of the following event 
	\[ \left\{ \tilde X_n :  H \left(\anv(\tilde X_n), \underset{z \in \mathcal{A} }{\arg \min}   ~ R(z,\theta_0) \right)    \leq  \left[\frac{2 \tau \left[ M (\e_n^2 + \eta_n(0)) \right]^{\frac{1}{2}}}{\kappa} \right]^\frac{1}{\delta} \right\}  \]
	is at least $1-\tau^{-1}$, where $M$ is the positive constant as defined in Theorem~\ref{thrm:NV}.

\end{theorem}

\subsection{Loss Calibrated VB}

Algorithm~2 summarizes the {\it Loss-calibrated VB} (LCVB) method \citep{LaSiGh2011}. 
\begin{algorithm}[ht]\label{algo:2}
    \SetKwInOut{Input}{Input}\SetKwInOut{Output}{Output}
    \Input{$R(\cdot,\cdot),~\bar X_n,\mathcal Q$ }
    \Output{$\alc$}
    \qquad Step 1. Compute: $\alc := \arg\min_{a \in \sA} \max_{Q \in  \sQ} \left\{ -\text{KL}(Q(\cdot) \| \Pi(\cdot|\bar X_n)) + \bbE_{Q}[R(a,\theta)]\right\}.$
    \caption{Loss-calibrated VB}
\end{algorithm}
Observe that this method combines the posterior approximation and decision-making problems into one minimax optimization problem. The objective here can be directly contrasted with that in Algorithm~1. Note that the inner maximization will result in an approximate (loss calibrated) posterior distribution at each decision point $a \in \sA$.  Moreover, also note that LCVB is same as RSVB for $\gamma=1$.


In this section, we compute a bound on the optimality gaps loss-calibrated optimal decision $\alc$ and optimal value. 

\begin{theorem}~\label{thrm:LC}
	Fix $a_0\in \sA$ and let $\e_n$ be a sequence such that $\e_n \to 0$ and $n \e_n^2 \to \infty$ as $n \to \infty$ and 
    $$L_n(\theta,\theta_0) = n \left(\sup_{a \in \mathcal{A}} \vert R(a,\theta) - R(a,\theta_0) \vert \right)^2. $$
     Then under Assumptions~\ref{assump:Asf1},~\ref{assump:Asf2},~\ref{assump:Asf3},~\ref{assump:Asf14}, and~\ref{assump:Asf15}, 
      for some positive constants $C,C_2,C_3,C_4(1), $ and $C_5(1)$ such that $\min(C,(C_4(1)+C_5(1))) > C_2 + C_3 + C_4(1)+ 2$, and for
	$$ \eta_n^{R}(1) := \frac{1}{n} \inf_{Q \in \sQ} \bbE_{P^{n}_0} \left[  \scKL(Q(\theta)\|\Pi(\theta|\nX))  
	- \inf_{a \in \sA} \bbE_Q[R(a,\theta)] \right], $$
	 the Loss calibrated VB approximation of the true posterior satisfies,
	\begin{align}
	\bbE_{P^{n}_0} \left[ \int_{\Theta} L_n(\theta,\theta_0) q_{a_0}^*(\theta| \nX) d\theta \right] \leq  n(M(1)\e_n^2 + M \eta_n^{R}(1)),
	\label{eq:prop3}
	\end{align}
	where \textcolor{black}{$M(1)= 2\left( C_{1}+ M C_4(1)  \right)$}~, and $M= \frac{2C_1}{\min(C,\lambda,1)}$.
\end{theorem}

Note that, the second term (inside the expectation) in the definition of $\eta_n^{R}(1)$ could result in either $\eta_n(0) > \eta_n^{R}(1)$ or vice versa and therefore could play an important role in comparing the LCVB and naive VB approximations to the true optimal decision. 

The next result establishes a bound on the optimality gap of the  LCVB estimated optimal value $R( \alc,\theta_0)$ from the true optimal value $V_0=\inf_{z \in \mathcal{A}} R(z,\theta_0) $.
\begin{theorem}\label{thrm:LCOGvalue}
    Suppose that the set $\sA$ is compact and Assumptions~\ref{assump:Asf1},~\ref{assump:Asf2},~\ref{assump:Asf3}, ~\ref{assump:Asf14}, and~\ref{assump:Asf15} are satisfied with $\min(C,(C_4(1)+C_5(1))) > C_2 + C_3 + C_4(1)+ 2$  for some positive constants $C,C_2,C_3,C_4(1), $ and $C_5(1)$. 
    Then, for any $\tau>0$, the $P_0^n-$ probability of the following event 
    \begin{align}
        \bigg\{ \nX:   R( \alc,\theta_0) - \inf_{z \in \mathcal{A}} R(z,\theta_0)  \leq 2\tau  \left[  (M(1)\e_n^2 + M \eta_n^{R}(1)) \right]^{\frac{1}{2}}  \bigg\}
    \end{align}
    is at least $1-\tau^{-1}$, where \textcolor{black}{$M(1)= 2\left( C_{1}+ M C_4(1)  \right)$}~, and $M= \frac{2C_1}{\min(C,\lambda,1)}$.
    \end{theorem}
Next, we bound the optimality gap between the approximate LC optimal decision rule $\alc$ and the  true optimal decision. 

\begin{theorem}\label{thrm:LCOGdecision}
    Suppose that the set $\sA$ is compact and  $R(a,\theta)$ has a growth function $\Psi(d)$  such that $ \Psi(d)/d^{\delta} = \kappa$ for a $\delta>0$. Moreover, Assumptions~\ref{assump:Asf1},~\ref{assump:Asf2}, ~\ref{assump:Asf3}, ~\ref{assump:Asf14}, and~\ref{assump:Asf15} are satisfied with $\min(C,(C_4(1)+C_5(1))) > C_2 + C_3 + C_4(1)+ 2$  for some positive constants $C,C_2,C_3,C_4(1), $ and $C_5(1)$. Then, for any $\tau>0$, the $P_0^n-$ probability of the following event 
	\[ \left\{ H(\alc, \underset{z \in \mathcal{A} }{\arg \min}   ~ R(z,\theta_0) )    \leq  \left[\frac{2 \tau \left[  (M(1)\e_n^2 + M \eta_n^{R}(1)) \right]^{\frac{1}{2}}}{\kappa} \right]^\frac{1}{\delta} \right\}  \]
	is at least $1-\tau^{-1}$, where \textcolor{black}{$M(1)= 2\left( C_{1}+ M C_4(1)  \right)$}~, and $M= \frac{2C_1}{\min(C,\lambda,1)}$.
%
\end{theorem}
	

\section{Applications}~\label{sec:App}
We illustrate our theoretical findings with the help of three examples:  the  single and multi-product \textit{newsvendor model} and Gaussian process classification. In the examples, we study the interplay between sample size $n$ and the  risk parameter $\gamma$, and their effect on  the optimality gap in decisions and values.

\subsection{Single-product Newsvendor Model}\label{sec:news}

In this section, we study a canonical data-driven decision-making problem with a `well-behaved' risk function $R(a,\theta)$, the data-driven newsvendor model. This problem has received extensive study in the literature, and remains a cornerstone of inventory management~\citep{Sc1960,bertsimas2005data,levi2015data}. Recall that the newsvendor loss function is defined as $$\ell( a,\xi) := h(  a-\xi)^+ + b(\xi-  a)^+$$
where $h$  (underage cost) and $b$ (overage cost) are given positive constants, $\xi  \in  [0,\infty)$ the random demand, and 
$a$ the inventory or decision variable, typically assumed to take values in a compact decision space $\mathcal{A}$ with $\underline{a} := \min \{a: a\in \mathcal A\}$ and $\bar{a} := \max \{a: a\in \mathcal A\}$, and $\underline{a} >0$. The distribution over the random demand, $P_{\theta}$ is assumed to be exponential with 
unknown rate parameter $\theta \in (0,\infty)$. 
The model risk can easily be derived as 
\begin{align}\label{eq:news-loss}
R(a,\theta) := \mathbb{E}_{P_{\theta}}[\ell(a,\xi)]= ha- \frac{h}{\theta} + (b+h) \frac{e^{-a \theta}}{\theta},
\end{align}
 which is convex in $a$. We assume that $\nX := \{\xi_1, \xi_2 \ldots \xi_n\}$ be $n$ observations of the random demand, assumed to be i.i.d random samples drawn from $P_{0}$. 
%

We fix the model space $\Theta=[T,\infty)$ for some $T>0$ and assume that $\theta_0$ lies in the interior of $\Theta$.
We now assume a non-conjugate truncated inverse-gamma ($\text{Inv}-\Gamma$) prior distribution restricted to $\Theta$, with shape and rate parameter $\alpha$ and $\beta$ respectively, that is for a set $A\subseteq \Theta$, we define $\Pi(A) = \text{Inv}-\Gamma_{\Theta}(A;\alpha,\beta)= \text{Inv}-\Gamma(A\cap \Theta;\alpha,\beta)/\text{Inv}-\Gamma(\Theta;\alpha,\beta) $  . We now verify Assumptions~\ref{assump:Asf2},~\ref{assump:Asf1},~\ref{assump:Asf3},~\ref{assump:Asf15} and ~\ref{assump:Asf14} (in that order) in this newsvendor setting. The proofs of the lemmas are delayed to the electronic companion for readability.

First, we fix the sieve set $\Theta_n(\e) = \Theta$, which clearly implies that the restricted inverse-gamma prior $\Pi(\theta)$, places no mass on the complement of this set and therefore satisfies Assumption~\ref{assump:Asf2}.

Second, under the condition that the true demand distribution is exponential with parameter $\theta_0$ (and $P_0 \equiv P_{\theta_0}$), we demonstrate the existence of test functions satisfying Assumption~\ref{assump:Asf1}. 
\begin{lemma}\label{lem:nv2}
	Fix $n\geq 5$. Then, for any $\e>\epsilon_n := \frac{1}{\sqrt{n}}$ with $\e_n \to 0$, and $n\e_n^2 \geq 1$, there exists a test function $\phi_n$~(depending  on $\e$) such that $L_n^{NV}(\theta,\theta_0) = n \left(\sup_{a \in \mathcal{A}} \vert R(a,\theta) - R(a,\theta_0)  \vert \right)^2 $	satisfies
    \vspace{0em}
    \begin{align}
        \bbE_{P^{n}_0}[ \phi_n ]  & \leq C_0 \exp(-C n \e^2 ), \label{eq:eq12nv} \\
        \underset{\{ \theta\in \Theta : L_n^{NV}(\theta,\theta_0) \geq C_1n \e^2  \}}{\sup} \bbE_{P^{n}_{\theta}}[ 1- \phi_n ]  & \leq \exp(-C n \e^2 ),\label{eq:eq13nv}
    \end{align} 
    where $C_0=20$ and $C = \frac{C_1}{2} (K_1^{NV})^{-2}$ for a constant $C_1>0$ and $K_1^{NV} = $ $\frac{\left[\left( \frac{h}{\theta_0}- \frac{h}{T} \right)^2  + (b+h)^2 \left( \frac{e^{-\underline{a} T}}{T}  -  \frac{e^{- \underline{a} \theta_0}}{\theta_0}\right)^2\right]^{1/2}}{d_{H}(T,\theta_0)} $.
\end{lemma}

The proof of the above result follows by showing that $d_L^{NV}=n^{-1/2}\sqrt{L^{NV}_n(\theta,\theta_0)} $ can be bounded above by the Hellinger distance between two exponential distributions on $\Theta$ (under which a test function  exists)  in Lemma~\ref{lem:HellNV} in the appendix.

Third, we show that there exist appropriate constants such that the inverse-gamma prior satisfies Assumption~\ref{assump:Asf3} when the demand distribution is exponential.
\begin{lemma}\label{lem:nv3}
	Fix $ n_2\geq 2$ and any $\lambda >1$. Let   		\(A_n := \left\{ \theta \in \Theta :D_{1+\lambda} \left( P_0^n \| P_{\theta}^n \right) \leq C_3 n \e_n^2 \right\}\), 	where $D_{1+\lambda} \left( P_0^n \| P_{\theta}^n \right)$ is the R\'enyi divergence between $P_0^n$ and  $P_{\theta}^n$.
    Then for $\e_n^2=\frac{\log n}{n}$ and any $C_3>0$ such $C_2=\alpha C_3 \geq 2$, the truncated inverse-gamma prior $\text{Inv}-\Gamma_{\Theta}(A;\alpha,\beta)$ satisfies 
	\vspace{0em}
	\begin{align*}
		\Pi(A_n)  \geq \exp(-n C_2 \e_n^2) , \forall n\geq n_2. 
	\end{align*} 
\end{lemma}

Fourth, it is straightforward to see that the newsvendor model risk $R(a,\theta)$ is bounded below for a given $a \in \sA$.
\begin{lemma}\label{lem:nv4}
	For any $a \in \sA$ and positive constants $h$ and $b$, the newsvendor model risk
	\[   R(a,\theta) = \left( ha- \frac{h}{\theta} + (b+h) \frac{e^{-a \theta}}{\theta} \right) \geq  \left(\frac{  h \underline{a}^2 {\theta^*}}{(1+a\theta^*)}\right), \]
	where $\underline{a} := \min \{a\in \mathcal A\}$ and $\theta^*$ satisfies \(h - (b+h)e^{-a\theta^*}(1+a\theta^*) = 0  \).
\end{lemma}
 
This implies that $R(a,\theta)$ satisfies Assumption~\ref{assump:Asf15}. 
Finally, we also show that the newsvendor model risk satisfies Assumption~\ref{assump:Asf14}.
\begin{lemma}\label{lem:nv5}
	Fix $n \geq 1$ and $\gamma>0$. For any $\e>\e_n$ and any $a\in  \sA$, $R(a,\theta )$ satisfies  
	\begin{align*}	
		\bbE_{\Pi}[\Ind{\left\{{ R(a,\theta ){\gamma}} > {C_4(\gamma)n \e^2} \right\} } e^{ {\gamma} R(a,\theta)}]  \leq \exp(-C_5(\gamma) n\e^2),
	\end{align*}
	for any  $C_4(\gamma) > 2\gamma \left(h\overline{a} + \frac{b}{T}  \right) $ and $C_5(\gamma)= C_4(\gamma) - 2\gamma \left(h\overline{a} + \frac{b}{T}  \right) $, where $\overline{a} := \max \{a\in \mathcal A\}$.
\end{lemma}

Note that Lemma~\ref{lem:nv2} implies that $C=\frac{C_1}{2 (K_1^{NV})^2}$ for any constant $C_1>0$. Fixing $\alpha=1$ and using Lemma~\ref{lem:nv3} we can choose $C_2=C_3=2$. Now,  $C_1$ can be chosen large enough such that $C > C_4(\gamma)+C_5(\gamma)$ for a given risk sensitivity  $\gamma>0$. Therefore, the condition on constants in Theorem~\ref{thrm:thm1} reduces to $C_5(\gamma)> 2+ C_2+C_3 =5$, and it can be satisfied easily by fixing $C_5(\gamma)=5.1 $(say). 

These lemmas show that when the demand distribution is exponential and with a non-conjugate truncated inverse-gamma prior, our results in Theorem~\ref{thrm:OGValue} and~\ref{thrm:OGdecision} can  be used for RSVB method to bound the optimality gap in decisions and  values for various values of the risk-sensitivity parameter $\gamma$. 
Recall that the bound obtained in Theorem~\ref{thrm:OGdecision} depends on $\e_n^2$ and $ \eta_n^R(\gamma)$.

Lemma~\ref{lem:nv3} implies that $\e_n^2=\frac{\log n}{n}$, but in order to get the complete bound we further need to characterize $\eta_n^R(\gamma)$.  Recall that, as a consequence of Assumption~\ref{assump:Asf11} in Proposition~\ref{prop:eta_n}, for a given $C_8= -  \inf_{Q \in \sQ} \inf_{a \in \sA}  \bbE_Q[R(a,\theta)]  $ that $C_9>0$ and \(\eta_n^R(\gamma) \leq \gamma  n^{-1}C_8 +C_9\e_n'^2. \)
 
Therefore, in our next result, we show that in the newsvendor setting, we can construct a sequence $\{Q_n(\theta)\} \subset \sQ$ that satisfies Assumption~\ref{assump:Asf11}, and thus identify $\e_n'$ and the constant $C_9$. We fix $\sQ$ to be the family of shifted gamma distributions with support $[T,\infty)$. 
\begin{lemma}\label{lem:nv6}
	Let $\{Q_n(\theta)\}$ be a sequence of shifted gamma distributions with shape parameter $a = n$ and rate parameter $b= \frac{n}{\theta_0}$, then for truncated inverse gamma prior and exponentially distributed likelihood model
	\[ \frac{1}{n} \left[ \scKL\left(Q_n(\theta)\|\Pi(\theta) \right) + \bbE_{Q_n(\theta)} \left[ \scKL\left(dP^n_{0}(\nX))\| dP^n_{\theta}(\nX) \right) \right]  \right] \leq C_9 \e_n'^2,  \]
where $\e_n'^2= \frac{\log n }{n}$ and $C_9 = \frac{1}{2} + \max \left(0,  2+  \frac{2\beta}{\theta_0}  - \log \sqrt{2\pi }    - \log \left( \frac{\beta^\alpha}{\Gamma(\alpha)} \right)  + \alpha\log \theta_0 \right)$ and prior parameters are chosen such that $C_9>0$.
	\end{lemma}

As a specific instance, consider the naive VB case. Since $\gamma\to 0^{+}$, the term $\eta_n(0)$ in Theorem~\ref{thrm:NVOGdecision} is bounded above by $C_9 \e_n'^2$ , where $C_9$ and $\e_n'^2$ are derived in the result above. For the LCVB case, observe that Lemma~\ref{lem:nv4} implies that $R(\cdot,\cdot)$ is bounded below and therefore $C_8 \leq - \left( \frac{  h \underline{a}^2 {\theta^*}}{(1+\bar a\theta^*)} \right) $, where $h,\underline{a}, \bar a,  \text{  and } \theta^*$ are given to  the modeler or are easily computable. Now since  $C_8< 0$, it is straight forward to observe that $\eta_n^{R}(\gamma)$ term in Theorem~\ref{thrm:LCOGdecision} is  bounded above by $C_9 \e_n'^2$. 

Now, using the result established in Lemmas above, we bound the optimality gap in  values for the single product newsvendor model risk.
\begin{theorem}\label{thrm:OGValue_NV}
    Fix $\gamma>0$. Suppose that the set $\sA$ is compact. 
    Then, for the newsvendor model with exponentially distributed demand with  rate $\theta \in \Theta = [T,\infty)$, prior distribution $\Pi(\cdot) = \text{Inv}-\Gamma_{\Theta}(\cdot;\alpha,\beta)= \text{Inv}-\Gamma(A\cap \Theta;\alpha,\beta)/\text{Inv}-\Gamma(\Theta;\alpha,\beta) $, and the variational family fixed to shifted (by $T>0$) gamma distributions,  
    and for any $\tau>0$, 
    the $P_0^n-$ probability of the following event 
    \begin{align}
        \bigg\{ \nX:   R( \ars,\theta_0) - \inf_{z \in \mathcal{A}} R(z,\theta_0)  \leq 2\tau  \textcolor{black}{M'(\gamma)} \left(\frac{\log n }{n }\right)^{1/2}  \bigg\}
    \end{align}
    is at least $1-\tau^{-1}$ for sufficiently large $n$ and for some mapping $\textcolor{black}{M':\bbR^{+}\to\bbR^{+}}$, where $R(\cdot,\theta)$ is the newsvendor model risk.
\end{theorem}
\begin{proof}
    The proof is a direct consequence of Theorem~\ref{thrm:OGValue}, Lemmas~\ref{lem:nv2},~\ref{lem:nv3},~\ref{lem:nv4}, ~\ref{lem:nv5}, ~\ref{lem:nv6}, and Proposition~\ref{prop:eta_gamma}.
\end{proof}

Next, we bound the  optimality gap between the approximate optimal decision rule $\ars$ and the true optimal decision. The bound, in particular, depends on the curvature of $R(a,\theta_0)$ around the true optimal decision, defined using the growth condition  in Assumption~\ref{assump:Growth}. 

\begin{theorem}\label{thrm:OGdecision_NV}
    \noindent Fix $\gamma>0$. Suppose that the set $\sA$ is compact and $R(a,\theta_0)$ satisfies the growth condition in~Assumption~\ref{assump:Growth},  with $\Psi(d)$  such that $  \Psi(d)/d^{\delta} = \kappa$, for any  $\delta >0$. Then, for the newsvendor model with exponentially distributed demand with  rate $\theta \in \Theta = [T,\infty)$, prior distribution $\Pi(\cdot) = \text{Inv}-\Gamma_{\Theta}(\cdot;\alpha,\beta)= \text{Inv}-\Gamma(A\cap \Theta;\alpha,\beta)/\text{Inv}-\Gamma(\Theta;\alpha,\beta) $, and the variational family fixed to shifted (by $T>0$) gamma distributions,   and for any $\tau>0$,
    the $P_0^n-$ probability of the following event 
    \[ \left\{ \tilde X_n :  H \left(\ars(\tilde X_n), \underset{z \in \mathcal{A} }{\arg \min}   ~ R(z,\theta_0) \right)    \leq  \left[\frac{2 \tau}{\kappa} \textcolor{black}{M'(\gamma)} \left(\frac{\log n }{n }\right)^{1/2} \right]^\frac{1}{\delta} \right\}  \]
    is at least $1-\tau^{-1}$ for sufficiently large $n$ and for some mapping $\textcolor{black}{M':\bbR^{+}\to\bbR^{+}}$, where $R(\cdot,\theta)$ is the newsvendor model risk.
\end{theorem}

\begin{proof}
    The proof is a direct consequence of Theorem~\ref{thrm:OGdecision}, Lemmas~\ref{lem:nv2},~\ref{lem:nv3},~\ref{lem:nv4}, ~\ref{lem:nv5}, ~\ref{lem:nv6}, and Proposition~\ref{prop:eta_gamma}.
\end{proof}



%

Next, we demonstrate the effect of varying the risk-sensitivity parameter $\gamma$. We fix $\theta_0 = 0.1$, $b=1$, $h=5$, $\alpha =1 \text{, and } \beta=4.1$. We run RSVB algorithm  with $\gamma\in \{0 (\text{ naive }),1,2,4.5,5,6\}$ and repeat the experiment over 100 sample paths. We plot the results in~Figure~\ref{fig:fig5}. In Figure~\ref{fig:fig5}(a) and (b), we plot the optimality gap in values and decisions, that is $R(\ars(\gamma),\theta_0)- R(a_0^*,\theta_0)$ and $|\ars(\gamma) - a_0^*|$ respectively, for various values of $\gamma$. We observe that the gap decreases when $n$\ increases. 
This  observation supports our results in Propositions~\ref{prop:eta_n} and~\ref{prop:eta_gamma} that establishes the properties of $\eta_n^R(\gamma)$ as $n$ increases.
Lastly, in Figure~\ref{fig:fig5}(c), we plot the variance of the RSVB posterior as $n$ increases for various values of $\gamma$; as anticipated the variance reduces as $n$ increases. To observe the effect of $\gamma$, first recall that as $\gamma$ increases the decision maker become more risk averse and so is our algorithmic framework RSVB. Indeed, from the rightmost variance plot in Figure~\ref{fig:fig5} it is evident  that for larger value of $\gamma$ ($>4$) the RSVB posterior is more concentrated on the subset  of $\Theta$, where risk is more and consequently  we observe large optimality gaps in values and decision (see first two plots in Figure~\ref{fig:fig5} ). Moreover, as $n$ increase the  effect of larger $\gamma$ reduces, since  as $n$ increases the incentive to deviate  from the  posterior reduces (due to increased \scKL~divergence dominance for larger $n$ in  RSVB). 


\begin{figure}[ht]
    \begin{subfigure}[b]{0.32\textwidth}
        \includegraphics[width=\textwidth]{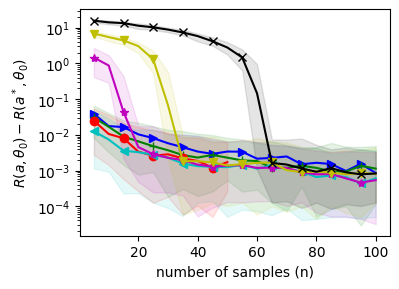}
        \caption{OG in Values}
    \end{subfigure}
    \hfill
    \begin{subfigure}[b]{0.32\textwidth}
        \includegraphics[width=\textwidth]{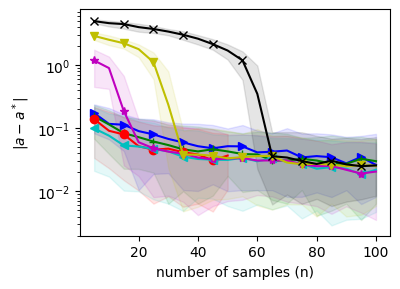}
        \caption{OG in Decisions}
    \end{subfigure}
    \hfill
    \begin{subfigure}[b]{0.32\textwidth}
        \includegraphics[width=\textwidth]{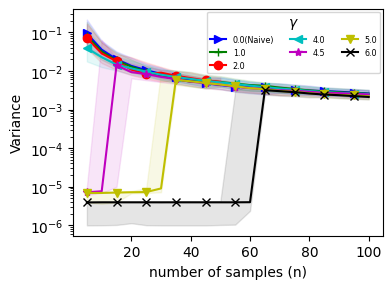}
        \caption{Variance}
    \end{subfigure}
    \caption{Optimality gap in values and decisions, and the variance of the RSVB posterior (mean over 100 sample paths) against the number of samples ($n$) for various values of $\gamma$.}
        \label{fig:fig5}
\end{figure}

\subsection{Multi-product newsvendor problem}~\label{sec:Mnews}

Analogous to the one-dimensional newsvendor loss function,  the loss function in its multi-product version is defined as $$\ell( {a},\xi) := h^T(  a-\xi)^+ + b^T(\xi-  a)^+$$
where $h$ and $b $ are given vectors of underage and overage costs respectively for each product and mapping $(\cdot)^+$ is defined component-wise. We assume that there are $d$ items or products and $\xi  \in  \bbR^{d}$ denotes the random vector of demands.
Let $a\in \sA\subset \bbR^d_+$ be the inventory or decision variable, typically assumed to take values in a compact decision space $\mathcal{A}$ with $\underline{a} := \{\{\min \{a_i: a_i\in \mathcal A_i\}\}_{i=1}^{d}$ and $\bar{a} := \{\{\max \{a_i: a_i\in \mathcal A_i\}\}_{i=1}^{d}$, and $\underline{a} > 0$, where $\sA_i$ is the marginal set of $i^{th}$ component of $\sA$. The random demand is assumed to be multivariate Gaussian, with 
unknown mean parameter $\theta \in \bbR^d$ but with known covariance matrix $\Sigma$. We also assume that $\Sigma$ is a symmetric positive definite matrix and can be decomposed as $Q^T\Lambda Q$, where $Q$ is an orthogonal matrix and $\Lambda$ is a diagonal matrix consisting of respective eigenvalues of $\Sigma$. We also define $\overline \Lambda = \max_{i\in \{1,2,\ldots d\} } 
\Lambda_{ii}$ and $\underbar{$\Lambda$} = \min_{i\in \{1,2,\ldots d\} } 
\Lambda_{ii}$. 
The model risk 
\begin{align}\label{eq:news-lossd}
    \nonumber
    R(a,\theta) &= \mathbb{E}_{P_{\theta}}[\ell(a,\xi)]
    = \sum_{i=1}^{d} \bbE_{P_{\theta_i}}[h_i (a_i - \xi_i)^+ + b_i (\xi_i - a_i)^+ ]
    \\
    \nonumber
    &= \sum_{i=1}^{d} \Bigg[(h_i+b_i)a_i \Phi\left(\frac{(a_i-\theta_{i})}{\sigma_{ii}}\right) - b_i a_i + \theta_{i} (b_i-h_i)  
    \\
    \nonumber
    &\quad \quad + \sigma_{ii} \left[ h\frac{\phi\left(\frac{(a_i-\theta_{i})}{\sigma_{ii}}\right)}{\Phi\left(\frac{(a_i-\theta_{i})}{\sigma_{ii}}\right)} + b\frac{\phi\left(\frac{(a_i-\theta_{i})}{\sigma_{ii}}\right)}{1-\Phi\left(\frac{(a_i-\theta_{i})}{\sigma_{ii}}\right)} \right] \Bigg],
\end{align}
which is convex in $a$. Here $P_{\theta_i}$ is the marginal distribution of $\xi$ for $i^{th}$ product, $\phi(\cdot)$ and $\Phi(\cdot)$ are probability and cumulative distribution function of the standard Normal distribution. We also assume that the true mean parameter $\theta_0$ lies in a  compact subspace $\Theta\subset \bbR^d$. We fix the prior to be uniformly distributed on  $\Theta$ with no correlation across its components, that  is $\pi(A)= \frac{m(A\I \Theta)}{m(\Theta)} = \prod_{i=1}^{d} \frac{m(A_i\I \Theta_i)}{m(\Theta_i)}$, where $m(B)$  is the Lebesgue measure (or volume) of $B\subset \bbR^d$ 
As in the previous  example, we fix the sieve set $\Theta_n(\e) = \Theta$, which clearly implies that $\Pi(\theta)$ places no mass on the complement of this set and therefore satisfies Assumption~\ref{assump:Asf2}. 

Then under the condition that the true demand distribution has  a multivariate Gaussian distribution (with known $\Sigma$)  and mean $\theta_0$ ($P_0 \equiv P_{\theta_0}$), we demonstrate the existence of test functions satisfying Assumption~\ref{assump:Asf1} by constructing a test function unlike the single-product newsvendor problem with exponential demand.

\begin{lemma}~\label{lem:mnv2}
    Fix $n\geq 1$. Then, for any $\e>\epsilon_n := \frac{1}{\sqrt{n}}$ with $\e_n \to 0$, and $n\e_n^2 \geq 1$ and  test function $\phi_{n,\e} := \mathbbm{1}_{ \left\{ \nX:  \left\|\hat \theta_n - \theta_0 \right\| >  \sqrt{\tilde C\e^2} \right\} }$,  $L_n^{MNV}(\theta,\theta_0) = n$ $ \left(\sup_{a \in \mathcal{A}} \vert R(a,\theta) - R(a,\theta_0)  \vert \right)^2 $	satisfies
    \vspace{0em}
    \begin{align}
        \bbE_{P^{n}_0}[ \phi_n ]  & \leq C_0 \exp(-C n \e^2 ),  \\
        \underset{\{ \theta\in \Theta : L_n^{MNV}(\theta,\theta_0) \geq C_1n \e^2  \}}{\sup} \bbE_{P^{n}_{\theta}}[ 1- \phi_n ]  & \leq \exp(-C n \e^2 ),\label{eq:eq13mnv}
    \end{align} 
    with $C_0=1$, $C_1= 4 K^2 C$ and $C= 1/8  \left(\frac{\tilde C}{d\overline \Lambda} -1\right) $ for sufficiently large $\tilde C$ such that $C>1$ and $\overline \Lambda = \max_{i\in \{1,2,\ldots d\} } 
    \Lambda_{ii}$, where $K=\sup_{\sA,\Theta} \|\partial_{\theta} R(a,\theta)\|$.
\end{lemma}


In the following result, we show that there exist appropriate constants such that prior distribution satisfies Assumption~\ref{assump:Asf3} when the demand distribution is a multivariate Gaussian with unknown mean.
\begin{lemma}\label{lem:mnv3}
    Fix $ n_2\geq 2$ and any $\lambda >1$. Let   		\(A_n := \left\{ \theta \in \Theta :D_{1+\lambda} \left( P_0^n \| P_{\theta}^n \right) \leq C_3 n \e_n^2 \right\}\), 	where $D_{1+\lambda} \left( P_0^n \| P_{\theta}^n \right)$ is the R\'enyi Divergence between $P_0^n$ and  $P_{\theta}^n$. Then for $\e_n^2=\frac{\log n}{n}$ and any $C_3>0$ such that $C_2=\frac{4d}{\overline \Lambda(\lambda+1)\left(\prod_{i=1}^{d} m(\Theta_i)\right)^{2/d}}  C_3 \geq 2$ and for large enough $n$, the uncorrelated uniform prior restricted to $\Theta$ satisfies 
    \begin{align*}
        \Pi(A_n)  \geq \exp(-n C_2 \e_n^2) .
    \end{align*} 
\end{lemma}

Next, it is straightforward to see that the multi-product newsvendor model risk $R(a,\theta)$ is bounded below for a given $a \in \sA$ on a  compact set $\Theta$ and thus it satisfies Assumption~\ref{assump:Asf15}.
Finally, we also show that the newsvendor model risk satisfies Assumption~\ref{assump:Asf14}.
\begin{lemma}\label{lem:mnv5}
    Fix $n \geq 1$ and $\gamma>0$. For any $\e>\e_n$ and $a\in  \sA$, $R(a,\theta )$ satisfies  
    \begin{align*}	
        \bbE_{\Pi}[\Ind{\left\{{ G(a,\theta ){\gamma}} > {C_4(\gamma)n \e^2} \right\} } e^{ {\gamma} G(a,\theta)}] \leq \exp(-C_5(\gamma) n\e_n^2),
    \end{align*}
    for any  $C_4(\gamma) > 2\gamma \sup_{\{a,\theta\}\in \sA \otimes \Theta } G(a,\theta) $ and $C_5(\gamma)= C_4(\gamma) - 2\gamma \sup_{\{a,\theta\}\in \sA \otimes \Theta } G(a,\theta) $.
\end{lemma}

Similar to single product example, in our next result, we show that in the multi-product newsvendor setting, we can construct a sequence $\{Q_n(\theta)\} \in \sQ$ that satisfies Assumption~\ref{assump:Asf11}, and thus identify $\e_n'$ and constant $C_9$. We fix $\sQ$  to be the family of uncorrelated Gaussian distributions restricted to $\Theta$.

\begin{lemma}\label{lem:mnv6}
    Let $\{Q_n(\theta)\}$ be a sequence of product of $d$ univariate Gaussian distribution  defined as $q_n^i(\theta) \propto \frac{1}{\sqrt{2\pi \sigma_{i,n}^2}} e^{-\frac{1}{2\sigma_{i,n}^2}(\theta-\mu_{i,n})^2} \Ind{\Theta_i} =  \frac{\mathcal{N}(\theta_i| \mu_{i,n}, \sigma_{i,n})\Ind{\Theta_i}}{\mathcal{N}(\Theta_i| \mu_{i,n}, \sigma_{i,n})} $ and  fix $\sigma_{i,n}=1/\sqrt{n}$ and  $\theta_{i} = \theta_0^i$  for all $i\in \{1,2,\ldots,d\}$.  Then for uncorrelated uniform distribution restricted to $\Theta$  and multivariate normal  likelihood model
    \[ \frac{1}{n} \left[ \scKL\left(Q_n(\theta)\|\Pi(\theta) \right) + \bbE_{Q_n(\theta)} \left[ \scKL\left(dP^n_{0}(\nX))\| dP^n_{\theta}(\nX) \right) \right]  \right] \leq C_9 \e_n'^2,  \]
    where $\e_n'^2= \frac{\log n }{n}$ and $C_9:= \frac{d}{2} + \max \left(0,   -\sum_{i=1}^{d}  [\log (\sqrt{2\pi e} ) -\log(m(\Theta_i))  ] + \frac{d}{2} {\underbar{$\Lambda$}^{-1}} \right)$.
\end{lemma}

Now, using the result established in lemmas above, we bound the optimality gap in  values for the multi-product newsvendor model risk.
\begin{theorem}\label{thrm:OGValue_MNV}
    Fix $\gamma>0$. Suppose that the set $\sA$ is compact.
    Then, for the multi-product newsvendor model with multivariate Gaussian distributed demand with known covariance  matrix $\Sigma$ and unknown mean vector $\theta $ lying  in a compact subset $\Theta\subset \bbR^d$, prior   $\Pi(\cdot) = \prod_{i=1}^{d}\frac{m(\{\cdot\}\cap \Theta_i)}{m(\Theta_i)}$, and the variational family fixed to uncorrelated Gaussian distribution restricted to $\Theta$,  
    and for any $\tau>0$, 
    the $P_0^n-$ probability of the following event 
    \begin{align}
        \bigg\{ \nX:   R( \ars,\theta_0) - \inf_{z \in \mathcal{A}} R(z,\theta_0)  \leq 2\tau  \textcolor{black}{M'(\gamma)} \left(\frac{\log n }{n }\right)^{1/2}  \bigg\}
    \end{align}
    is at least $1-\tau^{-1}$ for sufficiently large $n$ and for some mapping $\textcolor{black}{M':\bbR^{+}\to\bbR^{+}}$, where $R(\cdot,\theta)$ is the multi-product newsvendor model risk.
\end{theorem}
\begin{proof}
    The proof is a direct consequence of Theorem~\ref{thrm:OGValue}, Lemmas~\ref{lem:mnv2},~\ref{lem:mnv3},~\ref{lem:mnv5},~\ref{lem:mnv6}, and Proposition~\ref{prop:eta_gamma}.
\end{proof}

Next, we bound the  optimality gap between the approximate optimal decision rule $\ars$ and the true optimal decision. 

\begin{theorem}\label{thrm:OGdecision_MNV}
    \noindent Fix $\gamma>0$. Suppose that the set $\sA$ is compact and $R(a,\theta_0)$ satisfies the growth condition in~Assumption~\ref{assump:Growth},  with $\Psi(d)$  such that $  \Psi(d)/d^{\delta} = \kappa$, for any  $\delta >0$. 
    Then, for the multi-product newsvendor model with multivariate Gaussian distributed demand with known covariance  matrix $\Sigma$ and unknown mean vector $\theta $ lying  in a compact subset $\Theta\subset \bbR^d$, prior   $\Pi(\cdot) = \prod_{i=1}^{d}\frac{m(\{\cdot\}\cap \Theta_i)}{m(\Theta_i)}$, and the variational family fixed to uncorrelated Gaussian distribution restricted to $\Theta$,  
    and for any $\tau>0$, 
    the $P_0^n-$ probability of the following event 
    \[ \left\{ \tilde X_n :  H \left(\ars(\tilde X_n), \underset{z \in \mathcal{A} }{\arg \min}   ~ R(z,\theta_0) \right)    \leq  \left[\frac{2 \tau}{\kappa} \textcolor{black}{M'(\gamma)} \left(\frac{\log n }{n }\right)^{1/2} \right]^\frac{1}{\delta} \right\}  \]
    is at least $1-\tau^{-1}$ for sufficiently large $n$ and for a known function $\textcolor{black}{M'(\gamma)}$, where $R(\cdot,\theta)$ is the multi-product newsvendor model risk.
\end{theorem}

\begin{proof}
    The proof is a direct consequence of Theorem~\ref{thrm:OGdecision}, Lemmas~\ref{lem:mnv2},~\ref{lem:mnv3},~\ref{lem:mnv5},~\ref{lem:mnv6}, and Proposition~\ref{prop:eta_gamma}.
\end{proof}

%
%

\subsection{Gaussian  process classification}

Consider a problem of classifying an input pattern or features $Y$ lying in measure space $([0,1]^d,\sY,\nu)$ into one of two classes $\{-1,1\}$, where $\xi\in\{-1,1\}$ denote the class of $Y$. For  a given  $Y$,  we model the classifier using a Bernoulli distribution $p(\xi|Y,\theta)= \Psi_{\xi}(\theta(Y))$, where  $\theta:[0,1]^d \to \bbR$ is a non-parametric model parameter in a separable Banach space $(\Theta,\|\cdot\|)$ and measurable functions $\Psi_{1}(x)= (1+e^{-x})^{-1}$ and $\Psi_{-1}(x) = 1-\Psi_{1}(x)$. Note that $\Psi_{1}(\cdot)$ is a logistic function. We denote $\psi(\cdot)$ as the derivative of $\Psi_{1}(\cdot)$. We assume that $\nu(\cdot)$ is independent of $\xi$. Thus the sequence of independent observations  $\{\nY,\nX\}=\{(Y_1,\xi_1),(Y_2,\xi_2),\ldots$ $,(Y_n,\xi_n)\}$ are assumed to be generated  from model \[ dP_{\theta}(\xi,Y) = p(\xi|Y,\theta) \nu(Y). \]

In the above binary classification  problem, the objective is to estimate $\theta(\cdot)$ using the observation vector $\{\nY,\nX\}$.   
We posit a Gaussian process (GP) prior $\Pi(\cdot)$ on $\theta(\cdot)\in \Theta$ (to be defined later).  We also assume that  $\nu(\cdot)$  is known and  we do not place any  prior on it. Consequently, the posterior distribution over $\theta(\cdot)$ given observations $\{\nY,\nX\}$ can be defined as 
\[ d\Pi (\theta|\{\nY,\nX\}) = \frac{d\Pi(\theta) \prod_{i=1}^{n}\Psi_{\xi_i}(\theta(Y_i)) \nu(Y_i)  }{\int  \prod_{i=1}^{n}\Psi_{\xi_i}(\theta(Y_i)) \nu(Y_i) d\Pi(\theta) } = \frac{d\Pi(\theta) \prod_{i=1}^{n}\Psi_{\xi_i}(\theta(Y_i))  }{\int  \prod_{i=1}^{n}\Psi_{\xi_i}(\theta(Y_i))  d\Pi(\theta) } .\] 
Consider the loss function $\ell(a,\xi) $ defined  as
\begin{align} 
    \ell(a,\xi) : = \begin {cases} 
    0, \quad ~\text{if } a=\xi,
    \\
    c_+, \quad \text{if } a=+1, \xi=-1,
    \\
    c_-, \quad \text{if } a=-1, \xi=+1,
    \end{cases}
    \end{align}  
where $c_+$  and $c_-$  are known positive constants.
The model risk is given  by
\begin{align}\label{eq:GPCloss}
    R(a,\theta) &= \bbE_{P_{\theta}} [\ell(a,\xi)] = \begin {cases} 
    c_+\bbE_{\nu}[\Psi_{-1}(\theta(y))], \quad a=+1,
    \\
    c_- \bbE_{\nu}[\Psi_{1}(\theta(y))], \quad a=-1.
\end{cases} 
    \end{align}
We define the distance function as $L_n^{GP}(\theta,\theta_0) = n \left( \sup_{a \in \mathcal{A}} | R(a,\theta) - R(a,\theta_0)  | \right)^2 $. In anticipation of demonstrating that the binary classification  model with  GP prior and distance function $L_n^{GP}$  satisfy the desired set of assumptions, we recall the following result,  from~\cite{Zanten08}, which will be central in establishing Assumptions~\ref{assump:Asf1},~\ref{assump:Asf2}, and~\ref{assump:Asf3}.

\begin{lemma}~\label{lem:GP}[Theorem 2.1~\cite{Zanten08}]
    Let $\theta(\cdot)$ be  a Borel measurable, zero-mean  Gaussian random element in  a separable Banach space $(\Theta,\|\cdot\|)$ with reproducing  kernel Hilbert space (RKHS) $(\bbH,\|\cdot\|_{\bbH})$ and let $\theta_0$ be contained in the closure of $\bbH$ in $\Theta$. For any $\e>\e_n$ satisfying $\varphi_{\theta_0}(\e)\leq n\e^2$ , where 
    \begin{align}
        \varphi_{\theta_0}(\e)= \inf_{h\in \bbH:\|h-\theta_0\|<\e} \|h\|^2_{\bbH} - \log \Pi(\|\theta\| < \e) 
        \label{eq:psi}
        \end{align}
     and any $C_{10}>1$ with  $e^{-C_{10} n\e_n^2} < 1/2 $, there exists a measurable set $\Theta_n(\e) \subset \Theta$  such that 
    \begin{align}
        \label{eq:testGP}
        \log N (3\e, \Theta_n(\e),\|\cdot\|) &\leq 6C_{10} n\e^2,
        \\
        \Pi(\theta \notin \Theta_n(\e) ) &\leq e^{-C_{10} n\e^2},
        \label{eq:priorGP}
        \\
        \Pi(\|\theta-\theta_0\|<4\e_n) &\geq e^{-n\e_n^2}.
        \label{eq:renyiGP}
        \end{align}
    \end{lemma}
The proof of the lemma above can be easily adapted from the proof of~~\cite[Theorem 2.1]{Zanten08}, which  is specifically for $\e=\e_n$. 
Notice that the result above is true  for any norm $\|\cdot\|$ on the Banach space if that satisfies $\varphi_{\theta_0}(\e)\leq n\e^2$.
Moreover, if $\varphi_{\theta_0}(\e_n)\leq n\e_n^2$ is true, then it also holds for any $\e>\e_n$, since by definition $\varphi_{\theta_0}(\e)$ is a decreasing  function of $\e$.

All the  results in the  previous lemma depend on $\varphi_{\theta_0}(\e)$ being less than $ n\e^2$. In particular, observe that the second term in the definition of $\varphi_{\theta_0}(\e)$  depends on the  prior distribution on $\Theta$. Therefore,
~\cite[Theorem 4.5]{Zanten08} show  that $\varphi_{\theta_0}(\e_n)\leq n\e_n^2$ ( with $\|\cdot\|$ as supremum norm and for $\e_n$ as defined later in~\eqref{eq:rate} ) is satisfied by the  Gaussian  prior  of type 
\begin{align} W(\cdot)=  \sum_{j=1}^{\bar J_{\alpha}} \sum_{k=1}^{2^{jd}} \mu_j Z_{j,k} \vartheta_{j,k}(\cdot), 
\label{eq:GPprior}
\end{align}
where $\{\mu_j\}$ is a sequence that decreases with $j$, $\{Z_{i,j}\}$ are i.i.d. standard Gaussian random variables and $\{\vartheta_{j,k}\}$ form  a double-indexed orthonormal  basis (with respect to measure $\nu$), that is $\bbE_{\nu}[\vartheta_{j,k}\vartheta_{l,m}] = \Ind{\{j=l,k=m\}}$).
$\bar J_{\alpha}$ is the smallest integer satisfying $2^{\bar J_{\alpha} d} = n^{d/(2\alpha+d)}$ for  a  given $\alpha>0$. 
In particular, the GP above is constructed using the function class that is supported on $[0,1]^d$ and has a wavelet expansion, 
\[ w(\cdot) = \sum_{j=1}^{\infty} \sum_{k=1}^{2^{jd}} w_{j,k} \vartheta_{j,k}(\cdot). \] The wavelet  function space is equipped with the $L_2-$norm: $\|w\|_2 = \sum_{j=1}^{\infty}$ $ \left( \sum_{k=1}^{2^{jd}} |w_{j,k}|^2 \right)^{1/2}$; the supremum norm: $\|w\|_{\infty}= \sum_{j=1}^{\infty} 2^{jd} \max_{1\leq k\leq 2^{jd}}  |w_{j,k}|$; and the Besov $(\beta,\infty,\infty)-$norm: $\|w\|_{\beta;\infty,\infty}= \sup_{1\leq j < \infty} 2^{j\beta} 2^{jd} $ $\max_{1\leq k\leq 2^{jd}}  |w_{j,k}|$. Note that $W$ induces a measure over the RKHS $\bbH$,  defined as a collection of truncated  wavelet functions
\[ w(\cdot)= \sum_{j=1}^{\bar J_{\alpha}} \sum_{k=1}^{2^{jd}} w_{j,k}  \vartheta_{j,k}(\cdot), \] 
with norm induced by the inner-product on $\bbH$ as
\( \|w\|_{\bbH}^2 = \sum_{j=1}^{\bar J_{\alpha}} \sum_{k=1}^{2^{jd}} \frac{w_{j,k}^2}{\mu_j^2}. \) The RKHS kernel $K:[0,1]^d\times[0,1]^d \mapsto \bbR$ can be easily derived as 
\begin{align*}
    K(x,y)=\bbE[W(x)W(y)]&=  \bbE\left[\left(\sum_{j=1}^{\bar J_{\alpha}} \sum_{k=1}^{2^{jd}} \mu_j Z_{j,k} \vartheta_{j,k}(y) \right)\left(\sum_{j=1}^{\bar J_{\alpha}} \sum_{k=1}^{2^{jd}} \mu_j Z_{j,k} \vartheta_{j,k}(x) \right)\right] 
    \\
    &=\sum_{j=1}^{\bar J_{\alpha}} \sum_{k=1}^{2^{jd}}  \mu_j^2   \vartheta_{j,k}(y)  \vartheta_{j,k}(x). 
\end{align*}
Indeed, by the definition of this kernel and inner product, observe that
\begin{align*}
    \langle K(x,\cdot),w(\cdot)  \rangle
    = \sum_{j=1}^{\bar J_{\alpha}} \sum_{k=1}^{2^{jd}} w_{j,k} \mu_j^2 \vartheta_{j,k}(x)\frac{1}{\mu_j^2} = w(x).
\end{align*}
Moreover,
\( \langle K(x,\cdot), K(y,\cdot) \rangle = \sum_{j=1}^{\bar J_{\alpha}} \sum_{k=1}^{2^{jd}}  \mu_j^2 \vartheta_{j,k}(x) \mu_j^2 \vartheta_{j,k}(y) \frac{1}{\mu_j^2}  = K(x,y).\)
It is clear from its definition that $W$ is a centered Gaussian  random field on the RKHS.

Next, using the definition of the kernel, we derive the covariance operator of the Gaussian random field $W$. Recall that $Y\sim \nu$, which enables us to define  the covariance operator $\mathcal{C}$, following~\cite[(6.19)]{Stuart2010} as
\[ (\mathcal{C } h_{\nu})(x) = \int_{[0,1]^d} K(x,y) h_{\nu}(y) d\nu(y).\]
Also, observe that $\{\mu_j^2,\varphi_{j,k}\}$ is the eigenvalue and eigen function pair of the covariance operator  $\mathcal{C}$. Consequently, using \text{Karhunen Lo\'eve expansion}~\cite[Theorem 6.19]{Stuart2010} the prior induced by $W$ on $\bbH$ is a Gaussian distribution  denoted  as $\mathcal{N}(0,\mathcal{C})$.  We also recall the \text{Cameron-Martin space} denoted  as $\text{Im}(\mathcal{C}^{1/2})$ associated with a Gaussian measure $\mathcal{N}(0,\mathcal{C})$ on $\bbH$ to be the intersection of all linear spaces of full measure under $\mathcal{N}(0,\mathcal{C})$~\cite[(page 530)]{Stuart2010}. In particular, $\text{Im}(\mathcal{C}^{1/2})$ is the Hilbert space with inner product $\langle \cdot,\cdot \rangle_{\mathcal{C}} = \langle \mathcal{C}^{-1/2}\cdot,\mathcal{C}^{-1/2}\cdot \rangle$.

Next, we show the existence of  test functions in the following  result.

\begin{lemma}~\label{lem:gp2}
     For any $\e>\epsilon_n $ with $\e_n \to 0$, $n\e_n^2 \geq 2\log 2$, and $\varphi_{\theta_0}(\e)\leq n\e^2$,  there exists a test function $\phi_n$~(depending  on $\e$) such that $L_n^{GP}(\theta,\theta_0) = n \left( \sup_{a \in \mathcal{A}} | R(a,\theta) - R(a,\theta_0)  | \right)^2 $	satisfies
    \vspace{0em}
    \begin{align}
        \bbE_{P^{n}_0}[ \phi_n ]  & \leq C_0 \exp(-C n \e^2 ), \label{eq:eq12gp} \\
        \underset{\{ \theta\in \Theta : L_n^{GP}(\theta,\theta_0) \geq C_1n \e^2  \}}{\sup} \bbE_{P^{n}_{\theta}}[ 1- \phi_n ]  & \leq \exp(-C n \e^2 ),\label{eq:eq13gp}
    \end{align} 
    where $C=1/6$, $C_0=2$ and $C_1= (\max(c_+,c_-))^2$. 
\end{lemma}


Assumption~\ref{assump:Asf2} is  a direct consequence of~\eqref{eq:priorGP} in ~Lemma~\ref{lem:GP}. Next, we prove that prior distribution and the likelihood model satisfy Assumption~\ref{assump:Asf3}  using~\eqref{eq:renyiGP} of Lemma~\ref{lem:GP}.

\begin{lemma}\label{lem:gp3}
     For any $\lambda>1$, let   		\(A_n := \left\{ \theta \in \Theta :D_{1+\lambda} \left( P_0^n \| P_{\theta}^n \right) \leq C_3 n \e_n^2 \right\}\), 	where $D_{1+\l} \left( P_0^n \| P_{\theta}^n \right)$
     is the R\'enyi Divergence between $P_0^n$ and  $P_{\theta}^n$.
     Then for any $\e>\e_n$ satisfying $\varphi_{\theta_0}(\e)\leq n\e^2$ and  $C_3= 16(\lambda+1)$ and $C_2=1$, the GP prior satisfies 
    \begin{align*}
        \Pi(A_n)  \geq \exp(-n C_2 \e_n^2) .
    \end{align*} 
\end{lemma}


Assumption~\ref{assump:Asf14} and~\ref{assump:Asf15} are straightforward to satisfy since the model risk function $R(a,\theta)$ is  bounded from above and below.

Now, suppose the  variational family $\sQ_{GP}$ is a class of Gaussian distributions on $\Theta$, defined as $\mathcal{N}(m_q,\mathcal{C}_q)$,
$m_q$ belongs to $\Theta$ and $\mathcal{C}_q$ is the covariance operator defined as $\mathcal{C}_q=\mathcal{C}^{1/2}(I-S)\mathcal{C}^{1/2} $, for any $S$ which is a symmetric and Hilbert-Schmidt (HS) operator on $\Theta$ (eigenvalues of HS operator are square summable). Note that $S$ and $m_q$ span the distributions in $\sQ_{GP}$.

The following lemma verifies Assumption~\ref{assump:Asf11}, for a specific sequence of distributions in $\sQ$. 
\begin{lemma}\label{lem:gp4}
    For a given $J\in\bbN$, let $\{Q_n\}$ be a sequence variational distribution such that $Q_n$ is the measure induced by a GP,  $W_Q(\cdot)= \theta_0^J(y) + \sum_{j=1}^{J} \sum_{k=1}^{2^{jd}} \zeta_j^2 Z_{j,k}  \vartheta_{j,k}(\cdot) $, where  $\theta_0^J(\cdot)= \sum_{j=1}^{J} \sum_{k=1}^{2^{jd}} \theta_{0;j,k}  \vartheta_{j,k}(\cdot) $ and $\zeta_j^2 = \frac{\mu_j^2}{1+n\e_n^2\tau_j^2}$.  Then for GP prior induced by $W=  \sum_{j=1}^{ J} \sum_{k=1}^{2^{jd}} \mu_j Z_{j,k} \vartheta_{j,k}$ and $\mu_j=2^{-jd/2-ja}$  for some $a>0$, $\|\theta_0\|_{\beta;\infty,\infty}<\infty$,  and $\theta_0^{J}(y)$ lie in the Cameron-Martin space $\text{Im}(\mathcal{C}^{1/2})$, we have 
    \[  \frac{1}{n}\scKL(\mathcal{N}(\bar \theta_0^J,\mathcal{C}_q)\|\mathcal{N}(0,\mathcal{C}))  + \frac{1}{n}\bbE_{Q_n} \scKL(P_0^n\| P_{\theta}^n) \leq C_9 \e_n^2,  \]
    where 
    \begin{align}~\label{eq:rate}
        \e_n= \begin{cases}
            n^{-\beta/(2\alpha+d)}\log n & if a\leq \beta \leq \alpha
            \\
            n^{-\alpha/(2\alpha+d)}\log n & if a \leq \alpha  \leq \beta
            \\
            n^{-a/(2a+d)}(\log n)^{d/(2a+d)} & if \alpha \leq a  \leq \beta
            \\
            n^{-\beta/(2a+d)}(\log n)^{d/(2a+d)} & if \alpha \leq \beta  \leq a.
            \end{cases}
      \end{align} 
  and $C_9:=  \max\left({\|\theta_0\|^2_{\beta,\infty,\infty}} ,   \frac{2^{-2a}-2^{-2Ja-2a} }{1-2^{-2a} }, 2^d/(2^d-1),C'\right)$,  where  $C'$ is a positive constant satisfying  $\| \theta_0(y)-\theta_0^J(y))\|^2_{\infty} \leq C' 2^{-2J\beta}$. 
\end{lemma}

Using the result above together with Proposition~\ref{prop:eta_gamma} implies that  the RSVB posterior converges at the same rate as the true posterior, where the convergence rate of the true posterior is derived in~\cite[Theorem 4.5]{Zanten08} for the binary GP classification problem with truncated wavelet GP prior.

Finally, we use the results above  to obtain bound on the optimality gap in values of  the binary GP classification problem.

\begin{theorem}\label{thrm:OGValue_GP}
    Fix $\gamma>0$ and for a given $J\in \bbN$. 
    For the binary GP classification problem with GP prior induced by $W=  \sum_{j=1}^{ J} \sum_{k=1}^{2^{jd}} \mu_j Z_{j,k} \vartheta_{j,k}$ and $\mu_j=2^{-jd/2-ja}$  for some $a>0$, $\|\theta_0\|_{\beta;\infty,\infty}<\infty$,  and $\theta_0^{J}(y)$ lie in the Cameron-Martin space $\text{Im}(\mathcal{C}^{1/2})$, the variational family $\sQ_{GP}$,  
    and for any $\tau>0$, 
    the $P_0^n-$ probability of the following event 
    \begin{align}
        \bigg\{ \nX:   R( \ars,\theta_0) - \inf_{z \in \mathcal{A}} R(z,\theta_0)  \leq 2\tau  \textcolor{black}{M'(\gamma)} \e_n \bigg\}
    \end{align}
    is at least $1-\tau^{-1}$ for sufficiently large $n$ and for some mapping $\textcolor{black}{M':\bbR^{+}\to\bbR^{+}}$, where $R(\cdot,\theta)$ is defined in~\eqref{eq:GPCloss} and $\e_n$ as derived in~\eqref{eq:rate}.
\end{theorem}
\begin{proof}
    The proof is a direct consequence of Theorem~\ref{thrm:OGValue}, Lemmas~\ref{lem:gp2},~\ref{lem:gp3},~\ref{lem:gp4}, and Proposition~\ref{prop:eta_gamma}.
\end{proof}

\section{Conclusion}
In this paper, we introduced a novel framework \textit{risk-sensitive variational Bayes} (RSVB), which can be used to extract computational methods for risk-sensitive approximate Bayesian inference. There are number of future directions that we are currently working on. 
Recall from the simulation result presented in Section~\ref{sec:news}, that the performance of the RSVB framework is affected by the choice of the risk-sensitivity parameter $\gamma$. However, the high-probability (large-sample) bound derived in~Section~\ref{sec:FSB} does not reflect this dependence explicitly. Consequently, digging out this dependence of the RSVB predictive performance on $\gamma$ is one of the future directions that we are currently working on. Although we briefly discuss in the introduction that the Bayesian Markowitz problem~\cite{Lai2011} is a special case of our RSVB framework, we do not discuss this as part of our examples. The main challenge in studying the Bayesian Markowitz problem under the RSVB framework is constructing a test function (Assumption~\ref{assump:Asf1}) with exponentially bounded errors for all $\e>\e_n$ (as required to establish convergence rates of the variational posteriors in Theorem~\ref{thrm:thm1} and also in~\cite[Theorem 2.1]{Zhang2018} and  in~\cite[Theorem 3.5]{Yang2020}). We are currently studying the Bayesian Markowitz problem for portfolio selection under the RSVB framework, as this is one of the important problems where the RSVB framework could be relevant. Moreover, as part of future work, we also aim to extend this work to incorporate models with local latent variables such as mixture models, where VB methods have been empirically demonstrated to outperform sampling-based posterior approximation methods. Additionally, we are also investigating risk-sensitive frameworks comprising other types of divergence measures.
\bibliographystyle{plain} 
\bibliography{refs}

\appendix

\newpage
\section{Proofs}

\subsection{Alternative derivation of LCVB}
We present the alternative derivation of LCVB. Consider the logarithm of the Bayes posterior risk,
	\begin{align}
	\nonumber
	\log \mathbb{E}_{\Pi(\theta|\nX)}[\exp(R(a,\theta))]&= \log \int_{\Theta}
	\exp(R(a,\theta))
	d\Pi(\theta|\nX)\\
	\nonumber
	&=  \log \int_{\Theta} \frac{dQ(\theta)}{dQ(\theta)} \exp(R(a,\theta))
	d\Pi(\theta| \nX) \\
	\label{eq:6}
	&\geq - \int_{\Theta} dQ(\theta)\log \frac{dQ(\theta)}{\exp(R(a,\theta))
		d\Pi(\theta| \nX)} =:  \mathcal{F}(a;Q(\cdot), \nX)
	\end{align}
where the inequality follows from an application of Jensen's inequality (since, without loss of generality, $\exp(R(a,\theta)) > 0$ for all $a \in \sA$ and $\theta \in \Theta$), and $Q \in \sQ$. Then, it follows that
\begin{align}
	\nonumber
	\min_{a \in \sA} \log \mathbb{E}_{\Pi(\theta|\nX)}[\exp(R(a,\theta))]
	&\geq  \underset{a \in \sA}{\min}~\underset{q \in \mathcal{Q}}{\max}\
	\mathcal{F}(a;Q(\theta),\nX)\\
	\label{eq:reg-form2}
	&= \underset{a\in\sA}{\min}~\underset{q \in \mathcal{Q}}{\max}
	- \scKL\left(Q(\theta)|| \Pi(\theta|\nX)\right) + \int_{\Theta} 
	R(a,\theta) dQ(\theta).
	\end{align}

\subsection{Proof of Theorem~\ref{thrm:thm1}:}
We prove our main result after series of important lemmas. For brevity we denote $\mathcal{LR}_n(\theta,\theta_0)=\frac{p(\nX|\theta)}{p(\nX|\theta_0)}$.
\begin{lemma}\label{lem:lem1}
	For any $a' \in \sA$, $\gamma>0$, and $\zeta>0$,
	\begin{align}
		\nonumber
		\bbE_{P^{n}_0}& \left[ \zeta  \int_{\Theta}  L_n(\theta,\theta_0)\ d\qrs \right]  
		\\
		\nonumber
		\leq & \log  \bbE_{P^{n}_0} \left[ \int_{\Theta}  e^{  \zeta  L_n(\theta,\theta_0)} \frac {e^{\gamma R(a',\theta)}  \ \mathcal{LR}_n(\theta,\theta_0) d\Pi(\theta) } {\int_{\Theta} e^{\gamma R(a',\theta)}  \ \mathcal{LR}_n(\theta,\theta_0) d\Pi(\theta) }   \right] +  \inf_{Q \in \sQ} \bbE_{P^{n}_0} \bigg[  \scKL(Q(\theta)\|\Pi(\theta|\nX))
		\\
		& - \gamma \inf_{a \in  \mathcal{A}} \bbE_Q[R(a,\theta)] \bigg]  +\log  \bbE_{P^{n}_0} \left[ \int_{\Theta} e^{\gamma R(a',\theta)} \ \frac{\mathcal{LR}_n(\theta,\theta_0) d\Pi(\theta)}{\int_{\Theta} \mathcal{LR}_n(\theta,\theta_0) d\Pi(\theta)}  \right] .
		\label{eq:eqf2} 
	\end{align} 
\end{lemma}
\begin{proof}
For any fixed $a' \in \sA, \gamma > 0$, and $\zeta>0$,
and using the fact that $\scKL$ is non-negative, observe that the  integral in the LHS of equation~\eqref{eq:eqf2} satisfies, 	
\begin{align}
	\nonumber
	\zeta \bbE_{\qrs}\left[  L_n(\theta,\theta_0) \right]  & \leq  \bbE_{\qrs} \left[ \log e^{ \zeta L_n(\theta,\theta_0)} \right] \ 
	\\
	\nonumber
	& \quad + \scKL\left( d\qrs \bigg \|\frac{e^{\zeta L_n(\theta,\theta_0)} e^{\gamma R(a',\theta)}  \ d\Pi(\theta|\nX)}{\int_{\Theta}  e^{\zeta L_n(\theta,\theta_0)} e^{\gamma R(a',\theta)}  \ d \Pi(\theta|\nX) } \right)
	\\
	\nonumber
	=  \bbE_{\qrs}&\left[  \log  e^{ \zeta L_n(\theta,\theta_0)} \right] \ + \log \bbE_{\Pi_n}\left[  e^{ \zeta L_n(\theta,\theta_0)} e^{\gamma R(a',\theta)} \right]
	\\
	\nonumber 
	& +  \bbE_{\qrs}\left[  \log \frac{d\qrs}{ e^{\zeta L_n  (\theta,\theta_0)}  e^{\gamma R(a',\theta)} \ d\Pi(\theta|\nX)  }  \right]
	\\
	\nonumber
	=  \log \bbE_{\Pi_n}&\left[  e^{ \zeta L_n  (\theta,\theta_0)} e^{\gamma R(a',\theta)}  \right]  +   \bbE_{\qrs}\left[  \log \frac{d\qrs}{e^{\gamma R(a',\theta)} \ d\Pi(\theta|\nX)  } \right].
\end{align}
Next, using the definition of $\qrs$ in the second term of last equality, for any other $Q(\cdot) \in \sQ$
\begin{align}
	\nonumber
	\zeta \bbE_{\qrs}\left[ L_n(\theta,\theta_0)\right]  \leq \log &\bbE_{\Pi_n}\left[ e^{\zeta L_n  (\theta,\theta_0)} e^{\gamma R(a',\theta)}  \right]  +  \bbE_{Q}\left[  \log \frac{dQ(\theta)} {e^{\gamma R(a',\theta)} \ d\Pi(\theta|\nX)  } \right].
\end{align}
Finally, it follows from the definition of the posterior distribution that
\begin{align} 
	\nonumber
	\zeta &\bbE_{\qrs}\left[ L_n(\theta,\theta_0)\right] 
	\\
	\nonumber 
	&\leq \log \int_{\Theta}   e^{\zeta L_n  (\theta,\theta_0)} e^{\gamma R(a',\theta)}  \ \frac {\mathcal{LR}_n(\theta,\theta_0) d\Pi(\theta) } {\int_{\Theta} \mathcal{LR}_n(\theta,\theta_0) d\Pi(\theta) }   + \bbE_{Q}\left[ \log \frac{dQ(\theta)} {e^{\gamma R(a',\theta)} \ d\Pi(\theta|\nX)  } \right] ,
	\\
	\nonumber
	=&  \log \int_{\Theta}  e^{\zeta L_n  (\theta,\theta_0)}   \ \frac { e^{\gamma R(a',\theta)} \mathcal{LR}_n(\theta,\theta_0) d\Pi(\theta) } {\int_{\Theta} e^{\gamma R(a',\theta)} \mathcal{LR}_n(\theta,\theta_0) d\Pi(\theta) }   +   \bbE_{Q}\left[  \log \frac{dQ(\theta)} {e^{\gamma R(a',\theta)} \ d\Pi(\theta|\nX) } \right]
	\\
	&+ \log \int_{\Theta}  \  e^{\gamma R(a',\theta)}  \frac{\mathcal{LR}_n(\theta,\theta_0) d\Pi(\theta)}{\int_{\Theta} \mathcal{LR}_n(\theta,\theta_0) d\Pi(\theta)}  , 
	\label{eq:eqf1} 
\end{align}
where the last equality follows from adding and subtracting $ \log   \bbE_{\Pi}\left[ e^{\gamma R(a',\theta)}  \mathcal{LR}_n(\theta,\theta_0) \right] $. Now taking expectation on either side of equation~\eqref{eq:eqf1} and using Jensen's inequality on the first and the last term in the RHS yields
\begin{align}
	\nonumber
	\bbE_{P^{n}_0}& \left[ \zeta \bbE_{\qrs}\left[ L_n(\theta,\theta_0)\right]\right]  
	\\
	\nonumber
	\leq  &\log  \bbE_{P^{n}_0} \left[ \int_{\Theta}  e^{\zeta L_n  (\theta,\theta_0)} \frac {e^{\gamma R(a',\theta)}  \ \mathcal{LR}_n(\theta,\theta_0) d\Pi(\theta) } {\int_{\Theta} e^{\gamma R(a',\theta)}  \ \mathcal{LR}_n(\theta,\theta_0) d\Pi(\theta) }   \right] +  \inf_{Q \in \sQ} \bbE_{P^{n}_0} \bigg[  \scKL(Q\|\Pi_n)  
	\\ 
	& - \gamma \inf_{a \in  \mathcal{A}}  \bbE_{Q}\left[ { R(a,\theta)}\right] \bigg] +\log  \bbE_{P^{n}_0} \left[ \int_{\Theta} e^{\gamma R(a',\theta)}  \ \frac{\mathcal{LR}_n(\theta,\theta_0) d\Pi(\theta)}{\int_{\Theta} \mathcal{LR}_n(\theta,\theta_0) d\Pi(\theta)} \right],
\end{align}
where in the second term in RHS of~\eqref{eq:eqf1}, we first take infimum over all $a\in \sA$ which upper  bounds the second term in~\eqref{eq:eqf1} and then take infimum over all $Q \in \sQ$, since the LHS does not depend on $Q$. 
\end{proof}

Next, we state a technical result that is important  in proving our next lemma.
\begin{lemma}[Lemma 6.4 of \cite{Zhang2018}]\label{lem:lem6.4}
	Suppose random variable X satisfies
	\vspace{0em}
	$$\mathbb{P}(X\geq t) \leq c_1 \exp(-c_2 t),$$
	for all $t \geq t_0>0$. Then for any $0< \beta \leq c_2/2 $,
	$$ \bbE[ \exp(\beta X)] \leq \exp(\beta t_0) + c_1.$$
\end{lemma}
\begin{proof}
    Refer Lemma 6.4 of \cite{Zhang2018}.
    
\end{proof}

In the following result, we bound the first term on the RHS of equation~\eqref{eq:eqf2}. The arguments in the proof are essentially similar to~Lemma 6.3 in~\cite{Zhang2018} 
\begin{lemma}\label{lem:lem2}
	Under Assumptions~\ref{assump:Asf1},~\ref{assump:Asf2},~\ref{assump:Asf3},~\ref{assump:Asf14}, and~\ref{assump:Asf15}  and for $\min(C,C_4(\gamma) + C_5(\gamma)) > C_2 + C_3 + C_4(\gamma) + 2$ and any $\e \geq \e_n$,
	\begin{align}
		\bbE_{P^{n}_0} \left[ \int_{\Theta}  e^{\zeta L_n  (\theta,\theta_0)} \frac {e^{\gamma R(a',\theta)}  \ \mathcal{LR}_n(\theta,\theta_0) d\Pi(\theta)} {\int_{\Theta} e^{\gamma R(a',\theta)}  \ \mathcal{LR}_n(\theta,\theta_0) d\Pi(\theta) }   \right]	\leq e^{\zeta C_{1}n \e^2} + (1+C_0+ 3W^{-\gamma}), \label{eq:eqf6}
	\end{align}
	for $0< \zeta \leq C_{10}/2 $, where $C_{10} = \min \{\lambda,C, 1\}/ C_1$ for any $\lambda>0$.
\end{lemma}
\begin{proof}
First define the set
\begin{align}
	B_n &: = \left\{\nX : \int_{\Theta}  \mathcal{LR}_n(\theta,\theta_0) d\Pi(\theta) \geq e^{-(1+C_3)n\e^2}\Pi(A_n) \right\} \label{def:def1},
\end{align}
where set $A_n$ is  defined in Assumption~\ref{assump:Asf3}.
We demonstrate that, under Assumption~\ref{assump:Asf3}, $P^n_0 \left({B_n^c} \right)$ is bounded above by an exponentially  decreasing(in $n$) term. Note that for $A_n$ as defined in Assumption~\ref{assump:Asf3}:
\begin{align}
	\nonumber
	\mathbb P^n_0&\left( \frac{1}{\Pi(A_n)} \int_{\Theta} \mathcal{LR}_n(\theta,\theta_0) d\Pi(\theta) \leq e^{-(1+C_3) n\e^2}   \right) 
	\\
	& \leq \mathbb P^n_0\left(\frac{1}{\Pi(A_n)} \int_{\Theta \cap A_n} \mathcal{LR}_n(\theta,\theta_0) d\Pi(\theta) \leq e^{-(1+C_3) n\e^2}   \right).  \label{eq:lc1}
\end{align}
Let $d \tilde \Pi(\theta) : = \frac{ \mathbbm{1}_{ \{\Theta \cap A_n \} }(\theta) }{\Pi(A_n)} d\Pi(\theta), $ and use this in~\eqref{eq:lc1} for any $\lambda>0$ to obtain, 
\begin{align*}
	\mathbb P^n_0&\left(  \frac{1}{\Pi(A_n)} \int_{\Theta} \mathcal{LR}_n(\theta,\theta_0) d\Pi(\theta) \leq e^{-(1+C_3) n\e^2} \right) 
	\\
	& \leq \mathbb P^n_0\left(  \int_{\Theta} \mathcal{LR}_n(\theta,\theta_0) d \tilde \Pi(\theta) \leq e^{-(1+C_3) n\e^2}  \right) 
	\\
	&= \mathbb P^n_0 \left( \left[ \int_{\Theta} \mathcal{LR}_n(\theta,\theta_0) d \tilde \Pi(\theta) \right]^{-\lambda} \geq e^{ (1 + C_3) \lambda n\e^2}   \right).
\end{align*}
{Then, using the  Chernoff's inequality in the last equality above, we have  }
\begin{align}
	\nonumber
	\mathbb P^n_0&\left(  \frac{1}{\Pi(A_n)} \int_{\Theta} \mathcal{LR}_n(\theta,\theta_0) d\Pi(\theta) 
	\leq e^{-(1+C_3) n\e^2} \right)
	\\
	\nonumber
	&\leq e^{- (1+C_3) \lambda n\e^2} \bbE_{P^{n}_0} \left( \left[ \int_{\Theta} \mathcal{LR}_n(\theta,\theta_0) d \tilde \Pi(\theta) \right]^{-\lambda}    \right)
	\\
	\nonumber
	&\leq e^{- (1+C_3) \lambda n\e^2}  \left[ \int_{\Theta} \bbE_{P^{n}_0} \left( \left[  \mathcal{LR}_n(\theta,\theta_0) \right]^{-\lambda} \right) d \tilde \Pi(\theta) \right]    
	\\
	\nonumber
	&=  e^{- (1+C_3) \lambda n\e^2}  \left[ \int_{\Theta} \exp( \lambda D_{\lambda+1}\left( P_0^n \| P_{\theta}^n \right) ) d \tilde \Pi(\theta) \right]    
	\\
	&\leq e^{- (1+C_3) \lambda n\e^2} e^{\lambda C_3 n\e_n^2 } \leq \e^{- \lambda n \e^2}, \label{eq:eq_bn} 
\end{align}
where the second inequality follows from first applying Jensen's inequality (on the term inside $[\cdot]$) and then using Fubini's theorem, and the penultimate inequality follows from Assumption~\ref{assump:Asf3} and the definition of $\tilde \Pi(\theta)$. 

Next, define the set $K_n : = \{ \theta \in \Theta: L_n(\theta,\theta_0) > C_1n\e^2  \} $. Notice that set $K_n$ is the set of alternate hypothesis as defined in Assumption~\ref{assump:Asf1}. We bound the calibrated posterior probability of this set $K_n$ to get a bound on the first term in the RHS of equation~\eqref{eq:eqf2}. Recall the sequence of test function $\{\phi_{n,\e}\}$ from Assumption~\ref{assump:Asf1}. Observe that
\begin{align}
	\nonumber
	\bbE_{P^{n}_0} &\left[  \frac {\int_{K_n} e^{\gamma R(a',\theta)}  \ \mathcal{LR}_n(\theta,\theta_0) d\Pi(\theta) } {\int_{\Theta} e^{\gamma R(a',\theta)}  \ \mathcal{LR}_n(\theta,\theta_0) d\Pi(\theta) }  \right] 
	\\
	\nonumber
	&= \bbE_{P^{n}_0} \left[ (\phi_{n,\e} + 1- \phi_{n,\e} )  \frac {\int_{K_n} e^{\gamma R(a',\theta)}  \ \mathcal{LR}_n(\theta,\theta_0) d\Pi(\theta) } {\int_{\Theta} e^{\gamma R(a',\theta)}  \ \mathcal{LR}_n(\theta,\theta_0) d\Pi(\theta) }  \right] 
	\\
	\nonumber
	& \leq \bbE_{P^{n}_0} [ \phi_{n,\e} ]
	+ \bbE_{P^{n}_0} \left[ (1 - \phi_{n,\e}) \mathbbm{1}_{B_n^C}  \right]
	\\
	\nonumber
	& + \bbE_{P^{n}_0} \left[ (1 - \phi_{n,\e}) \mathbbm{1}_{B_n} \frac {\int_{K_n} e^{\gamma R(a',\theta)}  \ \mathcal{LR}_n(\theta,\theta_0) d\Pi(\theta) } {\int_{\Theta} e^{\gamma R(a',\theta)}  \ \mathcal{LR}_n(\theta,\theta_0) d\Pi(\theta) }  \right]
	\\
	& \leq \bbE_{P^{n}_0}  \phi_{n,\e} 
	+ \bbE_{P^{n}_0} \left[ \mathbbm{1}_{B_n^C}  \right] 
	+ \bbE_{P^{n}_0} \left[ (1 - \phi_{n,\e}) \mathbbm{1}_{B_n} \frac {\int_{K_n} e^{\gamma R(a',\theta)}  \ \mathcal{LR}_n(\theta,\theta_0) d\Pi(\theta) } {\int_{\Theta} e^{\gamma R(a',\theta)}  \ \mathcal{LR}_n(\theta,\theta_0) d\Pi(\theta) }  \right],
	\label{eq:f3}
\end{align} 
where in the second inequality, we first divide the second term over set $B_n$ and its complement, and then use the fact that $  \frac {\int_{K_n} e^{\gamma R(a',\theta)}  \ \mathcal{LR}_n(\theta,\theta_0) d\Pi(\theta) } {\int_{\Theta} e^{\gamma R(a',\theta)}  \ \mathcal{LR}_n(\theta,\theta_0) d\Pi(\theta) }  \leq 1$. The third inequality is due the fact that $\phi_{n,\e} \in[0,1]$.
Next, using Assumption~\ref{assump:Asf3} and~\ref{assump:Asf15} observe that on set $B_n$ 
\begin{align*}
	\int_{\Theta} e^{\gamma R(a',\theta)}  \ \mathcal{LR}_n(\theta,\theta_0) d\Pi(\theta)  & \geq W^{\gamma} \int_{\Theta}  \ \mathcal{LR}_n(\theta,\theta_0) d\Pi(\theta)
	\\
	& \geq W^{\gamma} e^{-(1+ C_2 +C_3 )n\e_n^2} \geq W^{\gamma} e^{-(1+ C_2 +C_3 )n\e^2}.
\end{align*}
Substituting  the equation above in the third term of equation~\eqref{eq:f3}, we obtain
\begin{align}
	\nonumber
	&	\bbE_{P^{n}_0} \left[  (1  - \phi_{n,\e}) \mathbbm{1}_{B_n}  \frac {\int_{K_n} e^{\gamma R(a',\theta)}  \ \mathcal{LR}_n(\theta,\theta_0) d\Pi(\theta) } {\int_{\Theta} e^{\gamma R(a',\theta)}  \ \mathcal{LR}_n(\theta,\theta_0) d\Pi(\theta) }  \right] 
	\\
	\nonumber
	& \leq  W^{-{\gamma}} e^{(1+C_2 + C_3 )n\e^2} \bbE_{P^{n}_0} \left[ (1 - \phi_{n,\e}) \mathbbm{1}_{B_n} {\int_{K_n} e^{\gamma R(a',\theta)}  \ \mathcal{LR}_n(\theta,\theta_0) d\Pi(\theta) }  \right]
	\\
	& \leq  W^{-{\gamma}} e^{(1+C_2 + C_3 )n\e^2} \bbE_{P^{n}_0} \left[ (1 - \phi_{n,\e}) {\int_{K_n} e^{\gamma R(a',\theta)}  \ \mathcal{LR}_n(\theta,\theta_0) d\Pi(\theta) }  \right].
	\tag{$\star$}
\end{align} 
Now using Fubini's theorem observe that,
\begin{align*}
	(\star) & =  W^{-{\gamma}} e^{(1+C_2 + C_3 )n\e^2} {\int_{K_n} e^{\gamma R(a',\theta)}  \bbE_{P_\theta^n} \left[ (1 - \phi_{n,\e})  \right] d\Pi(\theta) } 
	\\
	\nonumber
	& \leq W^{-{\gamma}} e^{(1+C_2 + C_3 + C_4({\gamma}))n\e^2} \bigg[ {\int_{K_n \cap \{e^{\gamma R(a',\theta)} \leq e^{C_4({\gamma}) n\e^2} \} }  \bbE_{P_\theta^n} \left[ (1 - \phi_{n,\e})  \right] d\Pi(\theta) }
	\\
	\nonumber
	& \quad \quad \quad \quad \quad \quad \quad \quad + e^{-C_4({\gamma}) n\e^2} {\int_{K_n \cap \{e^{\gamma R(a',\theta)} > e^{C_4({\gamma}) n\e^2} \} } e^{\gamma R(a',\theta)}   d\Pi(\theta) } \bigg],
\end{align*}
where in the last inequality, we first divide the integral over set $\{\theta \in \Theta : e^{\gamma R(a',\theta)} \leq e^{C_4({\gamma})n\e^2} \}$ and its complement and then use the upper bound on $e^{\gamma R(a',\theta)}$ in the first integral. 
Now, it follows that 
\begin{align*}
	(\star) & \leq W^{-\gamma} e^{(1+C_2 + C_3 + C_4({\gamma}) )n\e^2} \Bigg[ {\int_{K_n }  \bbE_{P_\theta^n} \left[ (1 - \phi_{n,\e})  \right] d\Pi(\theta) } 
	\\
	+& e^{-C_4({\gamma}) n\e^2} {\int_{\{e^{\gamma R(a',\theta)} > e^{C_4({\gamma}) n\e^2} \} } e^{\gamma R(a',\theta)} d\Pi(\theta) } \Bigg]
	\\
	&= W^{-\gamma} e^{(1+C_2 + C_3 + C_4({\gamma}) )n\e^2} \bigg[ {\int_{K_n \cap \Theta_n(\e) }  \bbE_{P_\theta^n} \left[ (1 - \phi_{n,\e})  \right] d\Pi(\theta) }  
	\\  +& {\int_{K_n \cap \Theta_n(\e)^c }  \bbE_{P_\theta^n} \left[ (1 - \phi_{n,\e})  \right] d\Pi(\theta) }  +  e^{-C_4({\gamma}) n\e^2} {\int_{\{e^{\gamma R(a',\theta)} > e^{C_4({\gamma}) n\e^2} \} } e^{\gamma R(a',\theta)} d\Pi(\theta) } \bigg]
	\\
	& \leq W^{-\gamma} e^{(1+C_2 + C_3 + C_4({\gamma}) )n\e^2} \bigg[ {\int_{K_n \cap \Theta_n(\e) }  \bbE_{P_\theta^n} \left[ (1 - \phi_{n,\e})  \right] d\Pi(\theta) } + \Pi(\Theta_n(\e)^c) 
	\\ & \quad \quad \quad \quad \quad \quad \quad \quad +  e^{-C_4(\gamma)n\e^2} {\int_{\{e^{\gamma R(a',\theta)} > e^{C_4(\gamma)n\e^2} \} } e^{\gamma R(a',\theta)} d\Pi(\theta) } \bigg],
\end{align*}
where the second equality is obtained by  dividing  the first integral on set $\Theta_n(\e)$  and its complement, and the second inequality is due the fact that $\phi_{n,\e} \in[0,1]$. 
Now, using the equation above and Assumption~\ref{assump:Asf1},~\ref{assump:Asf2}, and~\ref{assump:Asf14} observe that 
\begin{align*}
	\bbE_{P^{n}_0}& \left[  (1  - \phi_{n,\e}) \mathbbm{1}_{B_n}  \frac {\int_{K_n} e^{\gamma R(a',\theta)}  \ \mathcal{LR}_n(\theta,\theta_0) d\Pi(\theta) } {\int_{\Theta} e^{\gamma R(a',\theta)}  \ \mathcal{LR}_n(\theta,\theta_0) d\Pi(\theta) }  \right] 
	\\
	&\leq W^{-\gamma} e^{(1+C_2 + C_3 + C_4(\gamma))n\e^2} \left[2 e^{-Cn \e^2} + e^{-(C_5(\gamma)+C_4(\gamma))n \e^2} \right].
\end{align*}
Hence, choosing $C,C_2,C_3, C_4(\gamma)$ and $C_5(\gamma)$ such that $ - 1 >  1 + C_2 + C_3 +C_4(\gamma) - \min(C,(C_4(\gamma) + C_5(\gamma)))$ implies

\begin{align}
	\bbE_{P^{n}_0} \left[  (1  - \phi_{n,\e}) \mathbb{I}_{B_n}  \frac {\int_{K_n} e^{\gamma R(a',\theta)}  \ \mathcal{LR}_n(\theta,\theta_0) d\Pi(\theta) } {\int_{\Theta} e^{\gamma R(a',\theta)}  \ \mathcal{LR}_n(\theta,\theta_0) d\Pi(\theta) }  \right] \leq 3 W^{-\gamma} e^{-n\e^2}. \label{eq:eqf4}
\end{align}
By Assumption~\ref{assump:Asf1}, we have
\vspace{0em}
\begin{align}
	\bbE_{P^{n}_0}  \phi_{n,\e} & \leq C_0e^{-C n \e^2}. \label{eq:eq_1b} 
\end{align}
Therefore, substituting equation~\eqref{eq:eq_bn}, equation~\eqref{eq:eqf4}, and~\eqref{eq:eq_1b} into~\eqref{eq:f3}, we obtain
\begin{align}
	\bbE_{P^{n}_0} \left[  \frac {\int_{K_n} e^{\gamma R(a',\theta)}  \ \mathcal{LR}_n(\theta,\theta_0) d\Pi(\theta) } {\int_{\Theta} e^{\gamma R(a',\theta)}  \ \mathcal{LR}_n(\theta,\theta_0) d\Pi(\theta) }  \right]  \leq (1+C_0+ 3W^{-\gamma}) e^{-C_{10} C_1 n \e^2}, \label{eq:eqf5}
\end{align}
where $C_{10} = \min \{\lambda,C, 1\}/ C_1$.
Using Fubini's theorem, observe that the LHS in the equation~\eqref{eq:eqf5} can be expressed as $\mu(K_n)$, where 
$$d\mu(\theta) = \bbE_{P^{n}_0} \left[\frac {   \ \mathcal{LR}_n(\theta,\theta_0)  } {\int_{\Theta} e^{\gamma R(a',\theta)}  \ \mathcal{LR}_n(\theta,\theta_0) d\Pi(\theta) } \right] \Pi(\theta) e^{\gamma R(a',\theta)} d\theta.$$
Next, recall that the set $K_n = \{ \theta \in \Theta: L_n(\theta,\theta_0) > C_1n\e^2  \}$. Applying Lemma~\ref{lem:lem6.4} above with $X=L_n(\theta,\theta_0)$, $c_1 = (1+C_0+ 3W^{-\gamma})$ , $c_2= C_{10}$ , $t_0 = C_1 n \e_n^2$,  and for $0< \zeta \leq C_{10}/2 $, we obtain
\begin{align}
	\bbE_{P^{n}_0} \left[ \int_{\Theta}  e^{\zeta L_n  (\theta,\theta_0)} \frac {e^{\gamma R(a',\theta)}  \ \mathcal{LR}_n(\theta,\theta_0) \Pi(\theta)} {\int_{\Theta} e^{\gamma R(a',\theta)}  \ \mathcal{LR}_n(\theta,\theta_0) d\Pi(\theta) }  d\theta \right]	\leq e^{\zeta C_{1}n \e_n^2} + (1+C_0+ 3W^{-\gamma}). 
\end{align}

\end{proof}
Further, we have another technical lemma, that will be crucial in proving the subsequent lemma that upper bounds the last term in the equation~\eqref{eq:eqf2}.
\begin{lemma}\label{lem:lem6.5}
	Suppose a positive random variable X satisfies
	\vspace{0em}
	$$\mathbb{P}(X\geq e^t) \leq c_1 \exp(-(c_2+1) t),$$
	for all $t \geq t_0>0$, $c_1>0 $, and $c_2>0$. Then,
	\vspace{0em}
	$$ \bbE[ X] \leq \exp(t_0) + \frac{c_1}{c_2}.$$
\end{lemma}
\begin{proof}
For any $Z_0 > 1$,
\vspace{0em}
\begin{align*}
	\bbE[ X] &\leq Z_0 + \int_{Z_0}^{\infty} \mathbb{P} (X \geq  x) dx 
	\\
	&= Z_0+ \int_{\ln Z_0}^{\infty} \mathbb{P} (X \geq  e^y) e^y dy  \leq Z_0 + c_1 \int_{\ln Z_0}^{\infty} \exp(-c_2y) dy.
\end{align*}
Therefore, choosing $Z_0 = \exp(t_0)$,
\vspace{0em}
\begin{align*}
	\bbE[ X] \leq \exp(t_0) + \frac{c_1}{c_2} \exp(-c_2 t_0) \leq \exp(t_0) + \frac{c_1}{c_2}.
\end{align*}

\end{proof}
Next, we establish the following bound on the last term in equation~\eqref{eq:eqf2}.
\begin{lemma}
	Under Assumptions~\ref{assump:Asf1},~\ref{assump:Asf2},~\ref{assump:Asf3},~\ref{assump:Asf14},~\ref{assump:Asf15},  and for $  C_5(\gamma) > C_2 + C_3 + 2$,
	\begin{align}
		\bbE_{P^{n}_0} \left[ \int_{\Theta}  \frac { e^{\gamma R(a',\theta)}  \ \mathcal{LR}_n(\theta,\theta_0) d\Pi(\theta)} {\int_{\Theta} \ \mathcal{LR}_n(\theta,\theta_0) d\Pi(\theta) }   \right]	\leq e^{C_{4}(\gamma)n \e_n^2} + 2C_4(\gamma).
		\label{eq:eqf12}
	\end{align}
	for any $\lambda\geq 1+C_4(\gamma)$.
\end{lemma}
\begin{proof}
Define the  set 
\begin{align}
	M_n &: = \{ \theta \in \Theta : e^{\gamma R(a',\theta)} > e^{C_4(\gamma) n\e^2}  \}. 
\end{align}
Using the set $B_n$ in equation~\eqref{def:def1}, observe that the measure of the set $M_n$, under the posterior distribution satisfies, 
\begin{align}
	\bbE_{P^{n}_0} \left[  \frac {\int_{M_n} \ \mathcal{LR}_n(\theta,\theta_0) d\Pi(\theta) } {\int_{\Theta}  \ \mathcal{LR}_n(\theta,\theta_0) d\Pi(\theta) }  \right] & \leq \bbE_{P^{n}_0} \left[ \mathbbm{1}_{B_n^c} \right]
	+ \bbE_{P^{n}_0} \left[ \mathbbm{1}_{B_n} \frac {\int_{M_n}  \ \mathcal{LR}_n(\theta,\theta_0) d\Pi(\theta) } {\int_{\Theta}  \ \mathcal{LR}_n(\theta,\theta_0) d\Pi(\theta) }  \right].
	\label{eq:f9}
\end{align}
Now, the second term of equation~\eqref{eq:f9} can be  bounded as follows: recall Assumption~\ref{assump:Asf3} and the definition of set $B_n$, both together imply that,
\begin{align*}
	\bbE_{P^{n}_0} \left[  \mathbbm{1}_{B_n}  \frac {\int_{M_n}  \ \mathcal{LR}_n(\theta,\theta_0) d\Pi(\theta) } {\int_{\Theta}   \ \mathcal{LR}_n(\theta,\theta_0) d\Pi(\theta) }  \right] & \leq  e^{(1+C_2 + C_3 )n\e^2} \bbE_{P^{n}_0} \left[  \mathbbm{1}_{B_n} {\int_{M_n} \ \mathcal{LR}_n(\theta,\theta_0) d\Pi(\theta) }  \right]
	\\
	& \leq e^{(1+C_2 + C_3 )n\e^2} \bbE_{P^{n}_0} \left[  {\int_{M_n}  \ \mathcal{LR}_n(\theta,\theta_0) d\Pi(\theta) }  \right].
	\tag{$\star \star$}
\end{align*}
Then, using Fubini's Theorem $(\star \star ) =   e^{(1+C_2 + C_3 )n\e^2} \Pi(M_n)$. Next, using the definition of set $M_n$ and  then Assumption~\ref{assump:Asf14}, we obtain
\begin{align*}
	\bbE_{P^{n}_0} \left[  \mathbbm{1}_{B_n}  \frac {\int_{M_n}  \ \mathcal{LR}_n(\theta,\theta_0) d\Pi(\theta) } {\int_{\Theta}   \ \mathcal{LR}_n(\theta,\theta_0) d\Pi(\theta) }  \right] & \leq e^{(1+C_2 + C_3 )n\e^2} e^{-C_4(\gamma)n\e^2} \int_{M_n} e^{\gamma R(a',\theta)}d\Pi(\theta)
    \\
	& \leq e^{(1+C_2 + C_3 )n\e^2} e^{-C_4(\gamma)n\e^2} e^{-C_5(\gamma)n\e^2},
\end{align*}
Hence, choosing the constants $C_2,C_3, C_4(\gamma)$ and $C_5(\gamma)$ such that  $ - 1 >  1 + C_2 + C_3  - C_5(\gamma)$ implies
\begin{align}
	\bbE_{P^{n}_0} \left[ \mathbbm{1}_{B_n}  \frac {\int_{M_n}  \ \mathcal{LR}_n(\theta,\theta_0) d\Pi(\theta) } {\int_{\Theta}  \ \mathcal{LR}_n(\theta,\theta_0) d\Pi(\theta) }  \right] \leq e^{-(1+C_4(\gamma))n\e^2} \label{eq:eqf10}
\end{align}
Therefore, substituting~\eqref{eq:eq_bn} and~\eqref{eq:eqf10} into~\eqref{eq:f9}
\begin{align}
	\bbE_{P^{n}_0} \left[  \frac {\int_{M_n}  \ \mathcal{LR}_n(\theta,\theta_0) d\Pi(\theta) } {\int_{\Theta}  \ \mathcal{LR}_n(\theta,\theta_0) d\Pi(\theta) }  \right]  \leq 2 e^{-C_4(\gamma) (C_{11}(\gamma)+1)n \e^2}, \label{eq:eqf11}
\end{align}
where $C_{11}(\gamma) = \min \{\lambda,1+C_4(\gamma)\}/ C_4(\gamma)-1$.
Using Fubini's theorem, observe that the RHS in~\eqref{eq:eqf11} can be expressed as $\nu(M_n)$, where the measure $$d\nu(\theta) = \bbE_{P^{n}_0} \left[\frac {   \ \mathcal{LR}_n(\theta,\theta_0)  } {\int_{\Theta}   \ \mathcal{LR}_n(\theta,\theta_0) d\Pi(\theta) } \right] d \Pi(\theta).$$ 
Applying Lemma~\ref{lem:lem6.5} for $X=e^{\gamma R(a',\theta)}$,$c_1 = 2$ , $c_2= C_{11}(\gamma)$ , $t_0 = C_4(\gamma) n \e_n^2$ and $\lambda\geq 1+C_4(\gamma)$, we obtain
\begin{align}
	\bbE_{P^{n}_0} \left[ \int_{\Theta}  \frac { e^{\gamma R(a',\theta)}  \ \mathcal{LR}_n(\theta,\theta_0) d\Pi(\theta)} {\int_{\Theta} \ \mathcal{LR}_n(\theta,\theta_0) d\Pi(\theta) }   \right]	\leq e^{C_{4}(\gamma)n \e_n^2} + \frac{2}{C_{11}(\gamma)} \leq e^{C_{4}(\gamma)n \e_n^2} + 2C_4(\gamma) .
\end{align}

\end{proof}

\begin{proof}{Proof of Theorem~\ref{thrm:thm1}:}
Finally, recall~\eqref{eq:eqf2},
\begin{align}
	\nonumber
	\zeta \bbE_{P^{n}_0}& \left[ \int_{\Theta}  L_n(\theta,\theta_0)\ d\qrs \right]  
	\\ 
	\nonumber
	\leq & \log  \bbE_{P^{n}_0} \left[ \int_{\Theta}  e^{\zeta L_n  (\theta,\theta_0)} \frac {e^{\gamma R(a',\theta)}  \ \mathcal{LR}_n(\theta,\theta_0) d\Pi(\theta) } {\int_{\Theta} e^{\gamma R(a',\theta)}  \ \mathcal{LR}_n(\theta,\theta_0) d\Pi(\theta) }   \right] +  \inf_{Q \in \sQ} \bbE_{P^{n}_0} \bigg[  \scKL(Q\|\Pi_n) 
	\\
	& - \gamma \inf_{a \in  \mathcal{A}} \bbE_Q[R(a,\theta)] \bigg] +\log  \bbE_{P^{n}_0} \left[ \int_{\Theta} e^{\gamma R(a',\theta)}  \ \frac{\mathcal{LR}_n(\theta,\theta_0) d\Pi(\theta)}{\int_{\Theta} \mathcal{LR}_n(\theta,\theta_0) d\Pi(\theta)}  \right] . 
\end{align}
Substituting~\eqref{eq:eqf12} and ~\eqref{eq:eqf6} into the equation above and then using the  definition of $\eta_n^R(\gamma)$, we get
\begin{align}
	\nonumber
	\bbE_{P^{n}_0}& \left[ \int_{\Theta}  L_n(\theta,\theta_0)\ d\qrs \right]  \\
	\nonumber
	 \leq& \frac{1}{\zeta} \left\{ \log(e^{\zeta C_{1}n \e_n^2} + (1+C_0+ 3W^{-\gamma}))  + \log \left( e^{C_{4}(\gamma) n \e_n^2} + 2C_4(\gamma) \right)  + n \eta_n^R(\gamma)  \right\}
	\\
	\nonumber
	 \leq& \left( C_{1}+ \frac{1}{\zeta} C_4(\gamma) \right) n \e_n^2  + \frac{1}{\zeta}n \eta_n^R(\gamma)  +\frac{(1+C_0+ 3W^{-\gamma})e^{(-\zeta C_{1}n \e_n^2)} +2C_4(\gamma)e^{ - C_{4}(\gamma)n \e_n^2}  }{\zeta}  ,
\end{align}
where the last inequality uses the fact that $\log x \leq x -1$. Choosing $ \zeta= C_{10}/2= \frac{\min(C,\lambda,1)}{2C_1}$, 
\begin{align}
	\nonumber
	\bbE_{P^{n}_0}& \left[ \int_{\Theta}  L_n(\theta,\theta_0)\ d\qrs \right]  \\
	& \leq M(\gamma) n (\e_n^2)  + M' n \eta_n^R(\gamma)  +\frac{2 (1+C_0+ 3W^{-\gamma})e^{(- \frac{C_{10}}{2}n \e_n^2)} + 4C_{4}(\gamma) e^{ - C_{4}(\gamma)n \e_n^2}  }{C_{10}} 
	\label{eq:lc21}
\end{align}
where $M(\gamma)= C_{1}+ \frac{1}{\zeta} C_4(\gamma)$ and $M'=\frac{1}{\zeta}$ depend on $C,C_1, C_4(\gamma) ,W\text{ and } \lambda$. Since the last two terms in~\eqref{eq:lc21} decrease and the first term increases as $n$ increases, we can choose $M'$ large enough, 
such that for all $n \geq 1$
\[ M' n \eta_n^R(\gamma)  > \frac{2 (1+C_0+ 3W^{-\gamma})  }{C_{10}} + \frac{4C_{4}(\gamma)}{C_{10}}, \]
and therefore for $M=2M'$,
\begin{align}
	\bbE_{P^{n}_0} \left[ \int_{\Theta}  L_n(\theta,\theta_0)\ d\qrs \right] \leq M(\gamma) n (\e_n^2)  + M n \eta_n^R(\gamma). \label{eq:lc3}
\end{align}

Also, observe that  the LHS in  the above equation is always  positive, therefore $M(\gamma)  \e_n^2  + M  \eta_n^R(\gamma)\geq 0 \ \forall n \geq 1$ and $\gamma>0$.

\end{proof}



\subsection{Proof of Theorem~\ref{thrm:OGValue} and~\ref{thrm:OGdecision} }

\begin{lemma}~\label{lem:Lemma3}
    Given $a' \in \mathcal{A}$ and for a constant M, as defined in Theorem~\ref{thrm:thm1} 
    \begin{align}
        \bbE_{P^{n}_0} \left[ \sup_{a \in \mathcal{A}} \left \vert \bbE_{\qrs}[R(a,\theta)] - R(a,\theta_0) \right \vert \right]  \leq  \left [M(\gamma) \e_n^2  + M  \eta_n^R(\gamma)\right ]^{\frac{1}{2}}.
    \end{align}
\end{lemma}

\begin{proof}
    First, observe that
    \begin{align*}
        \left( \sup_{a \in \mathcal{A}} \left\vert \bbE_{\qrs}[R(a,\theta)] - R(a,\theta) \right\vert \right)^2 & \leq 
        \left( \bbE_{\qrs}\left[ \sup_{a \in \mathcal{A}} \vert R(a,\theta) - R(a,\theta_0)\vert  \right] \right)^2
        \\
        	\leq & \bbE_{\qrs}\left[ \left(\sup_{a \in \mathcal{A}} \vert R(a,\theta) - R(a,\theta_0) \vert \right)^2 \right],
    \end{align*}
    where the last inequality follows from Jensen's inequality. Now, using the Jensen's inequality again 
    \begin{align*}
        & \left ( \bbE_{P^{n}_0} \left[ \sup_{a \in \mathcal{A}} \left\vert \bbE_{\qrs}[R(a,\theta)] - R(a,\theta_0) \right\vert \right] \right)^2 
        \\ 
        & \quad \quad \leq  \bbE_{P^{n}_0} \left[  \left( \sup_{a \in \mathcal{A}} \left\vert \bbE_{\qrs}[R(a,\theta)] - R(a,\theta_0) \right\vert \right)^2 \right].
    \end{align*}
    Now, using Theorem~\ref{thrm:thm1} the result follows immediately.
    
\end{proof}
\begin{proof}[Proof of Theorem~\ref{thrm:OGValue}]
    Observe that 
    \begin{align}
        \nonumber
        &R( \ars,\theta_0) - \inf_{z \in \mathcal{A}} R(z,\theta_0)
        \\
        \nonumber
        & = \vert  R( \ars,\theta_0) - \inf_{z \in \mathcal{A}} R(z,\theta_0)  \vert 
        \\
        \nonumber
        & =  R( \ars,\theta_0) - \bbE_{Q^*_{\ars,\gamma}(\theta|\nX)}[  R(\ars,\theta)  ]  +   \bbE_{Q^*_{\ars,\gamma}(\theta|\nX)}[  R(\ars,\theta)  ]  - \inf_{z \in \mathcal{A}} R(z,\theta_0) 
        \\
        \nonumber
        & \leq  \left \vert  R( \ars,\theta_0) - \bbE_{Q^*_{\ars,\gamma}(\theta|\nX)}[  R(\ars,\theta)  ]   \right \vert + \left \vert  \bbE_{Q^*_{\ars,\gamma}(\theta|\nX)}[  R(\ars,\theta)  ]  - \inf_{a \in \mathcal{A}} R(a,\theta_0)  \right \vert 
        \\
        & \leq 2 \sup_{a \in \mathcal{A}} \left\vert \int  R(a,\theta) dQ^*_{\ars,\gamma}(\theta|\nX) - R(a,\theta_0) \right\vert. 
        \label{eq:eqt21b}
    \end{align}
    
    Given $\ars \in \mathcal{A}$ and for a constant M (defined in Theorem~\ref{thrm:thm1}), we have from Lemma~\ref{lem:Lemma3} for $a'=\ars$
    \begin{align}
        \bbE_{P^{n}_0} \left[ \sup_{a \in \mathcal{A}} \left \vert \int  R(a,\theta) dQ^*_{\ars,\gamma}(\theta|\nX) - R(a,\theta_0) \right \vert \right]  \leq  \left [ M(\gamma) \e_n^2  + M  \eta_n^R(\gamma)\right ]^{\frac{1}{2}}.
        \label{eq:OG2}
    \end{align}
    
    It follows from above that the $P_0^n-$ probability of the following event is at least $ 1- \tau^{-1}$:
    \begin{align}
        \bigg\{ \nX:   R( \ars,\theta_0) - \inf_{z \in \mathcal{A}} R(z,\theta_0)  \leq 2\tau  \left[ M(\gamma) \e_n^2  + M  \eta_n^R(\gamma) \right]^{\frac{1}{2}}  \bigg\}.
    \end{align}
\end{proof}

\begin{proof}[Proof of Theorem~\ref{thrm:OGdecision}:]
    Since, the result in Lemma~\ref{lem:Lemma3} holds for any $a' \in \sA$, we fix $a' =\ars $. Now observe  that for any $\gamma>0$ and $\tau>0$, the result in Lemma~\ref{lem:Lemma3} implies that the $P_0^n-$ probability of 
    \begin{align}
        \left\{ \left[ M (\e_n^2 + \eta_n^R(\gamma)) \right]^{-\frac{1}{2}} \sup_{a \in \mathcal{A}} \left\vert \int   R(a,\theta) dQ^*_{\ars,\gamma}(\theta|\nX) -R(a,\theta_0) \right\vert > \tau \right\}
    \end{align}
    is at most $\tau^{-1}$. For $\ars$, it follows from the definition of $ \Psi(\cdot)$ that
    \begin{align}
        \nonumber
        \Psi&\left( H(\ars, \underset{a \in \mathcal{A} }{\arg \min}   ~ R(a,\theta_0) ) \right) 
        \\
        \nonumber
        &\leq R(\ars,\theta_0) - \inf_{z \in \mathcal{A}} R(a,\theta_0) 
        \\
        \nonumber
        & =  R(\ars,\theta_0) - \bbE_{Q^*_{\ars,\gamma}(\theta|\nX)}[  R(\ars,\theta)  ]   +   \bbE_{Q^*_{\ars,\gamma}(\theta|\nX)}[  R(\ars,\theta)  ]  - \inf_{z \in \mathcal{A}} R(z,\theta_0)
        \\
        \nonumber
        & \leq  \left \vert  R(\ars,\theta_0) - \bbE_{Q^*_{\ars,\gamma}(\theta|\nX)}[  R(\ars,\theta)  ]  \right \vert + \left \vert  \bbE_{Q^*_{\ars,\gamma}(\theta|\nX)}[  R(\ars,\theta)  ] - \inf_{a \in \mathcal{A}} R(a,\theta_0)  \right \vert 
        \\
        & \leq 2 \sup_{a \in \mathcal{A}} \left\vert \int  R(a,\theta) dQ^*_{\ars,\gamma}(\theta|\nX) - R(a,\theta_0) \right\vert. 
        \label{eq:eqt22}
    \end{align}
    It follows from the inequality above that
    \begin{align}
        \nonumber
        \bigg\{ &\left[ M (\e_n^2 + \eta_n^R(\gamma)) \right]^{-\frac{1}{2}} \Psi\bigg( H(\ars, \underset{z \in \mathcal{A} }{\arg \min}   ~ R(a,\theta_0) ) \bigg) > 2 \tau \bigg\}
        \\
        & \subseteq \left\{ \left[ M (\e_n^2 + \eta_n^R(\gamma)) \right]^{-\frac{1}{2}} \sup_{a \in \mathcal{A}} \left\vert \int  R(a,\theta) dQ^*_{\ars,\gamma}(\theta|\nX) - R(a,\theta_0) \right\vert > \tau \right\} .
        \label{eq:eq32}
    \end{align} 
    Therefore, using the condition on the growth function in the statement of the theorem that, $\frac{ \Psi\left(H\left(\ars, \underset{a \in \mathcal{A} }{\arg \min}   ~ R(a,\theta_0) \right)\right)}{H\left(\ars, \underset{a \in \mathcal{A} }{\arg \min}   ~ R(a,\theta_0) \right)^{\delta}} =  \kappa,$ the $P_0^n-$ probability of the following event is at least $ 1- \tau^{-1}$:
    \begin{align}
        \bigg\{    H(\ars, \underset{a \in \mathcal{A} }{\arg \min}   ~ R(a,\theta_0) )  \leq \tau^{\frac{1}{\delta}} \left[ \frac{2 \left[ M(\gamma) \e_n^2  + M  \eta_n^R(\gamma) \right]^{\frac{1}{2}} }{ \kappa } \right]^{\frac{1}{\delta}} \bigg\}.
    \end{align}
    This concludes the proof.

\end{proof}

\subsection{Proofs in Section~\ref{sbsec:PropEta}}
\begin{proof}[Proof of Proposition~\ref{prop:eta_n}]

Using the definition of  $\eta_n^R( \gamma)$ and the posterior distribution $\Pi(\theta|\nX)$, observe that
\begin{align}
	\nonumber
	n\eta_n^R(\gamma) &= \inf_{Q \in \sQ} \bbE_{P^{n}_0} \left[  \scKL(Q\|\Pi_n) 
	- \gamma  \inf_{a \in \sA}  \bbE_Q[R(a,\theta)] \right]
	\\
	\nonumber
	= \inf_{Q \in \sQ} &\bbE_{P^{n}_0} \left[  \scKL(Q\|\Pi) + \int_{\Theta}   	dQ(\theta)  \log \left( \frac{\int d\Pi(\theta)p(\nX|\theta) }{p(\nX|\theta)} \right) 
	- \gamma  \inf_{a \in \sA}  \bbE_Q[R(a,\theta)] \right]
	\\
	\nonumber
	= \inf_{Q \in \sQ}& \left[ \scKL(Q\|\Pi ) - \gamma  \inf_{a \in \sA}  \bbE_Q[R(a,\theta)] + \bbE_{P^{n}_0} \left[  \bbE_{Q}\left[  \log \left( \frac{\int d\Pi(\theta)p(\nX|\theta) }{p(\nX|\theta)} \right) \right] \right]\right] .
\end{align} 
Now, using  Fubini's in the last term of the equation above, we obtain
\begin{align}
	\nonumber
	n\eta_n^R(\gamma) &=  \inf_{Q \in \sQ} \Bigg[ \scKL(Q(\theta)\|\Pi(\theta) ) - \gamma  \inf_{a \in \sA}  \bbE_Q[R(a,\theta)] 
	\\
	& + \bbE_Q \left[ \scKL \left(dP^{n}_0\| p(\nX|\theta) \right) - \scKL\left( dP^{n}_0 \bigg\| \int d\Pi(\theta)p(\nX|\theta) \right)  \right]\Bigg] .
\end{align} 
Observe that, $\int_{\sX^n}\int d\Pi(\theta)p(\nX|\theta) d\nX = 1$. Since, \scKL \ is always non-negative, it follows from the equation above that
\begin{align}
	\nonumber
	&\eta_n^R(\gamma) 
	\\
	\nonumber
	&\leq \frac{1}{n} \inf_{Q \in \sQ} \left[ \scKL\left(Q(\theta)\|\Pi(\theta) \right) - \gamma  \inf_{a \in \sA}  \bbE_Q[R(a,\theta)] + \bbE_Q \left[ \scKL\left(dP^{n}_0\| p(\nX|\theta) \right) \right]  \right]
	\\
	&\leq \frac{1}{n} \inf_{Q \in \sQ} \left[ \scKL\left(Q(\theta)\|\Pi(\theta) \right) + \bbE_Q \left[ \scKL\left(dP^{n}_0\| p(\nX|\theta) \right) \right]  \right] - \frac{\gamma}{n}  \inf_{Q \in \sQ} \inf_{a \in \sA}  \bbE_Q[R(a,\theta)],
	\label{eq:eq1}
\end{align}
where the last inequality follows from the following fact, for any functions $f(\cdot)$ and
\vspace{0em} $g(\cdot)$, $$\inf (f -g) \leq \inf f - \inf g.$$
Recall $\e_n'\geq \frac{1}{\sqrt{n}}$. 
Now, using Assumption~\ref{assump:Asf11}, it is straightforward to observe that the first term in~\eqref{eq:eq1}, 
\begin{align}
	\frac{1}{n} \inf_{Q \in \sQ} \left[ \scKL\left(Q(\theta)\|\Pi(\theta) \right) + \bbE_Q \left[ \scKL\left(dP^{n}_0\| p(\nX|\theta) \right) \right]  \right] \leq C_9 \e_n'^2. 
	\label{eq:eq1b}
\end{align}
Now consider the last term in~\eqref{eq:eq1}. Notice that the coefficient of $\frac{1}{n}$ is independent of $n$ and is bounded from below. Therefore, there exist a constant $C_8=-  \inf_{Q \in \sQ} \inf_{a \in \sA}  \bbE_Q[R(a,\theta)] $, 
such that with equation~\eqref{eq:eq1b} it follows  that $\eta_n^R(\gamma) \leq \gamma  n^{-1}C_8 +C_9\e_n'^2$ and the result follows.

\end{proof}

\begin{proof}[{Proof of Proposition~\ref{prop:eta_gamma}}]
    
    First recall that
    \begin{align} 
        \nonumber
        n\eta_n^R(\gamma) &= \inf_{Q \in \sQ} \bbE_{P^{n}_0} \left[  \scKL(Q(\theta)\|\Pi(\theta|\nX)) 
        - \gamma \inf_{a \in  \mathcal{A}} \bbE_Q[R(a,\theta)] \right]
        \\
        &= \inf_{Q \in \sQ} \bbE_{P^{n}_0} \left[  \scKL(Q(\theta)\|\Pi(\theta|\nX))  \right]
        - \gamma \inf_{a \in  \mathcal{A}} \bbE_Q[R(a,\theta)]. 
        \label{eq:eqeta}
    \end{align}
    
    Observe that the optimization problem is equivalent to solving :
    \begin{align} 
        \min_{Q \in \sQ} \bbE_{P^{n}_0} \left[  \scKL(Q(\theta)\|\Pi(\theta|\nX))  \right] \text{ s.t. } 
        - \inf_{a \in  \mathcal{A}} \bbE_Q[R(a,\theta)] \leq 0 .
    \end{align}
    
    Now for any $\gamma>0$, $Q^*_{\gamma}(\theta) \in \sQ$ that minimizes the objective in~\eqref{eq:eqeta} is primal feasible if
    \[- \inf_{a \in  \mathcal{A}} \int_{\Theta}  dQ^*_{\gamma}(\theta) R(a,\theta)	  \leq 0. \]

    Therefore, it is straightforward to  observe that 
    as $\gamma$ increases $n\eta_n^R(\gamma)$ decreases that is
    \[ \bbE_{P^{n}_0} \left[  \int_{\Theta}   	dQ^*_{\gamma}(\theta)  \log \frac{dQ^*_{\gamma}(\theta)} { d\Pi(\theta|\nX)  }   
    - \gamma \inf_{a \in  \mathcal{A}} \int_{\Theta}  dQ^*_{\gamma}(\theta) R(a,\theta)	  \right].\]
    
\end{proof} 

    


\subsection{Sufficient conditions on $R(a,\theta)$ for existence of tests }\label{sbsec:ConditionsL}
\textcolor{black}{To show the existence of test functions, as required in Assumption~\ref{assump:Asf1}, we will use the following result from ~\cite[Theorem 7.1]{GGV}}, that is applicable only to distance measures that are bounded above by the Hellinger distance.
\begin{lemma}[Theorem 7.1 of \citep{GGV}]\label{lem:ggv71}
	Suppose that for some non-increasing function $D(\e)$, some $\e_n>0$  and for every $\e > \e_n$,
	\[ N\left( \frac{\e}{2}, \left\{ P_{\theta}: \e \leq m(\theta,\theta_0) \leq 2\e \right\} ,m \right) \leq D(\e), \]
	where $m(\cdot,\cdot)$ is any distance measure bounded above by Hellinger distance.	Then for every $\e> \e_n$, there exists a test $\phi_n$ (depending on $\e>0$) such that, for every $j\geq 1$,
	\begin{align*}
		\bbE_{P^{n}_0}[ \phi_n ]  &\leq D(\e) \exp \left(-\frac{1}{2}n \e^2 \right)\frac{1}{1-\exp \left (-\frac{1}{2} n\e^2 \right)}, \text{and}\\
		\underset{\{ \theta\in \Theta_n(\e)  : m(\theta,\theta_0) > j \e  \}}{\sup} \bbE_{P^{n}_{\theta}}[ 1- \phi_n ] &\leq  \exp\left(- \frac{1}{2} n \e^2 j\right).
	\end{align*} 
\end{lemma}
\begin{proof}[Proof of Lemma~\ref{lem:ggv71}:]
    Refer Theorem 7.1 of \cite{GGV}.
\end{proof}

For the remaining part of this subsection we assume that $\Theta \subseteq \bbR^d$. In the subsequent paragraph, we state further assumptions on the risk function to show $L_n(\cdot,\cdot)$ as defined in~\eqref{eq:lossr} satisfies Assumption~\ref{assump:Asf1}. For brevity we denote $n^{-1/2} \sqrt{L_n(\theta,\theta_0)}$ by $d_L(\theta,\theta_0)$, that is
\vspace{0em}
\begin{equation}
d_L(\theta_1,\theta_2) := \sup_{a \in \mathcal{A}} \vert R(a,\theta_1) - R(a,\theta_2) \vert,~ \forall\{\theta_1,\theta_2\}\in\Theta
\label{eq:eqm1}
\end{equation}
and the covering number of the set $T(\e):= \{P_{\theta} : d_L(\theta,\theta_0) < \e\}$ as  $N(\delta, T(\e), d_L )$, where $\delta > 0 $ is the radius of each ball in the cover. 
We assume that the risk function $R(a,\cdot)$ satisfies the following bound. 
\vspace{0em}
\begin{assumption}\label{assump:Asf5} 
	The model risk satisfies 
	$$d_L(\theta_1,\theta_2) \vert  \leq K_1 d_{H}(\theta,\theta_0), $$
	where $d_{H}(\theta_1,\theta_2) $  is the Hellinger distance between two models $P_{\theta_1}$ and $P_{\theta_2}$. 
\end{assumption}
For instance, suppose the definition of model risk is $R(a,\theta)= \int_{\sX}\ell(x,a)p(y|\theta)dx$, where $\ell(x,a)$ is an underlying loss function. Then, observe that Assumption~\ref{assump:Asf5} is trivially satisfied if 
$\ell(x,a)$ is bounded in $x$ for a given $a \in \mathcal{A}$ and $\mathcal{A}$ is compact, since $d_L(\theta_1,\theta_2)$ can be bounded by the total variation distance $d_{TV}(\theta_1,\theta_2) = \frac{1}{2} \int \left  \vert   dP_{\theta_1}(x) - dP_{\theta_2}(x) \right \vert $ and total variation distance is bounded above by the Hellinger distance~\citep{gibbs2002choosing}. Under the  assumption above it also follows that we can apply Lemma~\ref{lem:ggv71} to the metric $d_L(\cdot,\cdot)$ defined in~\eqref{eq:eqm1}. Now, we will also assume an additional regularity condition on the risk function.

\begin{assumption}\label{assump:Asf4} 
	For every $\{\theta_1,\theta_2\} \in \Theta$, there exists a constant $K_2>0$ such that
    $$d_L(\theta_1,\theta_2) \leq K_2 \| \theta_1 -\theta_2 \|, $$
\end{assumption}

We can now show that the covering number of the set $T(\e)$ satisfies 
\begin{lemma}\label{lem:cover}
	Given $\e>\delta>0$, and under Assumption~\ref{assump:Asf4},
	\begin{align}
		N(\delta, T(\e), d_L ) < \left( \frac{2 \e}{\delta} + 2 \right)^d.
		\label{eq:eqc1}
	\end{align}
\end{lemma}

\begin{proof}[Proof of Lemma~\ref{lem:cover}:]	
    For any positive $k$ and $\e$, let $\theta \in [\theta_0 - k\e, \theta_0 + k\e ]^d \subset \Theta \subset \mathbb R^d$. Now consider a set $H_i= \{\theta_i^0,\theta_i^1, \ldots \theta_i^J, \theta_i^{J+1} \}$ and $H=\bigotimes_d H_i $ with $J= \lfloor\frac{2 k\e}{\delta'}\rfloor$, where  $\theta^j_i = \theta_0 - k\e + i\delta'$ for $j=\{0,1,\dots , J \}$ and $\theta_i^{J+1}= \theta_0 + k\e$. Observe that for any $\theta \in [\theta_0 - k\e, \theta_0 + k\e ]^d$, there exists a $\theta^j \in  H$ such that $\| \theta - \theta^j \| < \delta'$. Hence, union of the $\delta'-$balls for each element in set $H$ covers $[\theta_0 - k\e, \theta_0 + k\e ]^d$, therefore $N(\delta',[\theta_0 - k\e, \theta_0 + k\e ]^d, \|\cdot\|) = (J+2)^d$.
    
     Now, due  to Assumption~\ref{assump:Asf4}, for any $\theta \in [\theta_0 - k\e, \theta_0 + k\e ]^d$
    \begin{align*}
        d_L(\theta,\theta_0) \leq K_2 \| \theta - \theta^j \| \leq K_2 \delta',
    \end{align*}
    For brevity, we denote $n^{-1} {L_n(\theta,\theta_0)}$ by $d_L(\theta,\theta_0)$, that is
    \vspace{0em}
    \begin{equation}
        d_L(\theta_1,\theta_2) := \sup_{a \in \mathcal{A}} \vert R(a,\theta_1) - R(a,\theta_2) \vert,~ \forall\{\theta_1,\theta_2\}\in\Theta,
    \end{equation}
    and the covering number of the set $T(\e):= \{P_{\theta} : d_L(\theta,\theta_0) < \e\}$ as  $N(\delta, T(\e), d_L )$, where $\delta > 0 $ is the radius of each ball in the cover. 
    
    Hence, $\delta'$-cover of set $[\theta_0 - k\e, \theta_0 + k\e ]^d$ is $K_1\delta'$ cover of set $T(\e)$ with $k= 1/K_2$. Finally,
    \begin{align*}
        N(K_2 \delta', T(\e), d_L ) \leq (J+2)^d \leq \left( \frac{2 k\e}{\delta'} + 2 \right)^d=   \left( \frac{2 \e}{K_2\delta'} + 2 \right)^d
    \end{align*}
    
    which implies for $\delta = K_2\delta'$, 
    \begin{align*}
        N(\delta, T(\e), d_L ) \leq  \left(\frac{2 \e}{\delta} + 2 \right)^s.  
    \end{align*}
\end{proof}

Observe that the RHS in~\eqref{eq:eqc1} is a decreasing function of $\delta$, infact for $\delta=\e/2$, it is a constant in  $\e$. Therefore, using Lemmas~\ref{lem:ggv71} and \ref{lem:cover}, we show in the following result that $L_n(\theta,\theta_0)$ in~\eqref{eq:lossr} satisfies Assumption~\ref{assump:Asf1}. 

\begin{lemma}\label{lem:ass1} Fix $n \geq 1$. For a given $\e_n>0$ and every $\e> \e_n$, such that $n\e_n^2\geq 1$. Under Assumption~\ref{assump:Asf5} and ~\ref{assump:Asf4}
	, $L_n(\theta,\theta_0) = n \left(\sup_{a \in \mathcal{A}} \vert R(a,\theta) - R(a,\theta_0)  \vert \right)^2 $	satisfies
	\vspace{0em}
	\begin{align}
		\bbE_{P^{n}_0}[ \phi_n ]  & \leq C_0 \exp(-C n \e^2 ), \label{eq:eq12} \\
		\underset{\{ \theta\in \Theta : L_n(\theta,\theta_0) \geq C_1n \e^2  \}}{\sup} \bbE_{P^{n}_{\theta}}[ 1- \phi_n ]  & \leq \exp(-C n \e^2 ),\label{eq:eq13}
	\end{align} 
	where $C_0=2*10^s$ and $C = \frac{C_1}{2 K_1^2}$ for a constant $C_1>0$.
\end{lemma}

\begin{proof}[Proof of Lemma~\ref{lem:ass1}:]
    Recall $d_L(\theta,\theta_0) = \left(\sup_{a \in  \mathcal{A}} |R(a,\theta) -  R(a,\theta_0)  |\right)$ and $T(\e)= \{ P_{\theta}: d_L(\theta,\theta_0) < \e \}$. Using Lemma~\ref{lem:cover}, observe that for every $ \e > \e_n > 0$,
    \begin{align*}
        N\left(\frac{\e}{2}, \{ \theta: \e \leq d_L(\theta,\theta_0) \leq 2\e \}, d_L \right) \leq N \left(\frac{\e}{2}, \{ \theta: d_L(\theta,\theta_0) \leq 2\e \}, d_L \right) < 10^d.
    \end{align*}
    Next, using Assumption~\ref{assump:Asf5} we have
    \[ d_L(\theta,\theta_0) \leq  K_1   d_{H}(\theta,\theta_0). \]   
    It follows from the above two observations and Lemma 2 that, for every $\e > \e_n>0$, there exist tests $\{\phi_{n,\e}\}$ such that 
    \begin{align}
        \bbE_{P^{n}_0}[ \phi_{n,\e} ]  &\leq 10^d  \frac{\exp(-C' n \e^2 )}{1-\exp(-C'n\e^2)},  \\
        \underset{\{ \theta\in \Theta : d_L(\theta,\theta_0) \geq \epsilon  \}}{\sup} \bbE_{P^{n}_{\theta}}[ 1- \phi_{n,\e} ] &\leq  \exp(-C' n \e^2 ),\label{eq:eq3}
    \end{align} 
    where $C' = \frac{1}{2 K_1^2}$. Since the above two conditions hold for every $\e > \e_n$, we can choose a constant $K>0$ such that for every $\e >  \e_n$ 
    \begin{align}
        \bbE_{P^{n}_0}[ \phi_{n,\e} ]  &\leq 10^d  \frac{\exp(-C' K^2 n \e^2 )}{1-\exp(-C'K^2n\e^2)} \leq 2 (10^d) e^{-C' K^2  n \e^2 }, \label{eq:eq2} \\
        \underset{\{ \theta\in \Theta : L_n(\theta,\theta_0) \geq K^2n \e^2  \}}{\sup} \bbE_{P^{n}_{\theta}}[ 1- \phi_{n,\e} ]  & = \underset{\{ \theta\in \Theta : d_L(\theta,\theta_0) \geq K\e  \}}{\sup} \bbE_{P^{n}_{\theta}}[ 1- \phi_{n,\e} ] \leq  e^{-C' K^2 n \e^2 },\label{eq:eq31}
    \end{align} 
    where the second inequality in~\eqref{eq:eq2} holds $\forall n \geq n_0$, where $n_0:= \min\{n\geq 1: C' K^2 n \e^2 \geq \log(2)   \}$
    %
    Hence, the result follows for $C_1=K^2$ and $C= C' K^2$. 
\end{proof}

Since $L_n(\theta,\theta_0)= \frac{1}{n}d_L^2$ satisfies Assumption~\ref{assump:Asf1}, Theorem~\ref{thrm:thm1} implies the following bound.

\begin{corollary}\label{thrm:thm1a}
	Fix $a'\in \sA$ and $\gamma > 0$. Let $\e_n$ be a sequence such that $\e_n \to 0$ as $n \to \infty$, $n \e_n^2 \geq 1$  and
	\vspace{0em} $$L_n(\theta,\theta_0) = n \left(\sup_{a \in \mathcal{A}} \vert R(a,\theta) - R(a,\theta_0)  \vert \right)^2. $$ Then under the Assumptions of Theorem~\ref{thrm:thm1} and Lemma~\ref{lem:ass1} 
	; for $C = \frac{C_1}{2 K_1^2}$, $C_0= 2*10^s$, $C_1>0$ such that $\min(C,C_4(\gamma) + C_5(\gamma)) > C_2 + C_3 + C_4(\gamma) + 2$ , and for $\eta_n^R(\gamma)$ as defined in Theorem~\ref{thrm:thm1}, 
the RSVB approximator of the true posterior $\qrs$ satisfies,
\begin{align}
	\bbE_{P^{n}_0} \left[ \int_{\Theta} L_n(\theta,\theta_0) \qrs d\theta \right] \leq n (\textcolor{black}{M(\gamma)} \e_n^2 + M \eta_n^R(\gamma)), \label{eq:eq_lcbr}
\end{align}
 for sufficiently large $n$ and for a function \textcolor{black}{$M(\gamma)= 2\left( C_{1}+ M C_4(\gamma)  \right)$}~, where $M= \frac{2C_1}{\min(C,\lambda,1)}$
.

\end{corollary}

\begin{proof}[Proof of Corollary~\ref{thrm:thm1a}:]
    Using Lemma~\ref{lem:ass1} observe that for any $\Theta_n(\e) \subseteq \Theta$,  $L_n(\theta,\theta_0)$ satisfies Assumption~\ref{assump:Asf1} with $C_0=2* 10^s$, $C = \frac{C_1}{2K_1^2}$ and for any $C_1>0$, since 
    \begin{align*}
        \underset{\{ \theta\in \Theta_n(\e) : L_n(\theta,\theta_0) \geq C_1 n \epsilon_n^2  \}}{\sup} \bbE_{P^{n}_{\theta}}[ 1- \phi_{n,\e} ] \leq \underset{\{ \theta\in \Theta : L_n(\theta,\theta_0) \geq C_1 n \epsilon_n^2 \}}{\sup} \bbE_{P^{n}_{\theta}}[ 1- \phi_{n,\e} ] &\leq e^{-C n \e_n^2 }.
    \end{align*}
    Hence, applying Theorem~\ref{thrm:thm1} the proof follows.
\end{proof}

\subsection{Proof of Theorem~\ref{thrm:NV},~\ref{thrm:NVOGvalue}, and ~\ref{thrm:NVOGdecision}}
 
\begin{proof}[Proof of Theorem~\ref{thrm:NV}:]

The proof follows immediately from Theorem~\ref{thrm:thm1} by taking limit $\gamma\to 0^+$ on either side of its main result, that is
\begin{align}
    \bbE_{P^{n}_0} \left[ \int_{\Theta}  L_n(\theta,\theta_0)\ d\qrs \right] \leq M n (\e_n^2  + \eta_n^R(\gamma) ). 
\end{align}
Fix $n\geq 1$. Now first consider the LHS, use the fact that  for any $a'\in\sA$, $\lim_{\gamma\to 0^+} \qrs = \qnv$~\eqref{eq:eqRSVB}, the integrand is also non-negative , and $n\eta_n^{R}(\gamma)<\infty$ due to Proposition~\ref{prop:eta_gamma} (since a decreasing sequence is bounded given $\eta_n^{R}(0)<\infty$), therefore, using Fatou's Lemma we have 
\begin{align}
   \bbE_{P^{n}_0} \left[ \int_{\Theta}  L_n(\theta,\theta_0)\  d \qnv \right] \leq \liminf_{\gamma\to 0^+} \bbE_{P^{n}_0} \left[ \int_{\Theta}  L_n(\theta,\theta_0)\ d\qrs \right]  
\end{align}

On the other hand, using similar argument as used in~\eqref{eq:eqRSVB} to show that $\qrs \to \qnv$ as  $\gamma\to 0^+$, it follows that 
\[ \liminf_{\gamma\to 0^+} \eta_n^{R}(\gamma) = \eta_n^{R}(0).\]
Thus the result follows.

\end{proof}

Next, we obtain a high-probability bound on the regret, defined as the uniform difference between the Naive VB approximate posterior risk and the expected loss under the true data generating measure $P_{0}$. 

\begin{lemma}\label{corr:cor2}
	For a constant M as defined in Theorem~\ref{thrm:NV}
	\begin{align}
		\bbE_{P^{n}_0} \left[  \sup_{a \in \mathcal{A}} \left \vert \int    R(a,\theta)  d \qnv - R(a,\theta_0)  \right \vert \right] \leq \left[ M (\e_n^2 + \eta_n(0)) \right]^{\frac{1}{2}}.
	\end{align}
\end{lemma}
\begin{proof}
The result follows immediately from the following inequalities 
\begin{align*}
	\left( \sup_{a \in \mathcal{A}} \left\vert \bbE_{\qnv}\left[  R(a,\theta)\right] - R(a,\theta_0) \right\vert \right)^2 & \leq  \left( \bbE_{\qnv}\left[ \sup_{a \in \mathcal{A}} \vert R(a,\theta) - R(a,\theta_0) \vert   \right]  \right)^2
	\\
	&\leq \bbE_{\qnv}\left[ \left(\sup_{a \in \mathcal{A}} \vert R(a,\theta) - R(a,\theta_0) \vert \right)^2 \right],
\end{align*}
where the last inequality is  a consequence of Jensens' inequality. Now, using Jensen's inequality again
\vspace{0em} 
\begin{align*}
	&\left ( \bbE_{P^{n}_0} \left[ \sup_{a \in \mathcal{A}} \left \vert \int   R(a,\theta)  d \qnv - R(a,\theta_0) \right \vert \right] \right)^2 
	\\
	&\quad \quad \leq  \bbE_{P^{n}_0} \left[  \left( \sup_{a \in \mathcal{A}} \left\vert \int   R(a,\theta)  d \qnv - R(a,\theta_0) \right\vert \right)^2 \right].
\end{align*} 
Now the result follows immediately using Theorem~\ref{thrm:NV}.
\end{proof}


\begin{proof}[Proof of Theorem~\ref{thrm:NVOGvalue} and ~\ref{thrm:NVOGdecision} ]

The proof is similar to  Theorem~\ref{thrm:OGValue} and ~\ref{thrm:OGdecision} and hence omitted.

\end{proof}

%
%
%

\subsection{Proof of Theorem~\ref{thrm:LC}, \ref{thrm:LCOGvalue} and~\ref{thrm:LCOGdecision}}

\begin{proof}[Proof of Theorem~\ref{thrm:LC}:]

The proof follows immediately from Theorem~\ref{thrm:LC} by substituting  $\gamma= 1$. 

\end{proof}

\begin{proof}{Proof of Theorem~\ref{thrm:LCOGvalue}~and~\ref{thrm:LCOGdecision}}: The proof is similar to  Theorem~\ref{thrm:OGValue} and ~\ref{thrm:OGdecision} and hence omitted.

\end{proof}

\subsection{Newsvendor Problem}

We fix $n^{-1/2}\sqrt{L^{NV}_n(\theta,\theta_0)} =  (\sup_{a \in  \mathcal{A}} | R(a,\theta) - R(a,\theta_0) |)$. Next, we aim to show that the exponentially distributed model $P_{\theta}$ satisfies Assumption~\ref{assump:Asf1}, for distance function $L^{NV}_n(\theta,\theta_0)$. To show this, in the next result we first prove that $d_L^{NV}(\theta,\theta_0) = n^{-1/2}\sqrt{L^{NV}_n(\theta,\theta_0)}$  satisfy Assumption~\ref{assump:Asf5}. Also, recall that the square of Hellinger distance between two exponential distributions with rate parameter $\theta$ and $\theta_0$ is 
\( d_{H}^2(\theta,\theta_0) = 1-2\frac{\sqrt{\theta \theta_0}}{\theta+ \theta_0} = 1-2\frac{\sqrt{ \theta_0/\theta}}{1+ \theta_0/\theta}.\)

 \begin{lemma}\label{lem:HellNV}
    For any $\theta\in \Theta=[T,\infty)$,  and $a\in \sA$,
    \[d_L^{NV}(\theta,\theta_0) \leq \left[
    \frac{\left( \frac{h}{\theta_0}- \frac{h}{T} \right)^2  + (b+h)^2 \left( \frac{e^{-\underline{a} T}}{T}  -  \frac{e^{- \underline{a} \theta_0}}{\theta_0}\right)^2}{d^2_{H}(T,\theta_0)}\right]^{1/2} d_{H}(\theta,\theta_0)\]
    where $\underline{a} := \min \{a\in \mathcal A\}$ and $\underline{a} >0$ and $\theta_0$ lies in the interior of $\Theta$.
     \end{lemma}
 \begin{proof}
    

Observe that for any $a\in \sA$, 
    \begin{align} 
        \nonumber
        &|R(a,\theta)- R(a,\theta_0)|^2 
        \\
        \nonumber
        =&  \left| \frac{h}{\theta_0}- \frac{h}{\theta} + (b+h) \left( \frac{e^{-a \theta}}{\theta}  -  \frac{e^{-a \theta_0}}{\theta_0}\right) \right|^2 
        \\
        \nonumber
        &= \left( \frac{h}{\theta_0}- \frac{h}{\theta} \right)^2  + (b+h)^2 \left( \frac{e^{-a \theta}}{\theta}  -  \frac{e^{-a \theta_0}}{\theta_0}\right)^2 +2 \left( \frac{h}{\theta_0}- \frac{h}{\theta} \right)(b+h) \left( \frac{e^{-a \theta}}{\theta}  -  \frac{e^{-a \theta_0}}{\theta_0}\right) 
        \\
        & \leq \left( \frac{h}{\theta_0}- \frac{h}{\theta} \right)^2  + (b+h)^2 \left( \frac{e^{-a \theta}}{\theta}  -  \frac{e^{-a \theta_0}}{\theta_0}\right)^2 ,
        \label{eq:NV2}
        \end{align}    
    
    where the last inequality follows since for $\theta\geq \theta_0$, $\left(\frac{h}{\theta_0}- \frac{h}{\theta} \right)\geq 0$ and  $\left( \frac{e^{-a \theta}}{\theta}  -  \frac{e^{-a \theta_0}}{\theta_0}\right)<0$ and vice versa if $\theta<\theta_0$  that together makes the last term in the penultimate equality   negative for  all $\theta \in  \Theta$.
   Moreover, the first  derivative of the upperbound with respect to $\theta$ is
    \[ 2 \left( \frac{h}{\theta_0}- \frac{h}{\theta} \right) \frac{h}{\theta^2}  - 2(b+h)^2 \left( \frac{e^{-a \theta}}{\theta}  -  \frac{e^{-a \theta_0}}{\theta_0}\right) e^{-a\theta}\left[ \frac{1}{\theta^2} + \frac{a}{\theta}\right],\]
    and it is negative when $\theta\leq \theta_0$ and  positive when $\theta > \theta_0$ for  all $b>0,h>0,$ and $a\in\sA$.  Therefore, the upperbound in~\eqref{eq:NV2} above is decreasing function of $\theta$ for all $\theta\leq \theta_0$ and increasing function of $\theta$ for all $\theta > \theta_0 $. The upperbound is tight at $\theta=\theta_0$.
    
    Now recall that the squared  Hellinger distance between  two exponential distributions with rate parameter $\theta$ and $\theta_0$ is 
    \[ d_{H}^2(\theta,\theta_0) = 1-2\frac{\sqrt{\theta \theta_0}}{\theta+ \theta_0} = 1-2\frac{\sqrt{ \theta_0/\theta}}{1+ \theta_0/\theta}  = \frac{(1-\sqrt{ \theta_0/\theta})^2}{1+ (\sqrt{ \theta_0/\theta})^2}.\]
    
    Note that for $\theta\leq \theta_0$,  $d_{H}^2(\theta,\theta_0) $ is a decreasing function of $\theta$  and for all $\theta > \theta_0$ it is an increasing function of $\theta$. Also, note that as $\theta\to \infty$, the squared Hellinger distance as well as the upperbound computed in~\eqref{eq:NV2} converges to a constant for a given $h,b,\theta_0$ and $a$.  However, as $\theta \to 0$, the $d_{H}^2(\theta,\theta_0) \to 1$  but  the upperbound computed in~\eqref{eq:NV2} diverges.

    Since, $\Theta=[T,\infty)$ for some $T>0$ and $T\leq \theta_0$,  observe that if we scale $d_{H}^2(\theta,\theta_0)$ by factor by which the upperbound computed in~\eqref{eq:NV2} is greater than $d_{H}$ at $\theta=T$,  then   
    \begin{align*} 
        &\left( \frac{h}{\theta_0}- \frac{h}{\theta} \right)^2  + (b+h)^2 \left( \frac{e^{-a \theta}}{\theta}  -  \frac{e^{-a \theta_0}}{\theta_0}\right)^2 
        \\
    &\leq \frac{\left( \frac{h}{\theta_0}- \frac{h}{T} \right)^2  + (b+h)^2 \left( \frac{e^{-{a} T}}{T}  -  \frac{e^{- {a} \theta_0}}{\theta_0}\right)^2}{d_{H}^2(T,\theta_0)} d_{H}^2(\theta,\theta_0)
    \\
    &\leq \frac{\left( \frac{h}{\theta_0}- \frac{h}{T} \right)^2  + (b+h)^2 \left( \frac{e^{-\underline{a} T}}{T}  -  \frac{e^{- \underline{a} \theta_0}}{\theta_0}\right)^2}{d_{H}^2(T,\theta_0)} d_{H}^2(\theta,\theta_0),
        \end{align*}
    where $\underline{a} = \inf \{a:a\in \sA\}$ and in the last  inequality we used the fact that $\left( \frac{e^{-{a} T}}{T}  -  \frac{e^{- {a} \theta_0}}{\theta_0}\right)^2$ is a decreasing function of $a$ for any $b,h,T,$ and $\theta_0$    . 
    Since, the RHS in the equation  above does not depend on  $a$, it follows from the result in~\eqref{eq:NV2} and the definition of $L_n^{NV}(\theta,\theta_0)$ that 
    \[d_L^{NV}(\theta,\theta_0) \leq \left[
    \frac{\left( \frac{h}{\theta_0}- \frac{h}{T} \right)^2  + (b+h)^2 \left( \frac{e^{-\underline{a} T}}{T}  -  \frac{e^{- \underline{a} \theta_0}}{\theta_0}\right)^2}{d_{H}^2(T,\theta_0)} \right]^{1/2} d_{H}(\theta,\theta_0).\] 
\end{proof}

    \begin{lemma}\label{lem:LipHellNV}
        For any $\theta\in \Theta = [T,\infty)$,  for sufficiently small $T>0$,  and $\theta_0$ lying in the interior of $\Theta$, we have
        \[ d_{H}^2(\theta,\theta_0) = 1-2\frac{\sqrt{\theta \theta_0}}{\theta+ \theta_0} \leq \left(\frac{\theta_0}{(T+\theta_0)^2}\left(\sqrt{\frac{\theta_0}{T}} -\sqrt{\frac{T}{\theta_0}}\right)\right)|\theta - \theta_0|.\]
        \end{lemma}
    \begin{proof}
        Observe that 
        \[\frac{\partial d_{H}^2(\theta,\theta_0) }{\partial \theta} =- 2\frac{(\theta+\theta_0)\frac{\sqrt{\theta_0}}{2\sqrt{\theta}}-\sqrt{\theta \theta_0} }{(\theta+\theta_0)^2} 
        = \frac{\theta_0 }{(\theta+\theta_0)^2} \left({\sqrt{\frac{\theta}{\theta_0}}}-{\sqrt{\frac{\theta_0}{\theta}}}\right). \]
        
        Observe that $\theta \to 0$, $\frac{\partial d_{H}^2(\theta,\theta_0) }{\partial \theta}\to \infty$. Since,$ \theta\in \Theta = [T,\infty)$, therefore the $\sup_{\theta \in \Theta}\left| \frac{\partial d_{H}^2(\theta,\theta_0) }{\partial \theta} \right| <\infty$. In fact, for sufficiently small $T>0$,
        $\sup_{\theta \in \Theta}\left| \frac{\partial d_{H}^2(\theta,\theta_0) }{\partial \theta} \right|= \left|\frac{\theta_0}{(T+\theta_0)^2}\left(\sqrt{\frac{T}{\theta_0}}-\sqrt{\frac{\theta_0}{T}}\right)\right| = \left(\frac{\theta_0}{(T+\theta_0)^2}\left(\sqrt{\frac{\theta_0}{T}} -\sqrt{\frac{T}{\theta_0}}\right)\right)$. Now the result follows immediately since  the derivative of $d_{H}^2(\theta,\theta_0)$ is bounded on $\Theta$, which implies that $d_{H}^2(\theta,\theta_0)$ is Lipschitz on $\Theta$.
        \end{proof}
    
    \begin{lemma}\label{lem:LipNV}
        For any $\theta\in \Theta=[T,\infty)$,  and $a\in \sA$,
        \[d_L^{NV}(\theta,\theta_0) \leq \frac{h}{T^2}|\theta-\theta_0| .\]
    \end{lemma}
    \begin{proof}
         Recall, 
        \[R(a,\theta)=ha- \frac{h}{\theta} + (b+h) \frac{e^{-a \theta}}{\theta}.\]
        First, observe that for any $a\in \sA$,
        \begin{align}
            \frac{\partial R(a,\theta)}{\partial \theta} = \frac{h}{\theta^2} - a (b+h) \frac{e^{-a \theta}}{\theta} - (b+h) \frac{e^{-a \theta}}{\theta^2} = \frac{1}{\theta^2} \left( h - (b+h)e^{-a\theta}(1+a\theta) \right) \leq \frac{h}{\theta^2}.
            \label{eq:gradNV1}
        \end{align}
        The result follows immediately, since $\sup_{\theta\in \Theta} \frac{\partial R(a,\theta)}{\partial \theta} \leq \frac{h}{T^2}$. 
        \end{proof}

\begin{proof}{Proof of Lemma~\ref{lem:nv2}}
    
    It follows from Lemma~\ref{lem:HellNV}  that $d_L^{NV}(\theta,\theta_0)$ for any $\theta\in \Theta=[T,\infty)$ and $\theta_0$ lying the  interior of $\Theta$,  satisfies Assumption~\ref{assump:Asf5} with \[K_1= \left[
    \frac{\left( \frac{h}{\theta_0}- \frac{h}{T} \right)^2  + (b+h)^2 \left( \frac{e^{-\underline{a} T}}{T}  -  \frac{e^{- \underline{a} \theta_0}}{\theta_0}\right)^2}{d_{H}^2(T,\theta_0)} \right]^{1/2}:=K_1^{NV}\]. Similarly, it follows from  Lemma and~\ref{lem:LipNV} that for sufficiently small $T>0$, $d_L^{NV}(\theta,\theta_0)$ satisfies Assumption~\ref{assump:Asf4} with  $K_2 = h/T^2 := K_2^{NV}  $. Now using similar arguments as used in Lemma~\ref{lem:cover} and Lemma~\ref{assump:Asf1}, for a given $\e_n>0$ and every $\e> \e_n$, such that $n\e_n^2\geq 1$, it can be shown  that 
    , $L_n^{NV}(\theta,\theta_0) = n \left(\sup_{a \in \mathcal{A}} \vert R(a,\theta) - R(a,\theta_0)  \vert \right)^2 $	satisfies
    \vspace{0em}
    \begin{align}
        \bbE_{P^{n}_0}[ \phi_n ]  & \leq C_0 \exp(-C n \e^2 ), \label{eq:eq12nv1} \\
        \underset{\{ \theta\in \Theta : L_n^{NV}(\theta,\theta_0) \geq C_1n \e^2  \}}{\sup} \bbE_{P^{n}_{\theta}}[ 1- \phi_n ]  & \leq \exp(-C n \e^2 ),\label{eq:eq13nv1}
    \end{align} 
    where $C_0=20$ and $C = \frac{C_1}{2 (K_1^{NV})^2}$ for a constant $C_1>0$.
    \end{proof}

 \begin{proof}{Proof  of Lemma~\ref{lem:nv3}:}

First, we write the R\'enyi divergence between $P_{0}^n$ and $P_{\theta}^n$,
\begin{align*}
    D_{1+\lambda} \left( P_0^n \| P_{\theta}^n \right) = \frac{1}{\lambda} \log \int \left(\frac{dP_0^n}{dP_{\theta}^n}\right)^\lambda dP_0^n & = n \frac{1}{\lambda} \log \int \left(\frac{dP_0}{dP_{\theta}}\right)^\lambda dP_0
    \\
    &= n  \left( \log \frac{\theta_0}{\theta} + \frac{1}{\lambda} \log \frac{\theta_0}{(\lambda + 1)\theta_0 -\lambda \theta } \right), 
\end{align*} 
when  $\left((\lambda + 1)\theta_0 -\lambda \theta\right)>0$ and $D_{1+\lambda} \left( P_0^n \| P_{\theta}^n \right) = \infty$ otherwise. Also, observe that, $D_{1+\lambda} \left( P_0^n \| P_{\theta}^n \right)$ is non-decreasing in $\lambda$ (this also follows from non-decreasing property of the R\'enyi divergence  with respect to $\lambda$). Therefore, observe  that
\begin{align*}
    \Pi(D_{1+\lambda} \left( P_0^n \| P_{\theta}^n \right) \leq C_3n\e_n^2)\geq \Pi(D_{\infty} \left( P_0^n \| P_{\theta}^n \right) \leq C_3n\e_n^2)  &= \Pi\left( 0 \leq \log \frac{\theta_0}{\theta} \leq C_3 \e_n^2\right)  
    \\
    &= \Pi\left( \theta_0e^{-C_3 \e_n^2}\leq {\theta} \leq \theta_0 \right).
\end{align*}

Now, recall that for a set $A\subseteq \Theta=[T,\infty)$, we define $\Pi(A) = \text{Inv}-\Gamma(A\cap \Theta)/\text{Inv}-\Gamma(\Theta) $. Now, observe that for sufficiently small $T$ and large enough $n$,  we have
\[ \Pi\left( \theta_0e^{-C_3 \e_n^2}\leq {\theta} \leq \theta_0 \right) \geq  \text{Inv}-\Gamma\left( \theta_0e^{-C_3 \e_n^2}\leq {\theta} \leq \theta_0 \right) \]
The  cumulative  distribution function of inverse-gamma distribution  is $\text{Inv}-\Gamma(\{\theta\in \Theta: \theta< t \} ):= \frac{\Gamma\left(\alpha,\frac{\beta}{t} \right)}{\Gamma(\alpha)}$, where $\alpha(>0)$ is the  shape parameter, $\beta(>0)$ is the  scale parameter, $\Gamma(\cdot)$ is the Gamma function,  and $\Gamma(\cdot,\cdot)$ is the  incomplete Gamma function. Therefore, it  follows for $\alpha>1$ that
\begin{align} 
    \nonumber
    \text{Inv}-\Gamma&\left( \theta_0e^{-C_3 \e_n^2}\leq {\theta} \leq \theta_0 \right)
    \\
    \nonumber
    & = \frac{\Gamma\left(\alpha,\beta/\theta_0 \right) - \Gamma\left(\alpha,\beta/\theta_0 e^{C_3  \e_n^2}  \right)}{\Gamma(\alpha)} = \frac{\int_{\beta/\theta_0}^{\beta/\theta_0 e^{C_3  \e_n^2}} e^{-x} x^{\alpha-1}dx }{\Gamma(\alpha)}
    \\
    \nonumber
    &\geq \frac{e^{-\beta/\theta_0 e^{C_3  \e_n^2} + \alpha C_3\e_n^2 }}{\alpha \Gamma(\alpha)} \left(  \frac{\beta}{\theta_0}\right)^{\alpha} \left[ 1 - e^{-\alpha C_3  \e_n^2 }  \right] 
    \\
    \nonumber
    & \geq \frac{e^{-\beta/\theta_0 e^{C_3  }  }}{\alpha \Gamma(\alpha)} \left(  \frac{\beta}{\theta_0}\right)^{\alpha} \left[e^{-\alpha{C_3 n \e_n^2 }} \right] 
\end{align}
where the penultimate inequality folows since $0<\e_n^2<1$ and the last inequality follows from the fact  that, 
$ 1 - e^{-\alpha C_3  \e_n^2} \geq e^{-\alpha{C_3 n \e_n^2 }}$, for large enough $n $. Also note that, $ 1 - e^{-\alpha C_3  \e_n^2} \geq e^{-\alpha{C_3 n \e_n^2 }}$ can't hold true for $\e_n^2=1/n$. However, for $\e_n^2=\frac{\log n}{n}$ it holds for any $n\geq 2$ when $\alpha C_3 >2$. Therefore, for inverse-Gamma prior restricted to $\Theta$, $C_2=\alpha C_3$ and any $\lambda>1$ the  result follows for sufficiently large $n$. 

 \end{proof}



\begin{proof}{Proof of Lemma~\ref{lem:nv4}:}
    Recall, 
    \[R(a,\theta)=ha- \frac{h}{\theta} + (b+h) \frac{e^{-a \theta}}{\theta}.\]
    First, observe that for any $a\in \sA$,
    \begin{align}
        \frac{\partial R(a,\theta)}{\partial \theta} = \frac{h}{\theta^2} - a (b+h) \frac{e^{-a \theta}}{\theta} - (b+h) \frac{e^{-a \theta}}{\theta^2} = \frac{1}{\theta^2} \left( h - (b+h)e^{-a\theta}(1+a\theta) \right).
        \label{eq:gradNV}
    \end{align}
    
    Using the above equation the (finite) critical point $\theta^*$ must satisfy, \(h - (b+h)e^{-a\theta^*}(1+a\theta^*) = 0  \). Therefore,
    \[R(a,\theta) \geq R(a,\theta^*) = h \left(a - \frac{1}{\theta^*} + \frac{1}{\theta^*(1+a\theta^*)} \right)  =  \frac{  h a^2 {\theta^*}}{(1+a\theta^*)} . \]
    Since $h,b >0$ and $a\theta^*>0$, hence
    \[R(a,\theta) \geq  \frac{  h \underline{a}^2 { \theta^*}}{(1+a\theta^*)}, \]
    where $\underline{a} := \min \{a\in \mathcal A\}$ and $\underline{a} >0$.
    
    
\end{proof}

 \begin{proof}{Proof of Lemma~\ref{lem:nv5}:}

First, observe that $R(a,\theta)$ is bounded above in $\theta$ for a given $a \in \mathcal{A}$ 
\begin{align*}
R(a,\theta)&= ha- \frac{h}{\theta} + (b+h) \frac{e^{-a \theta}}{\theta}\\
&\leq ha + \frac{b}{\theta}.
\end{align*}
Using the above fact and the Cauchy-Schwarz inequality, we obtain
\begin{align}
\nonumber
		&\int_{\left\{e^{ \gamma R(a,\theta)} > e^{C_4(\gamma)n \e_n^2} \right\} } e^{ \gamma R(a,\theta)} \pi(\theta) d\theta  
		\\
		\nonumber
		& \leq \left( \int  e^{ 2\gamma R(a,\theta)} \pi(\theta) d\theta \right)^{1/2} \left( \int  \Ind{e^{ \gamma R(a,\theta)} > e^{C_4(\gamma)n \e_n^2}  } \pi(\theta) d\theta \right)^{1/2} 
		\\
		\nonumber
		& \leq  \left( \int  e^{2\gamma\left(ha + \frac{b}{\theta} \right) } \pi(\theta) d\theta \right)^{1/2} \left( \int  \Ind{\{e^{\gamma\left(ha + \frac{b}{\theta} \right) } > e^{C_4(\gamma)n \e_n^2} \} } \pi(\theta) d\theta \right)^{1/2}
		\\
		& \leq e^{-{C_4(\gamma)}n \e_n^2} \left( \int  e^{2\gamma\left(ha + \frac{b}{\theta} \right) } \pi(\theta) d\theta \right),
\end{align}
where the last inequality follows from using the Chebyshev's inequality. 

Now  using the definition  of the prior distribution, which is  an inverse gamma prior restricted to $\Theta= [T,\infty)$, we have 
\begin{align}
    \nonumber
    \int_{\left\{e^{ \gamma R(a,\theta)} > e^{C_4(\gamma)n \e_n^2} \right\} } e^{ \gamma R(a,\theta)} \pi(\theta) d\theta &\leq e^{-{C_4(\gamma)}n \e_n^2} \left( \int  e^{2\gamma\left(ha + \frac{b}{\theta} \right) } \pi(\theta) d\theta \right)
    \\
    \nonumber
    &\leq e^{-{C_4(\gamma)}n \e_n^2} e^{2\gamma\left(h\overline{a} + \frac{b}{T} \right) },
    \end{align}
where $\overline{a} := \max \{a\in \mathcal A\}$ and $\overline{a} >0$. Since $n\e_n^2\geq 1$, we must fix $C_4(\gamma)$ such that
$e^{C_4(\gamma)}> e^{2\gamma\left(h \overline{a} + \frac{b}{T} \right) }$, that is $C_4(\gamma) > 2\gamma \left(h\overline{a} + \frac{b}{T}  \right) $ and $C_5(\gamma)= C_4(\gamma) - 2\gamma \left(h \overline{a} + \frac{b}{T}  \right) $.
%

 
 \end{proof}

\begin{proof}{Proof of Lemma~\ref{lem:nv6}:}
    Since family $\sQ$ contains all shifted-gamma distributions, observe that $\{q_n(\cdot) \in \sQ \} \forall n\geq 1$. By definition, $q_n(\theta) = \frac{n^n}{\theta_0^n \Gamma(n)} (\theta-T)^{n-1}e^{-n\frac{(\theta-T)}{\theta_0}}$. Now consider the first term; using the definition of the $\scKL$ divergence it follows that 
    \begin{align}
        \scKL(q_n(\theta)\| \pi(\theta)) = \int_T^{\infty} q_n(\theta) \log (q_n(\theta))d\theta - \int_T^{\infty} q_n(\theta) \log (\pi(\theta))d\theta.
        \label{eq:eqr1c}
    \end{align}  
    Substituting $q_n(\theta)$ in the first term of the equation above and expanding the logarithm term, we obtain
    \begin{align}
        \nonumber
        &\int_T^{\infty} q_n(\theta) \log (q_n(\theta))d\theta
        \\
        \nonumber
        &= (n-1) \int_T^{\infty} \log (\theta-T) \frac{n^n}{\theta_0^n \Gamma(n)} (\theta-T)^{n-1}e^{-n\frac{\theta-T}{\theta_0}} d\theta - n + \log \left( \frac{n^n}{\theta_0^n \Gamma(n)} \right)
        \\
        &= - \log \theta_0 + (n-1) \int_T^{\infty} \log \frac{\theta-T}{\theta_0} \frac{n^n}{\theta_0^n \Gamma(n)} (\theta-T)^{n-1}e^{-n\frac{\theta-T}{\theta_0}} d\theta -n + \log \left( \frac{n^n}{ \Gamma(n)} \right)
        \label{eq:eqr2c}
    \end{align}
    Now consider the second term in the equation above. Substitute $\theta = \frac{t \theta_0}{n}+T$ into the  integral, we have \begin{align}
        \nonumber
        \int_T^{\infty} \log \frac{\theta-T}{\theta_0} \frac{n^n}{\theta_0^n \Gamma(n)} (\theta-T)^{n-1}e^{-n\frac{\theta-T}{\theta_0}} d\theta  &=   \int_0^{\infty} \log \frac{t}{n}  \frac{1}{ \Gamma(n)} t^{n-1}e^{-t} dt
        \\
        &\leq \int \left( \frac{t}{n}-1  \right) \frac{1}{ \Gamma(n)} t^{n-1}e^{-t} dt =0.
        \label{eq:eqr2ac}
    \end{align} 
    Substituting the above result into~\eqref{eq:eqr2c}, we get
    \begin{align}
        \nonumber
        \int_T^{\infty} q_n(\theta) \log (q_n(\theta))d\theta 
        &\leq - \log \theta_0 - n + \log \left( \frac{n^n}{ \Gamma(n)} \right)
        \\
        \nonumber
        &\leq - \log \theta_0 - n  + \log \left( \frac{n^n}{ \sqrt{2\pi n} n^{n-1}e^{-n}} \right) 
        \\
        &= - \log \sqrt{2\pi } \theta_0  + \frac{1}{2}\log n, 
        \label{eq:eqr3c}
    \end{align}
    where the  second inequality uses the fact that $\sqrt{2\pi n} n^{n}e^{-n} \leq n \Gamma(n) $. Recall $\pi(\theta)= \frac{\beta^\alpha}{\Gamma(\alpha)} \theta^{-\alpha-1}e^{-\frac{\beta}{\theta}}$. Now consider the second term in~\eqref{eq:eqr1c}. Using the definition of inverse-gamma prior and expanding the logarithm function, we have
    \begin{align}
        \nonumber
        -&\int_T^{\infty} q_n(\theta) \log (\pi(\theta))d\theta
        \\
        \nonumber
        &=  - \log \left( \frac{\beta^\alpha}{\Gamma(\alpha)} \right) + (\alpha+1) \int_T^{\infty} \log {\theta} \frac{n^n}{\theta_0^n \Gamma(n)} (\theta-T)^{n-1}e^{-n\frac{\theta-T}{\theta_0}} d\theta  + \beta \frac{n}{(n-1)\theta_0}
        \\
        \nonumber
        &=  - \log \left( \frac{\beta^\alpha}{\Gamma(\alpha)} \right) + \int_T^{\infty} \log \frac{\theta}{\theta_0} \frac{n^n}{\theta_0^n \Gamma(n)} (\theta-T)^{n-1}e^{-n\frac{\theta-T}{\theta_0}} d\theta  
        \\ 
        \nonumber
        &\quad + \beta \frac{n}{(n-1)\theta_0} + (\alpha+1)\log \theta_0
        \\
        \nonumber
        &\leq  - \log \left( \frac{\beta^\alpha}{\Gamma(\alpha)} \right) + \int_T^{\infty}  \frac{\theta-T}{\theta_0} \frac{n^n}{\theta_0^n \Gamma(n)} (\theta-T)^{n-1}e^{-n\frac{\theta-T}{\theta_0}} d\theta  
        \\ 
        \nonumber
        &\quad + \beta \frac{n}{(n-1)\theta_0} + (\alpha+1)\log \theta_0
        \\
        &= - \log \left( \frac{\beta^\alpha}{\Gamma(\alpha)} \right) +\beta \frac{n}{(n-1)\theta_0} + (\alpha+1)\log \theta_0,
        \label{eq:eqr4c}
    \end{align} 
    where the first inequality is  due to fact that $\bbE_{q_n}[\beta/\theta] \leq \bbE_{q_n}[\beta/(\theta-T)] $ for any $\theta>T$ and the penultimate inequality follows from the  observation  in~\eqref{eq:eqr2ac} and the fact that $\log\frac{\theta}{\theta_0} \leq \frac{\theta}{\theta_0} -  1 \leq \frac{\theta}{\theta_0} -\frac{T}{\theta_0}$ for any $\theta_0
    >T$. Substituting~\eqref{eq:eqr4c} and~\eqref{eq:eqr3c} into~\eqref{eq:eqr1c} and dividing either sides by $n$, we obtain
    \begin{align}
        \nonumber
        \frac{1}{n}&\scKL(q_n(\theta)\| \pi(\theta)) 
        \\
        \nonumber
        &\leq  \frac{1}{n}\left( - \log \sqrt{2\pi } \theta_0  + \frac{1}{2}\log n - \log \left( \frac{\beta^\alpha}{\Gamma(\alpha)} \right)   + \beta \frac{n}{(n-1)\theta_0} +  (\alpha+1)\log \theta_0 \right)
        \\
        &= \frac{1}{2}\frac{\log n}{n} + \beta \frac{1}{(n-1)\theta_0} + \frac{1}{n}\left( - \log \sqrt{2\pi }    - \log \left( \frac{\beta^\alpha}{\Gamma(\alpha)} \right)  + (\alpha)\log \theta_0 \right).
        \label{eq:eqr5c}
    \end{align}  
    Now consider the second term in the assertion of the lemma. Since $\xi_i, i\in\{1,2\ldots n\}$ are independent and identically distributed, we obtain 
    \begin{align*}
        \frac{1}{n} \bbE_{q_n(\theta)} \left[ \scKL\left(  dP^{n}_0\| p(\nX|\theta) \right) \right]   =  \bbE_{q_n(\theta)} \left[ \scKL\left(dP_0\| p(\xi|\theta) \right) \right]  
    \end{align*}
    \sloppy
    Now using the expression  for \scKL~ divergence between the two exponential distributions, we have
    \begin{align}
    \nonumber
        \frac{1}{n} \bbE_{q_n(\theta)} \left[ \scKL\left(  dP^{n}_0\| p(\nX|\theta) \right) \right] & = \int_T^{\infty} \left(\log \frac{\theta_0}{\theta}  + \frac{\theta}{\theta_0} -1\right) \frac{n^n}{\theta_0^n \Gamma(n)} (\theta-T)^{n-1}e^{-n\frac{\theta-T}{\theta_0}} d\theta \\
        &\leq  \frac{n}{n-1} + 1 - 2  = \frac{1}{n-1},
        \label{eq:eqr6c}
    \end{align}
    where second inequality uses the fact that $\log x \leq x-1\leq x-\frac{T}{\theta_0}$ for $\theta_0>T$.
    Combined together~\eqref{eq:eqr6c} and~\eqref{eq:eqr5c} for $n \geq 2$ implies that
    \begin{align}
        \nonumber
        \frac{1}{n} &\left[ \scKL\left(q_n(\theta)\|\pi(\theta) \right) + \bbE_{q_n(\theta)} \left[ \scKL\left(dP^{n}_0)\| p(\nX|\theta) \right) \right]  \right] 
        \\
        &\leq \frac{1}{2} \frac{\log n}{n} +  \frac{1}{n}\left( 2+  \frac{2\beta}{\theta_0}  - \log \sqrt{2\pi }    - \log \left( \frac{\beta^\alpha}{\Gamma(\alpha)} \right)  + \alpha\log \theta_0 \right) \leq C_9 \frac{\log n }{n}. 
    \end{align}
    where $C_9:= \frac{1}{2} + \max \left(0,  {2}+  \frac{2\beta}{\theta_0}  - \log \sqrt{2\pi }    - \log \left( \frac{\beta^\alpha}{\Gamma(\alpha)} \right)  + \alpha\log \theta_0 \right)$ and the  result follows.
\end{proof}

\begin{proof}{Proof of Lemma~\ref{lem:nv6}:}
Since family $\sQ$ contains all gamma distributions, observe that $\{q_n(\cdot) \in \sQ \} \forall n\geq 1$. By definition, $q_n(\theta) = \frac{n^n}{\theta_0^n \Gamma(n)} \theta^{n-1}e^{-n\frac{\theta}{\theta_0}}$. Now consider the first term; using the definition of the $\scKL$ divergence it follows that 
\begin{align}
	\scKL(q_n(\theta)\| \pi(\theta)) = \int q_n(\theta) \log (q_n(\theta))d\theta - \int q_n(\theta) \log (\pi(\theta))d\theta.
	\label{eq:eqr1}
 	\end{align}  
Substituting $q_n(\theta)$ in the first term of the equation above and expanding the logarithm term, we obtain
\begin{align}
	\nonumber
	\int q_n(\theta) \log (q_n(\theta))d\theta &= (n-1) \int \log \theta \frac{n^n}{\theta_0^n \Gamma(n)} \theta^{n-1}e^{-n\frac{\theta}{\theta_0}} d\theta - n + \log \left( \frac{n^n}{\theta_0^n \Gamma(n)} \right)
	\\
	= - \log \theta_0 + (n-1) &\int \log \frac{\theta}{\theta_0} \frac{n^n}{\theta_0^n \Gamma(n)} \theta^{n-1}e^{-n\frac{\theta}{\theta_0}} d\theta -n + \log \left( \frac{n^n}{ \Gamma(n)} \right)
	\label{eq:eqr2}
	\end{align}
Now consider the second term in the equation above. Substitute $\theta = \frac{t \theta_0}{n}$ into the  integral, we have \begin{align}
	\nonumber
	 \int \log \frac{\theta}{\theta_0} \frac{n^n}{\theta_0^n \Gamma(n)} \theta^{n-1}e^{-n\frac{\theta}{\theta_0}} d\theta  &=   \int \log \frac{t}{n}  \frac{1}{ \Gamma(n)} t^{n-1}e^{-t} dt
	 \\
	 &\leq \int \left( \frac{t}{n}-1  \right) \frac{1}{ \Gamma(n)} t^{n-1}e^{-t} dt =0.
	 \label{eq:eqr2a}
	 \end{align} 
Substituting the above result into~\eqref{eq:eqr2}, we get
\begin{align}
	\nonumber
	\int q_n(\theta) \log (q_n(\theta))d\theta 
	&\leq - \log \theta_0 - n + \log \left( \frac{n^n}{ \Gamma(n)} \right)
	\\
	\nonumber
	&\leq - \log \theta_0 - n  + \log \left( \frac{n^n}{ \sqrt{2\pi n} n^{n-1}e^{-n}} \right) 
	\\
	&= - \log \sqrt{2\pi } \theta_0  + \frac{1}{2}\log n, 
	\label{eq:eqr3}
\end{align}
where the  second inequality uses the fact that $\sqrt{2\pi n} n^{n}e^{-n} \leq n \Gamma(n) $. Recall $\pi(\theta)= \frac{\beta^\alpha}{\Gamma(\alpha)} \theta^{-\alpha-1}e^{-\frac{\beta}{\theta}}$. Now consider the second term in~\eqref{eq:eqr1}. Using the definition of inverse-gamma prior and expanding the logarithm function, we have
\begin{align}
	\nonumber
		-\int & q_n(\theta) \log (\pi(\theta))d\theta 
		\\
		\nonumber
		&=  - \log \left( \frac{\beta^\alpha}{\Gamma(\alpha)} \right) + (\alpha+1) \int \log {\theta} \frac{n^n}{\theta_0^n \Gamma(n)} \theta^{n-1}e^{-n\frac{\theta}{\theta_0}} d\theta  + \beta \frac{n}{(n-1)\theta_0}
		\\
		\nonumber
		&=  - \log \left( \frac{\beta^\alpha}{\Gamma(\alpha)} \right) + (\alpha+1) \int \log \frac{\theta}{\theta_0} \frac{n^n}{\theta_0^n \Gamma(n)} \theta^{n-1}e^{-n\frac{\theta}{\theta_0}} d\theta  
		\\ 
		\nonumber
		&\quad + \beta \frac{n}{(n-1)\theta_0} + (\alpha+1)\log \theta_0
		\\
		&\leq - \log \left( \frac{\beta^\alpha}{\Gamma(\alpha)} \right)  + \beta \frac{n}{(n-1)\theta_0} + (\alpha+1)\log \theta_0,
		\label{eq:eqr4}
	\end{align} 
where the last inequality follows from the  observation  in~\eqref{eq:eqr2a}. Substituting~\eqref{eq:eqr4} and~\eqref{eq:eqr3} into~\eqref{eq:eqr1} and dividing either sides by $n$, we obtain
\begin{align}
	\nonumber
	\frac{1}{n}&\scKL(q_n(\theta)\| \pi(\theta)) 
	\\
	\nonumber
	&\leq  \frac{1}{n}\left( - \log \sqrt{2\pi } \theta_0  + \frac{1}{2}\log n - \log \left( \frac{\beta^\alpha}{\Gamma(\alpha)} \right)  + \beta \frac{n}{(n-1)\theta_0} + (\alpha+1)\log \theta_0 \right)
	\\
	=& \frac{1}{2}\frac{\log n}{n} + \beta \frac{1}{(n-1)\theta_0} + \frac{1}{n}\left( - \log \sqrt{2\pi }    - \log \left( \frac{\beta^\alpha}{\Gamma(\alpha)} \right)  + (\alpha)\log \theta_0 \right).
	\label{eq:eqr5}
\end{align}  
Now, consider the second term in the assertion of the lemma. Since, $\xi_i, i\in\{1,2\ldots n\}$ are independent and identically distributed, we obtain 
\begin{align*}
\frac{1}{n} \bbE_{q(\theta)} \left[ \scKL\left(  dP^{n}_0\| p(\nX|\theta) \right) \right]   =  \bbE_{q_n(\theta)} \left[ \scKL\left(dP_0\| p(\xi|\theta) \right) \right]  
\end{align*}
Now using the expression  for \scKL~ divergence between the two exponential distributions, we have
\begin{align}
\nonumber
	\frac{1}{n}& \bbE_{q(\theta)} \left[ \scKL\left(  dP^{n}_0\| p(\nX|\theta) \right) \right]  = \int \left(\log \frac{\theta_0}{\theta}  + \frac{\theta}{\theta_0} -1\right) \frac{n^n}{\theta_0^n \Gamma(n)} \theta^{n-1}e^{-n\frac{\theta}{\theta_0}} d\theta 
	\\
	&\leq  \frac{n}{n-1} + 1 - 2  = \frac{1}{n-1},
	\label{eq:eqr6}
	\end{align}
where second inequality uses the fact that $\log x \leq x-1$.
Combined together~\eqref{eq:eqr6} and~\eqref{eq:eqr5} for $n \geq 2$ implies that
\begin{align}
	\nonumber
	\frac{1}{n} &\left[ \scKL\left(q(\theta)\|\pi(\theta) \right) + \bbE_{q(\theta)} \left[ \scKL\left(dP^{n}_0)\| p(\nX|\theta) \right) \right]  \right] 
	\\
	&\leq \frac{1}{2} \frac{\log n}{n} +  \frac{1}{n}\left( 2+  \frac{2\beta}{\theta_0}  - \log \sqrt{2\pi }    - \log \left( \frac{\beta^\alpha}{\Gamma(\alpha)} \right)  + \alpha\log \theta_0 \right) \leq C_9 \frac{\log n }{n}. 
	\end{align}
where $C_9:= \frac{1}{2} + \max \left(0,  {2}+  \frac{2\beta}{\theta_0}  - \log \sqrt{2\pi }    - \log \left( \frac{\beta^\alpha}{\Gamma(\alpha)} \right)  + \alpha\log \theta_0 \right)$ and the  result follows.
 \end{proof}

\subsection{Multi-product Newsvendor problem}
In the multi-dimensional newsvendor problem,
we fix $n^{-1/2}\sqrt{L^{MNV}_n(\theta,\theta_0)} =  (\sup_{a \in  \mathcal{A}} | R(a,\theta) - R(a,\theta_0) |)$, where 
$ R(a,\theta)=\sum_{i=1}^{d} \Big[(h_i+b_i)a_i \Phi(a_i) - b_i a_i + \theta_{i} (b_i-h_i) $ $ + \sigma_{ii} \left[ h\frac{\phi((a_i-\theta_{i})/\sigma_{ii})}{\Phi((a_i-\theta_{i})/\sigma_{ii})} + b\frac{\phi((a_i-\theta_{i})/\sigma_{ii})}{1-\Phi((a_i-\theta_{i})/\sigma_{ii})} \right] \Big] .$

For brevity, we denote $  d_{L}^{MNV}(\theta,\theta_0)= n^{-1/2}\sqrt{L^{MNV}_n(\theta,\theta_0)}$.  First, we show that 

\begin{lemma}~\label{lem:LipMNV}
    For any compact decision space $ \sA$  and compact model space  $\Theta$, 
    \[
    d_{L}^{MNV}(\theta,\theta_0) \leq K \|\theta-\theta_0 \|,
    \]
    for a  constant $K$ depending on compact sets $\sA$ and $\Theta$  and given $b,h$ and $\Sigma$. 
    \end{lemma}
\begin{proof}
    Observe that
    \begin{align}
        \nonumber
        &\partial_{\theta_{i}} R(a,\theta)
        \\
        \nonumber
        &=   (b_i-h_i)  + (a_i-\theta_{i}) / \sigma_{ii} \phi((a_i-\theta_{i})/\sigma_{ii}) \left[ \frac{h}{\Phi((a_i-\theta_{i})/\sigma_{ii})} + \frac{b}{1-\Phi((a_i-\theta_{i})/\sigma_{ii})} \right] 
        \\
        \nonumber
        &  + \sigma_{ii} \phi\left(\frac{(a_i-\theta_{i})}{\sigma_{ii}}\right) \left[ \frac{h\phi((a_i-\theta_{i})/\sigma_{ii})}{\sigma_{ii} \Phi((a_i-\theta_{i})/\sigma_{ii})^2} - \frac{b\phi((a_i-\theta_{i})/\sigma_{ii})}{\sigma_{ii}(1-\Phi((a_i-\theta_{i})/\sigma_{ii}))^2} \right] 
        \\ 
        \nonumber
        &=   (b_i-h_i)  + (a_i-\theta_{i}) / \sigma_{ii} \phi((a_i-\theta_{i})/\sigma_{ii}) \left[ \frac{h}{\Phi((a_i-\theta_{i})/\sigma_{ii})} + \frac{b}{1-\Phi((a_i-\theta_{i})/\sigma_{ii})} \right] 
        \\
        &   + \phi\left(\frac{(a_i-\theta_{i})}{\sigma_{ii}}\right) \left[ \frac{h\phi((a_i-\theta_{i})/\sigma_{ii})}{ \Phi((a_i-\theta_{i})/\sigma_{ii})^2} - \frac{b\phi((a_i-\theta_{i})/\sigma_{ii})}{(1-\Phi((a_i-\theta_{i})/\sigma_{ii}))^2} \right] .
    \end{align}

    Since, $\sA$ and $\Theta$  are compact sets, therefore $\{(a_i-\theta_i)/\sigma_{ii}\}_{i=1}^d$ lie in  a compact set. Consequently, $\phi((a_i-\theta_{i})/\sigma_{ii})$ and $\Phi((a_i-\theta_{i})/\sigma_{ii})$  also lie in bounded subset of $\bbR$ and thus $\sup_{\sA,\Theta} \|\partial_{\theta_{i}} R(a,\theta)\| \leq K$ for a given $b$, $h$ and $\Sigma$. Since , the norm of the derivative of $R(a,\theta)$ is bounded on $\Theta$ for any $a\in \sA$, therefore, $d_{L}^{MNV}(\theta,\theta_0)$
    is uniformly Lipschitz in $\sA$ with Lipschitz constant $K$, that  is 
    \[
    d_{L}^{MNV}(\theta,\theta_0) \leq K \|\theta-\theta_0 \|.
    \]
    \end{proof}

Next, we show that the $P_{\theta}$ satisfies Assumption~\ref{assump:Asf1}, for distance function $L^{MNV}_n(\theta,\theta_0)$. 
\begin{proof}{Proof of Lemma~\ref{lem:mnv2}:}

 First consider the following test function, constructed using $\nX= \{\xi_1, \xi_2, \ldots,\xi_n \}$.
\vspace{0em}
\[ \phi_{n,\e} := \mathbbm{1}_{ \left\{ \nX:  \left\|\hat \theta_n - \theta_0 \right\| >  \sqrt{C\e^2} \right\} },  \]
where $\hat \theta_n = \frac{\sum_{i=1}^{n}{\xi_i}}{n}$. Note that $\hat \theta_n - \theta_0  \sim \mathcal{N} (\cdot|0,\frac{1}{n}\Sigma)$, where $\frac{1}{n}\Sigma$ is a symmetric positive definite matrix. Therefore it can be decomposed as $\Sigma = Q^T\Lambda Q$, where $Q$ is an orthogonal matrix and $\Lambda$ is a daigonal matrix consisting of respective eigen values and consequently $\hat \theta_n - \theta_0 \sim Q \mathcal{N} (\cdot|0,\frac{1}{n}\Lambda)$. So, we have $\|\hat \theta_n - \theta_0\|^2 \sim \|\mathcal{N} (\cdot|0,\frac{1}{n}\Lambda)\|^2 $. Notice that $\|\mathcal{N} (\cdot|0,\frac{1}{n}\Lambda)\|^2$ is a linear combination of $d$ $\chi^2_{(1)}$ random variable weighted by elements of the diagonal matrix $\frac{1}{n}\Lambda$. 
Using this observation, we first verify that $\phi_{n,\e}$ satisfies condition (\textit{i}) of the Lemma. Observe that 
\begin{align}
\nonumber
	\bbE_{P^{n}_0}[ \phi_n ] = P^{n}_0 \left(  \nX:  \left\|\hat \theta_n - \theta_0 \right\|^2 >  {C\e^2}  \right) & = P^{n}_0 \left(  \nX:  \|\mathcal{N} (\cdot|0,\Lambda)\|^2>  {Cn\e^2}  \right).
\end{align}
Note that $\chi^2_{(1)}$ is $\Gamma$ distributed with shape $1/2$ and scale $2$, which implies $\chi^2_{(1)} - 1$ is a sub-gamma random variable with scale factor $2$ and variance factor $2$. Now observe  that for $\hat \Lambda = \max_{i\in \{1,2,\ldots d\} } 
\Lambda_{ii}$, 
\begin{align}  
    \nonumber
    P^{n}_0 \left(  \nX:  \|\mathcal{N} (\cdot|0,\Lambda)\|^2>  {Cn\e^2}  \right) &\leq   P^{n}_0 \left(  \nX:  \chi^2_{(1)}>  \frac{1}{d\hat \Lambda}{Cn\e^2}  \right)
    \\
    \nonumber
    &\leq P^{n}_0 \left(  \nX:  \chi^2_{(1)}>  \frac{1}{d\hat \Lambda}{Cn\e^2}  \right)
    \\
    \nonumber
    &= P^{n}_0 \left(  \nX:  \chi^2_{(1)}-1>  \frac{1}{d\hat \Lambda}{Cn\e^2} -1 \right)
    \\
    \nonumber
    &\leq e^{-\frac{ \left(\frac{1}{d\hat \Lambda}{Cn\e^2} -1\right)^{2}}{2\left(2+2\left(\frac{1}{d\hat \Lambda}{Cn\e^2} -1\right)\right)}  } 
    \\
    &\leq e^{-1/8 \frac{1}{d\hat \Lambda}{Cn\e^2} +1/8 } \leq e^{-1/8  \left(\frac{C}{d\hat \Lambda} -1\right){n\e^2} },
    \label{eq:eqmnv1}
\end{align}
where in  the third inequality we used the well known tail bound for sub-gamma random variable (Lemma 3.12~\cite{Lugosi}) assuming that $C$ is sufficiently large such that $\left(\frac{1}{d\hat \Lambda}{Cn\e^2} -1\right)>1$ and in the last inequality follows from the assumption that $n\e^2>n\e_n^2 \geq 1$.

Now, we fix the alternate set to be $\{\theta \in \bbR^d: \|\theta - \theta_0 \|\geq 2 \sqrt{C \e^2}\}$. Next, we verify that  $\phi_{n,\e}$ satisfies condition (\textit{ii}) of the lemma.  First, observe that   
\begin{align}
    \bbE_{P^{n}_{\theta}}[ 1-\phi_n ] = P^{n}_{\theta} \left(  \nX:  \left\|\hat \theta_n - \theta_0 \right\|^2 \leq  {C\e^2}  \right) & \leq  P^{n}_\theta \left(  \nX:  \|\hat \theta_n - \theta \| \geq \|\theta - \theta_0 \| - \sqrt{C\e^2}  \right),
\end{align}
where in the last inequality, we used the fact that $\| \theta - \theta_0\| \leq \|\hat \theta_n - \theta \| + \left\|\hat \theta_n - \theta_0 \right\|  $. Now on alternate set $\{\theta \in \bbR^d: \|\theta - \theta_0 \|\geq 2\sqrt{C \e^2}\}$,
 \begin{align}
     \nonumber
     \bbE_{P^{n}_{\theta}}[ 1-\phi_n ] & \leq  P^{n}_\theta \left(  \nX:  \|\hat \theta_n - \theta \| \geq \|\theta - \theta_0 \| - \sqrt{C\e^2}  \right)
     \\
     \nonumber
     &\leq P^{n}_\theta \left(  \nX:  \|\hat \theta_n - \theta \| \geq \|\theta - \theta_0 \| - \sqrt{C\e^2}  \right)
     \\
     &\leq P^{n}_\theta \left(  \nX:  \|\hat \theta_n - \theta \| \geq  \sqrt{C\e^2}  \right).
 \end{align}

Now, it follows from~\eqref{eq:eqmnv1} and $\Theta\subset \bbR^d $ that
 \begin{align*}
     \bbE_{P^{n}_0}[ \phi_n ]  & \leq e^{-1/8  \left(\frac{C}{d\hat \Lambda} -1\right){n\e^2} },  \\
     \underset{\{\theta \in \Theta: \|\theta - \theta_0 \|\geq 2\sqrt{C \e^2}\}}{\sup} \bbE_{P^{n}_{\theta}}[ 1- \phi_n ] \leq \underset{\{\theta \in \bbR^d: \|\theta - \theta_0 \|\geq 2\sqrt{C \e^2}\}}{\sup} \bbE_{P^{n}_{\theta}}[ 1- \phi_n ]  & \leq e^{-1/8  \left(\frac{C}{d\hat \Lambda} -1\right){n\e^2} }.
 \end{align*} 
 
 Using Lemma~\ref{lem:LipMNV}, 
 \( \{\theta \in \Theta: n^{-1/2}\sqrt{L^{MNV}_n(\theta,\theta_0)} \geq 2 K\sqrt{C \e^2}\}   = \{\theta \in \Theta: d_L^{MNV}(\theta,\theta_0) \geq 2 K\sqrt{C \e^2}\} \subseteq \{\theta \in \Theta: \|\theta - \theta_0 \|\geq 2\sqrt{C \e^2}\}, \)
which  implies that 
 \[ \underset{\{\theta \in \Theta: {L^{MNV}_n(\theta,\theta_0)} \geq 4 K^2 {C n\e^2}\}}{\sup} \bbE_{P^{n}_{\theta}}[ 1- \phi_n ] \leq \underset{\{\theta \in \Theta: \|\theta - \theta_0 \|\geq 2\sqrt{C \e^2}\}}{\sup} \bbE_{P^{n}_{\theta}}[ 1- \phi_n ] .\]
 
 Therefore, $P_{\theta}$ for $\theta\in \Theta$, satisifes Assumptions~\ref{assump:Asf1} for $L_n(\theta,\theta_0)= L_n^{MNV}(\theta,\theta_0)$ for $C_0=1$, $C_1= 4 K^2 C$ and $C= 1/8  \left(\frac{C}{d\hat \Lambda} -1\right) $.
\end{proof}

\begin{proof}{Proof  of Lemma~\ref{lem:mnv3}:}
    
    First, we write the R\'enyi divergence between two multivariate Gaussian distribution with known $\Sigma$ as
    \begin{align}
        D_{1+\lambda}( {\mathcal{N}(\cdot|\theta_0)}\|{\mathcal{N}(\cdot|\theta)}  ) = \frac{\lambda+1}{2} {(\theta-\theta_0)^T{\Sigma}(\theta-\theta_0)},
    \end{align}
    and $D_{1+\lambda}( {\mathcal{N}(\cdot|\theta)}\|{\mathcal{N}(\cdot|\theta_0)}  )<\infty$ if and only if $\Sigma^{-1}$ is positive definite ~\citep{gil2013renyi}. 

Since, we assumed that the sequence of models are iid, therefore, 
    \( D_{1+\lambda} \left( P_0^n \| P_{\theta}^n \right) = \frac{1}{\lambda} \log \int \left(\frac{dP_0^n}{dP_{\theta}^n}\right)^\lambda dP_0^n = n \frac{1}{\lambda} \log \int \left(\frac{dP_0}{dP_{\theta}}\right)^\lambda dP_0 = n  \left( \frac{\lambda+1}{2} {(\theta-\theta_0)^T{\Sigma}(\theta-\theta_0)} \right), \)
    when  $\Sigma^{-1}$ is positive definite and $D_{1+\lambda} \left( P_0^n \| P_{\theta}^n \right) = \infty$ otherwise. Now observe  that
    \begin{align}
        \nonumber
        \Pi(D_{1+\lambda} \left( P_0^n \| P_{\theta}^n \right) &\leq n  C_3 \e_n^2 )  = \Pi\left( \left(  {(\theta-\theta_0)^T{\Sigma}(\theta-\theta_0)} \right) \leq \frac{2}{\lambda+1} C_3 \e_n^2\right)
        \\
        \nonumber
        &  = \Pi\left( \left(  {[(\theta-\theta_0)Q]^T{\Lambda}[Q(\theta-\theta_0)]} \right) \leq \frac{2}{\lambda+1} C_3 \e_n^2\right)
        \\
        \nonumber
        &\geq \Pi\left( \left(  {[(\theta-\theta_0)Q]^T[Q(\theta-\theta_0)]} \right) \leq \frac{2}{\hat \Lambda(\lambda+1)} C_3 \e_n^2\right), 
        \\
         &=  \Pi\left( \left(  {[(\theta-\theta_0)]^T[(\theta-\theta_0)]} \right) \leq \frac{2}{\hat \Lambda(\lambda+1)} C_3 \e_n^2\right),    
        \end{align}    
    where $\hat \Lambda = \max_{i\in \{1,2,\ldots d\} } 
    \Lambda_{ii}$ and in  the second equality we used eigen value decomposition of $\Sigma=Q^T\Lambda Q$. Next, observe that,
    \begin{align}
        \nonumber
        \Pi(D_{1+\lambda} \left( P_0^n \| P_{\theta}^n \right) \leq n  C_3 \e_n^2 )  
        & =  \Pi\left( \left(  {[(\theta-\theta_0)]^T[(\theta-\theta_0)]} \right) \leq \frac{2}{\hat \Lambda(\lambda+1)} C_3 \e_n^2\right)
        \\
        \nonumber
        &= \Pi\left(  {\|(\theta-\theta_0)\|}  \leq \sqrt{\frac{2}{\hat \Lambda(\lambda+1)} C_3 \e_n^2}\right)
        \\
        \nonumber
        &\geq \Pi\left(  {\|(\theta-\theta_0)\|_{\infty}}  \leq \sqrt{\frac{2}{\hat \Lambda(\lambda+1)} C_3 \e_n^2}\right)
        \\
        \nonumber
        &= \prod_{i=1}^{d}\Pi_i\left(  {|(\theta_i-\theta_0^i)|}  \leq \sqrt{ \frac{2}{\hat \Lambda(\lambda+1)} C_3 \e_n^2} \right),
    \end{align}   
    where in the last equality we used the  fact that the prior distribution is uncorrelated.
    Now, the  result follows immediately for sufficiently large $n$, if the prior  distribution is uncorrelated and uniformly distributed on  the  compact  set $\Theta_i$, for each $i\in \{1,2,\ldots,d\}$ . In particular observe that for large enough $n$, we have
    \begin{align*}
        \Pi(D_{1+\lambda} \left( P_0^n \| P_{\theta}^n \right) &\leq n  C_3 \e_n^2 )  \geq  \prod_{i=1}^{d} \frac{\theta_0^i + \sqrt{\frac{2}{\hat \Lambda(\lambda+1)} C_3 \e_n^2} - \theta_0^i + \sqrt{\frac{2}{\hat \Lambda(\lambda+1)} C_3 \e_n^2} }{m(\Theta_i)} 
        \\
        &=
        \frac{2^d \left(\frac{2}{\hat \Lambda(\lambda+1)} C_3 \e_n^2 \right)^{d/2} }{\prod_{i=1}^{d} m(\Theta_i)} =  \left(\frac{8}{ (\hat \Lambda(\lambda+1))}\left(\prod_{i=1}^{d} m(\Theta_i)\right)^{-2/d}  C_3 \e_n^2 \right)^{d/2},
    \end{align*}   
where $m(A)$ is the Lebesgue measure (volume) of any set $A\subset \bbR$. 
Now if $\e_n^2=\frac{\log n }{n}$, then for  $\frac{8}{\hat \Lambda(\lambda+1)\left(\prod_{i=1}^{d} m(\Theta_i)\right)^{2/d}}  C_3 >2$, $\frac{8}{\hat \Lambda(\lambda+1)\left(\prod_{i=1}^{d} m(\Theta_i)\right)^{2/d}}  C_3 \e_n^2  \geq e^{-\frac{8}{\hat \Lambda(\lambda+1)\left(\prod_{i=1}^{d} m(\Theta_i)\right)^{2/d}}  C_3 n\e_n^2}$  for all $n\geq 2$, therefore,
\begin{align}
    \nonumber
    \Pi(D_{1+\lambda} \left( P_0^n \| P_{\theta}^n \right) \leq n  C_3 \e_n^2 )  &\geq e^{-\frac{4d}{\hat \Lambda(\lambda+1)\left(\prod_{i=1}^{d} m(\Theta_i)\right)^{2/d}}  C_3 n\e_n^2}.
    \end{align}
    \end{proof}

\begin{proof}{Proof of Lemma~\ref{lem:mnv6}:}
    Since family $\sQ$ contains all uncorrelated Gaussian  distributions restricted  to $\Theta$, observe that $\{q_n(\cdot) \in \sQ \} \forall n\geq 1$. By definition, $q_n^i(\theta) \propto \frac{1}{\sqrt{2\pi \sigma_{i,n}^2}} e^{-\frac{1}{2\sigma_{i,n}^2}(\theta-\mu_{i,n})^2} \Ind{\Theta_i} =  \frac{\mathcal{N}(\theta_i| \mu_{i,n}, \sigma_{i,n})\Ind{\Theta_i}}{\mathcal{N}(\Theta_i| \mu_{i,n}, \sigma_{i,n})} $ and  fix $\sigma_{i,n}=1/\sqrt{n}$ and  $\theta_{i} = \theta_0^i$  for all $i\in \{1,2,\ldots,d\}$. Now consider the first term; using the definition of the $\scKL$ divergence it follows that 
    \begin{align}
        \scKL(q_n(\theta)\| \pi(\theta)) = \int q_n(\theta) \log (q_n(\theta))d\theta - \int q_n(\theta) \log (\pi(\theta))d\theta.
        \label{eq:mnv1c}
    \end{align}  
    Substituting $q_n(\theta)$ in the first term of the equation above and expanding the logarithm term, we obtain
    \begin{align}
        \nonumber
        \int q_n(\theta) \log (q_n(\theta))d\theta &= \sum_{i=1}^{d}  \int q_n^i(\theta_i) \log (q_n^i(\theta_i))d\theta_i
        \\
        \nonumber
        &\leq \sum_{i=1}^{d}  \int \mathcal{N}(\theta_i| \mu_{i,n}, \sigma_{i,n}) \log \mathcal{N}(\theta_i| \mu_{i,n}, \sigma_{i,n}) d\theta_i 
        \\
        &= - \sum_{i=1}^{d}  [\log (\sqrt{2\pi e} ) + \log \sigma_{i,n}],
        \label{eq:mnv2c}
    \end{align}
    where in  the last  equality, we used the well known expression for  the differential entropy of Gaussian distributions.
     Recall $\pi(\theta)= \prod_{i=1}^{d} \frac{1}{m(\Theta_i)} $. Now consider the second term in~\eqref{eq:mnv1c}. It is straightforward to observe that,
    \begin{align}
        -\int q_n(\theta) \log (\pi(\theta))d\theta &=   \sum_{i=1}^{d} \log(m(\Theta_i)) .
        \label{eq:mnv4c}
    \end{align} 
     Substituting~\eqref{eq:mnv4c} and~\eqref{eq:mnv2c} into~\eqref{eq:mnv1c} and dividing either sides by $n$ and substituting $\sigma_{i,n}$, we obtain
    \begin{align}
        \nonumber
        \frac{1}{n}\scKL(q_n(\theta)\| \pi(\theta)) &\leq   - \frac{1}{n}\sum_{i=1}^{d}  [\log (\sqrt{2\pi e} ) -\log(m(\Theta_i))  -\frac{1}{2} \log n]
        \\
        &= \frac{d}{2}\frac{\log n}{n} - \frac{1}{n}\sum_{i=1}^{d}  [\log (\sqrt{2\pi e} ) -\log(m(\Theta_i))  ].
        \label{eq:mnv5c}
    \end{align}  
    Now, consider the second term in the assertion of the lemma. Since $\xi_i, i\in\{1,2\ldots n\}$ are independent and identically distributed, we obtain 
    \begin{align*}
        \frac{1}{n} \bbE_{q_n(\theta)} \left[ \scKL\left(  dP^{n}_0\| p(\nX|\theta) \right) \right]   =  \bbE_{q_n(\theta)} \left[ \scKL\left(dP_0\| p(\xi|\theta) \right) \right]  
    \end{align*}
    Now using the expression  for \scKL~ divergence between the two multivariate Gaussian distributions, we have
    \begin{align}
        \nonumber
        \frac{1}{n} \bbE_{q_n(\theta)} \left[ \scKL\left(  dP^{n}_0\| p(\nX|\theta) \right) \right] & = \frac{1}{2}\bbE_{q_n(\theta)} \left[ (\theta-\theta_0)^{T}\Sigma^{-1}(\theta-\theta_0) \right]
        \\
        \nonumber
        &\leq   \frac{\check \Lambda^{-1}}{2}\bbE_{q_n(\theta)} \left[ (\theta-\theta_0)^{T}(\theta-\theta_0) \right]
        \\
        &\leq \frac{d}{n} \frac{\check \Lambda^{-1}}{2}
        \label{eq:mnv6c}
    \end{align}
    where $\check \Lambda = \min_{i\in \{1,2,\ldots d\} } \Lambda_{ii}$, and $\Sigma^{-1} = Q^T\Lambda^{-1} Q$, where $Q$ is an orthogonal matrix and $\Lambda$ is a daigonal matrix consisting of the respective eigen values of $\Sigma$.
    Combined together~\eqref{eq:mnv6c} and~\eqref{eq:mnv5c} implies that
    \begin{align}
        \nonumber
        \frac{1}{n} &\left[ \scKL\left(q_n(\theta)\|\pi(\theta) \right) + \bbE_{q_n(\theta)} \left[ \scKL\left(dP^{n}_0)\| p(\nX|\theta) \right) \right]  \right] 
        \\
        &\leq \frac{d}{2}\frac{\log n}{n} - \frac{1}{n}\sum_{i=1}^{d}  [\log (\sqrt{2\pi e} ) -\log(m(\Theta_i))  ] + \frac{d}{n} \frac{\check \Lambda^{-1}}{2}  \leq C_9 \frac{\log n }{n}. 
    \end{align}
    where $C_9:= \frac{d}{2} + \max \left(0,   -\sum_{i=1}^{d}  [\log (\sqrt{2\pi e} ) -\log(m(\Theta_i))  ] + \frac{d}{2} {\check \Lambda^{-1}} \right)$ and the  result follows.
\end{proof}

\subsection{Gaussian process classification}

\begin{proof}[Proof of Lemma~\ref{lem:gp2}]
    In view of Theorem~7.1 in~\cite{GGV}, it suffices to show that 
     \[N\left( \e, \Theta_n(\e),d_\text{TV} \right) \leq e^{\bar C n\e^2},\] for some $\bar C>0$. Now, first observe that 
     \begin{align}
         \nonumber
         d_{\text{TV}}(P_{\theta(y)},P_{\theta_0(y)}) &= \frac{1}{2} \bbE_{\nu}\left( |\Psi_1(\theta(y))-\Psi_1(\theta_0(y)) |  + |\Psi_{-1}(\theta(y))-\Psi_{-1}(\theta_0(y)) | \right)
         \\
         \nonumber
         &=  \bbE_{\nu}\left( |\Psi_1(\theta(y))-\Psi_1(\theta_0(y)) |  \right)
         \\
         &\leq \bbE_{\nu}\left( |\theta(y)-\theta_0(y) |  \right) \leq \|\theta(y)-\theta_0(y) \|_{\infty}, 
         \end{align}
    where the second equality uses the definition of $\Psi_{-1}(\cdot)$. Since, total-variation distance above is bounded above by supremum norm, there exists  a constant $0<c'<1/2$, such that
      \begin{align}
          N\left( \e, \Theta_n(\e),d_\text{TV} \right) \leq N\left( c'\e, \Theta_n(\e),\|\cdot\|_{\infty} \right) \leq  e^{\frac{2}{3} c'^2 C_{10}n\e^2},
          \end{align}
      where the last inequality follows from~\eqref{eq:testGP} in Lemma~\ref{lem:GP}. 
    Then if follows from Theorem~7.1 in~\cite{GGV} that for every $\e> \e_n$, there exists a test $\phi_n$ (depending on $\e>0$) such that, for every $j\geq 1$,
    \begin{align*}
        \bbE_{P^{n}_0}[ \phi_n ]  &\leq e^{\frac{2}{3} c'^2 C_{10}n\e^2} e^{ -\frac{1}{2}n \e^2 }\frac{1}{1-\exp \left (-\frac{1}{2} n\e^2 \right)}, \text{and}\\
        \underset{\{ \theta\in \Theta_n(\e)  : d_{TV}(P_{\theta},P_{\theta_0}) > j \e  \}}{\sup} \bbE_{P^{n}_{\theta}}[ 1- \phi_n ] &\leq  \exp\left(- \frac{1}{2} n \e^2 j\right).
    \end{align*} 
    Now for all $n$ such that $n\e^2>n\e_n^2>2\log2$ and $C_{10}=c'^{-2}/4>1$ and $j=1$, we have
    \begin{align}
        \label{eq:eqt1}
        \bbE_{P^{n}_0}[ \phi_n ]  &\leq 2  e^{ -\frac{1}{3}n \e^2 }, \text{and}\\
        \underset{\{ \theta\in \Theta_n(\e)  : d_{TV}(P_{\theta},P_{\theta_0}) >  \e  \}}{\sup} \bbE_{P^{n}_{\theta}}[ 1- \phi_n ] 
        &\leq  e^{- \frac{1}{2} n \e^2 }\leq e^{ -\frac{1}{3}n \e^2 } .
        \label{eq:eqt2}
    \end{align} 
    
    Now observe that 
    \begin{align}
        \nonumber
        \sup_{a \in \mathcal{A}}& | G(a,\theta) - G(a,\theta_0)  | \\
        \nonumber  
        &= \max\left(  c_+ |\bbE_{\nu}[\Psi_{-1}(\theta(y))] -\bbE_{\nu}[\Psi_{-1}(\theta_0(y))] |   , c_- |\bbE_{\nu}[\Psi_{1}(\theta(y))] -\bbE_{\nu}[\Psi_{1}(\theta_0(y))] |\right) 
        \\
        \nonumber
        &= \max\left(  c_+ |\bbE_{\nu}[\Psi_{1}(\theta_0(y))] -\bbE_{\nu}[\Psi_{1}(\theta(y))] |   , c_- |\bbE_{\nu}[\Psi_{1}(\theta(y))] -\bbE_{\nu}[\Psi_{1}(\theta_0(y))] |\right) 
        \\
        \nonumber
        &= \max(c_+,c_-)|\bbE_{\nu}[\Psi_{1}(\theta_0(y))] -\bbE_{\nu}[\Psi_{1}(\theta(y))] |
        \\
        \nonumber
        &\leq \max(c_+,c_-)\bbE_{\nu}[|\Psi_{1}(\theta_0(y)) -\Psi_{1}(\theta(y))| ]
        \\
        &\leq  \max(c_+,c_-)d_{TV}(P_{\theta},P_{\theta_0})
    \end{align} 
    where the second equality uses the fact that $\Psi_{-1}(\cdot)= 1-\Psi_{1}(\cdot)$.
     Consequently,
    \[ \{ \theta\in \Theta_n(\e)  : \sup_{a \in \mathcal{A}} | G(a,\theta) - G(a,\theta_0)  | > { \max(c_+,c_-)\e} \} \subseteq \{ \theta\in \Theta_n(\e)  : d_{TV}(P_{\theta},P_{\theta_0}) >  \e  \} \] Therefore, it follows from~\eqref{eq:eqt1} and~\eqref{eq:eqt2} and the  definition of $L_n(\theta,\theta_0)$ that 
    \begin{align}
        \bbE_{P^{n}_0}[ \phi_n ]  &\leq 2  e^{ -\frac{1}{3}n \e^2 }, \text{and}\\
        \underset{\{ \theta\in \Theta_n(\e)  : L_n(\theta,\theta_0) > ( \max(c_+,c_-))^2 n \e^2  \}}{\sup} \bbE_{P^{n}_{\theta}}[ 1- \phi_n ] 
        &\leq  e^{- \frac{1}{2} n \e^2 }\leq e^{ -\frac{1}{3}n \e^2 } .
    \end{align} 
    Finally, the result follows for $C=1/3$, $C_0=2$ and $C_1= (\max(c_+,c_-))^2$.

\end{proof}

\begin{proof}[Proof of Lemma~\ref{lem:gp3}]
    The R\'enyi divergence 
    \begin{align}
    \nonumber
        &D_{1+\lambda}(P_0^n\| P_{\theta}^n) 
        \\
        \nonumber
        &=  n \frac{1}{\lambda} \ln \int \left( \Psi_{1}(\theta_0(y))^{1+\lambda}\Psi_{1}(\theta(y))^{-\lambda} + \Psi_{-1}(\theta_0(y))^{1+\lambda}\Psi_{-1}(\theta(y))^{-\lambda}  \right) \nu(dy)
        \\
        &= n \frac{1}{\lambda} \ln \int e^{\lambda \frac{1}{\lambda}\ln \left( \Psi_{1}(\theta_0(y))^{1+\lambda}\Psi_{1}(\theta(y))^{-\lambda} + \Psi_{-1}(\theta_0(y))^{1+\lambda}\Psi_{-1}(\theta(y))^{-\lambda}  \right) } \nu(dy).
        \label{eq:gpRenyi}
    \end{align}
    
    Note that  the derivative of the  exponent in the integrand  above with respect  to  $\theta(y)$   is 
    \begin{align} 
        \nonumber
        &\frac{\left(-  \lambda  \Psi_{1}(\theta_0(y))^{1+\lambda}\Psi_{1}(\theta(y))^{-\lambda-1} \psi(\theta(y)) +\lambda \Psi_{-1}(\theta_0(y))^{1+\lambda}\Psi_{-1}(\theta(y))^{-\lambda-1} \psi(\theta(y))  \right)}{ \left(   \Psi_{1}(\theta_0(y))^{1+\lambda}\Psi_{1}(\theta(y))^{-\lambda} + \Psi_{-1}(\theta_0(y))^{1+\lambda}\Psi_{-1}(\theta(y))^{-\lambda}   \right)  } 
        \\
        \nonumber
        &= \lambda \psi(\theta(y)) \frac{\left(-   \Psi_{1}(\theta_0(y))^{1+\lambda}\Psi_{1}(\theta(y))^{-\lambda-1} + \Psi_{-1}(\theta_0(y))^{1+\lambda}\Psi_{-1}(\theta(y))^{-\lambda-1}   \right) }{ \left(   \Psi_{1}(\theta_0(y))^{1+\lambda}\Psi_{1}(\theta(y))^{-\lambda} + \Psi_{-1}(\theta_0(y))^{1+\lambda}\Psi_{-1}(\theta(y))^{-\lambda}   \right)  } 
        \\
        \nonumber
        &= \lambda \frac{\psi(\theta(y))}{\Psi_{1}(\theta(y)) \Psi_{-1}(\theta(y))} \frac{\left(-   \Psi_{1}(\theta_0(y))^{1+\lambda}\Psi_{-1}(\theta(y))^{\lambda+1} + \Psi_{-1}(\theta_0(y))^{1+\lambda}\Psi_{1}(\theta(y))^{\lambda+1}   \right) }{ \left(   \Psi_{1}(\theta_0(y))^{1+\lambda}\Psi_{-1}(\theta(y))^{\lambda} + \Psi_{-1}(\theta_0(y))^{1+\lambda}\Psi_{1}(\theta(y))^{\lambda}   \right)  } 
        \\
        \nonumber
        &= \lambda  \frac{\left(-   \Psi_{1}(\theta_0(y))^{1+\lambda}\Psi_{-1}(\theta(y))^{\lambda+1} + \Psi_{-1}(\theta_0(y))^{1+\lambda}\Psi_{1}(\theta(y))^{\lambda+1}   \right) }{ \left(   \Psi_{1}(\theta_0(y))^{1+\lambda}\Psi_{-1}(\theta(y))^{\lambda} + \Psi_{-1}(\theta_0(y))^{1+\lambda}\Psi_{1}(\theta(y))^{\lambda}   \right)  } 
        \\
        \nonumber
        &= \lambda  \frac{\left(-   e^{-(\lambda+1)\theta(y)} + e^{-(1+\lambda)\theta_0(y)}\right) }{ \left(   e^{-\lambda\theta(y)} + e^{-(\lambda+1)\theta_0(y)}  \right) (1+e^{-\theta(y)})  } 
        \\
        \nonumber
        &= \lambda  \frac{ e^{-(1+\lambda)\theta_0(y)} \left(1 -   e^{-(\lambda+1)(\theta(y)-\theta_0(y))} \right) }{ \left(   e^{-\lambda\theta(y)} + e^{-(\lambda+1)\theta_0(y)}  \right) (1+e^{-\theta(y)})  } 
        \\
        \nonumber   
        &\leq \lambda  \frac{ (\lambda+1)  (\theta(y)-\theta_0(y))  }{ \left(   e^{-\lambda\theta(y)+(\lambda+1){\theta_0(y)}} + 1 \right) (1+e^{-\theta(y)})  } 
        \\
        &\leq \lambda  { (\lambda+1)  |\theta(y)-\theta_0(y)|  },
    \end{align}
    where in the fourth equality we used definition of the logistic function and the penultimate inequality follows from the well known inequality that $1-e^{-x}\leq x$. Consequently, using Taylor's theorem it follows that the exponent in the integrand of the R\'enyi divergence in~\eqref{eq:gpRenyi} is bounded above by $\lambda(\lambda+1)|\theta(y)-\theta_0(y)|^2 $  and thus by $\lambda(\lambda+1)\|\theta(y)-\theta_0(y)\|^2_{\infty} $. Therefore, 
    \begin{align*}
        D_{1+\lambda}&(P_0^n\| P_{\theta}^n) 
        \\
        &=  n \frac{1}{\lambda} \ln \int \left( \Psi_{1}(\theta_0(y))^{1+\lambda}\Psi_{1}(\theta(y))^{-\lambda} + \Psi_{-1}(\theta_0(y))^{1+\lambda}\Psi_{-1}(\theta(y))^{-\lambda}  \right) \nu(dy)
        \\
        &\leq n \frac{1}{\lambda} \ln \int e^{\lambda(\lambda+1)\|\theta(y)-\theta_0(y)\|^2_{\infty}} \nu(dy)
        \\
        &= n (\lambda+1)\|\theta(y)-\theta_0(y)\|^2_{\infty}.
    \end{align*}
    
    Now using the inequality for $C_3=16(\lambda+1)$ above observe that
    \begin{align}
        \nonumber
        \Pi(A_n)&=  \Pi(D_{1+\lambda} \left( P_0^n \| P_{\theta}^n \right) \leq C_3 n \e_n^2)
        \\
        \nonumber
        &\geq \Pi( n(\lambda+1)\|\theta(y)-\theta_0(y)\|^2_{\infty} \leq C_3 n\e_n^2  )
        \\
        &= \Pi( \|\theta(y)-\theta_0(y)\|_{\infty} \leq 4\e_n  )
        \geq  e^{-n\e_n^2}
    \end{align}
    and the  result follows from~\eqref{eq:renyiGP} of~Lemma~\ref{lem:GP}.

\end{proof}

\begin{proof}[Proof of Lemma~\ref{lem:gp4}]
    Let us first analyze  the $\scKL$  divergence between the prior  distribution and variational family.
    Recall that  two Gaussian measures on infinite dimensional spaces  are either equivalent or singular.~\cite[Theorem 6.13]{Stuart2010} specify the condition  required  for the two Gaussian measures to be equivalent. In particular, note that  $\theta_0^J(\cdot)\in \text{Im}(\mathcal{C}^{1/2})$. Now observe that the covariance operator of $Q_n$ has eigenvalues ${\{\zeta^2_j\}_{j=1}^{J}}_{k=1}^{2^{jd}}$, therefore operator $S$ in the definition of $\mathcal{C}_q$ has eigenvalues ${\{1-\zeta^2_j/\mu_j^2\}_{j=1}^{J}}_{k=1}^{2^{jd}}$. For $\tau_j^2=2^{-2ja-jd}$ for any $a>0$,  $\sum_{j=1}^{J} 2^{jd} \left(\frac{n\e_n^2 2^{-2ja-jd} }{1+ n\e_n^2 2^{-2ja-jd} } \right)^2 = \sum_{j=1}^{J} 2^{-jd} \left(\frac{n\e_n^2 2^{-2ja} }{1+ n\e_n^2 2^{-2ja-jd} } \right)^2 <\infty$, therefore $S$ is an HS operator. 
    %
    
    For any integer $J\leq J_{\alpha}$  define $\bar \theta_0^J = \int \theta_0^J(y) \nu(dy)$, where $\theta_0^J(\cdot)= \sum_{j=1}^{J} $ $\sum_{k=1}^{2^{jd}} \theta_{0;j,k}  \vartheta_{j,k}(\cdot) $. Since, $\theta_0^J(\cdot)\in \text{Im}(\mathcal{C}^{1/2})$ and  $S$ is a  symmetric and HS operator, we invoke Theorem 5 in~\cite{quang2019}, to write  
    \begin{align*}
        \scKL(\mathcal{N}(\bar\theta_0^J,\mathcal{C}_q)\|\mathcal{N}(0,\mathcal{C})) &= \frac{1}{2}\| \mathcal{C}^{-1/2}\bar \theta_0^J \|^2 - \frac{1}{2} \log \det(I-S) + \frac{1}{2} tr(-S) ,
        \\
        &= \frac{1}{2}  \sum_{j=1}^{J} \sum_{k=1}^{2^{jd}} \frac{ \theta_{0;k,j}^2}{\mu_j^2} -\frac{1}{2} \log \prod_{j=1}^{J} \prod_{k=1}^{2^{jd}} (1-\kappa_{j}^2) -\frac{1}{2} \sum_{j=1}^{J} \sum_{k=1}^{2^{jd}} \kappa_j^2
        \\
        &= \frac{1}{2}  \sum_{j=1}^{J} \sum_{k=1}^{2^{jd}} \frac{ \theta_{0;k,j}^2}{\mu_j^2} -\frac{1}{2} \log \prod_{j=1}^{J}  (1-\kappa_{j}^2)^{2^{jd}} -\frac{1}{2} \sum_{j=1}^{J} 2^{jd} \kappa_j^2
        \\
        &= \frac{1}{2}  \sum_{j=1}^{J} \sum_{k=1}^{2^{jd}} \frac{ \theta_{0;k,j}^2}{\mu_j^2} -\frac{1}{2}   \sum_{j=1}^{J} {2^{jd}} \log (1-\kappa_{j}^2) -\frac{1}{2} \sum_{j=1}^{J} 2^{jd} \kappa_j^2.
    \end{align*}
    Now for $\mu_j 2^{jd/2}=2^{-ja}$, and using the definition of  Besov norm of $\theta_0$ denoted as $\|\theta_0\|^2_{\beta,\infty,\infty}$, and denoting $1-\kappa_j^2 = \frac{1}{1+n\e_n^2 \tau_j^2}$, we  have 
    \begin{align*}
        \scKL&(\mathcal{N}(\bar\theta_0^J,\mathcal{C}_q)\|\mathcal{N}(0,\mathcal{C})) 
        \\
        &\leq \frac{1}{2}  \sum_{j=1}^{J} {2^{j(2a-2\beta+d)}} \|\theta_0\|^2_{\beta,\infty,\infty}  -\frac{1}{2}   \sum_{j=1}^{J} {2^{jd}} \log (1-\kappa_{j}^2) -\frac{1}{2} \sum_{j=1}^{J} 2^{jd} \kappa_j^2
        \\
        &= \frac{1}{2}  \sum_{j=1}^{J} {2^{j(2a-2\beta+d)}} \|\theta_0\|^2_{\beta,\infty,\infty} - \frac{1}{2}  \sum_{j=1}^{J} 2^{jd} \left( \log (1-\kappa_{j}^2)    +  \kappa_j^2 \right)
        \\
        &= \frac{1}{2}  \sum_{j=1}^{J} {2^{j(2a-2\beta+d)}} \|\theta_0\|^2_{\beta,\infty,\infty} + \frac{1}{2}  \sum_{j=1}^{J} 2^{jd} \left( \log (1+n\e_n^2\tau_{j}^2)    - \frac{n\e_n^2\tau_j^2}{1+n\e_n^2\tau_j^2} \right)
        \\
        &\leq \frac{1}{2}  \sum_{j=1}^{J} {2^{j(2a-2\beta+d)}} \|\theta_0\|^2_{\beta,\infty,\infty} + \frac{1}{2}  \sum_{j=1}^{J} 2^{jd} \left(  n\e_n^2\tau_j^2 \right),
    \end{align*}
    where the last inequality follows from the  fact that, $\log(1+x)-\frac{x}{1+x}\leq \frac{x^2}{1+x}\leq x$ for $x>0$.
    Substituting $\tau_j^2= 2^{-2ja-jd}$, we have
    \begin{align*}
        \frac{1}{n}\scKL(\mathcal{N}(\bar \theta_0^J,\mathcal{C}_q)\|\mathcal{N}(0,\mathcal{C})) &\leq \frac{1}{2n}  \sum_{j=1}^{J} {2^{j(2a-2\beta+d)}} \|\theta_0\|^2_{\beta,\infty,\infty} + \frac{\e_n^2}{2}  \sum_{j=1}^{J} 2^{-2ja} 
        \\
        & \leq \frac{\|\theta_0\|^2_{\beta,\infty,\infty}}{2n}  \sum_{j=1}^{J} {2^{j(2a-2\beta+d)}}  + \frac{2^{-2a}}{2}   \frac{1-2^{-2Ja} }{1-2^{-2a} } \e_n^2.
    \end{align*}
    The summation in the first term above is bounded by $\e_n^2$ as derived in~\cite[Theorem 4.5]{Zanten08}. Therefore,
    \begin{align}
        \frac{1}{n}\scKL(\mathcal{N}(\bar \theta_0^J,\mathcal{C}_q)\|\mathcal{N}(0,\mathcal{C})) 
        & \leq \max\left({\|\theta_0\|^2_{\beta,\infty,\infty}} ,   \frac{2^{-2a}-2^{-2Ja-2a} }{1-2^{-2a} }  \right)   \e_n^2.
        \label{eq:eqKL1}
    \end{align}
    Now consider the second term 
    \begin{align*}
        \frac{1}{n}&\bbE_{Q_n} \scKL(P_0^n\| P_{\theta}^n) 
        \\
        &=  \bbE_{Q_n} \int \left( \Psi_{1}(\theta_0(y)) \log\frac{\Psi_{1}(\theta_0(y))}{\Psi_{1}(\theta(y))} + \Psi_{-1}(\theta_0(y)) \log\frac{\Psi_{-1}(\theta_0(y))}{\Psi_{-1}(\theta(y))}  \right) \nu(dy)
        \\
        &\leq \bbE_{Q_n} \int \langle \theta(y)-\theta_0(y),\theta(y)-\theta_0(y)\rangle  \nu(dy)
        \\
        &= \bbE_{Q_n} \int \| \theta(y)-\theta_0^J(y) - ( \theta_0(y)-\theta_0^J(y))\|_2^2 \nu(dy)
        \\
        &= \bbE_{Q_n} \int \| \theta(y)-\theta_0^J(y)\|^2_2 + \| \theta_0(y)-\theta_0^J(y))\|_2^2 -2\langle \theta(y)-\theta_0^J(y),\theta_0(y)-\theta_0^J(y)\rangle  \nu(dy)
        \\
        &\leq \bbE_{Q_n} \int \| \theta(y)-\theta_0^J(y)\|^2_2  \nu(dy) + \| \theta_0(y)-\theta_0^J(y))\|^2_{\infty}
        \\
        &= \bbE_{Q_n} \int | \sum_{j=1}^{J} \sum_{k=1}^{2^{jd}} \zeta_j Z_{j,k} \vartheta_{j,k}(y)|^2  \nu(dy) + \| \theta_0(y)-\theta_0^J(y))\|^2_{\infty}
        \\
        &\leq \bbE_{Q_n} \sum_{j=1}^{J} \sum_{k=1}^{2^{jd}} \zeta_j^2 Z_{j,k}^2 \int  \vartheta_{j,k}(y)^2  \nu(dy) + \| \theta_0(y)-\theta_0^J(y))\|^2_{\infty} 
        \\
        &= \sum_{j=1}^{J} \sum_{k=1}^{2^{jd}} \zeta_j^2 \bbE_{Q_n}[Z_{j,k}^2] + \| \theta_0(y)-\theta_0^J(y))\|^2_{\infty}
        \\
        &=\sum_{j=1}^{J} \sum_{k=1}^{2^{jd}}   {\mu_j^2(1-\kappa_j^2)} + \| \theta_0(y)-\theta_0^J(y))\|^2_{\infty}
        \\
        &= \sum_{j=1}^{J} {2^{jd}}   \frac{\mu_j^2}{1+n\e_n^2\tau_j^2} + \| \theta_0(y)-\theta_0^J(y))\|^2_{\infty}
        \\
        &\leq   \frac{1}{n\e_n^2} \sum_{j=1}^{J}  \frac{2^{-2ja}}{  \tau_j^2}  + \| \theta_0(y)-\theta_0^J(y))\|^2_{\infty} 
        \\
        &= \frac{1}{n\e_n^2} \sum_{j=1}^{J} 2^{jd} + \| \theta_0(y)-\theta_0^J(y))\|^2_{\infty} 
        \\
        &=\frac{2^d}{n\e_n^2}  \frac{2^{dJ} -1 }{2^{d}-1}  + \| \theta_0(y)-\theta_0^J(y))\|^2_{\infty} 
        \\
        &\leq  \frac{2^d/(2^d-1)}{(\log n )^2}   + C' \e_n^2,
    \end{align*}
    where in the second inequality we used the second assertion of Lemma 3.2~\cite{Zanten08} for logistic function, the fifth  inequality uses  the fact that $\theta(y)-\theta_0^J(y)$ is orthogonal to $\theta_0(y)-\theta_0^J(y)$. For any $a\leq \alpha$ fix $J=J_{\alpha}$ otherwise $J=J_a$, and then it is straight forward to check from the definition of $\e_n$ given in the assertion of the theorem that $(2^{dJ-1}/n\e_n^2) \leq (\log n)^{-2}$. The term  $\| \theta_0(y)-\theta_0^J(y))\|^2_{\infty} $ is also bounded by $C'\e_n^2$ as shown in the proof of Theorem 4.5 in~\cite{Zanten08}. Consequently, the term $\frac{1}{n}\bbE_{Q_n} \scKL(P_0^n\| P_{\theta}^n) $ is bounded above by $\e_n^2$ (upto a constant) for sufficiently large $n$ since $(\log n)^{-2}< \e_n^2$ and the result follows.
\end{proof}

\end{document}